\documentclass[10pt,a4paper,reqno]{amsart}
 \usepackage{amsfonts,amsmath,amscd,amssymb,amsbsy,amsthm,amstext,amsopn,amsxtra}
 \usepackage{color, mathrsfs,verbatim}
 \usepackage{stmaryrd}
 \usepackage{fancyvrb}
 
 \usepackage{extarrows}
 \usepackage{changepage}
 \usepackage[all]{xy}
 \usepackage{bm}   
 \usepackage{todonotes}
 \usepackage[colorlinks]{hyperref}
 \usepackage{tikz}
 \usetikzlibrary{shapes.geometric}
 \usepackage{enumitem}
\usetikzlibrary{arrows,decorations.pathmorphing,decorations.pathreplacing,positioning,shapes.geometric,shapes.misc,decorations.markings,decorations.fractals,calc,patterns}

 \hypersetup{colorlinks,linkcolor=[rgb]{0,0,1},citecolor=[rgb]{1,0,0}}
 
 \tikzset{>=stealth',
        cvertex/.style={circle,draw=black,inner sep=1pt,outer sep=3pt},
        vertex/.style={circle,fill=black,inner sep=1pt,outer sep=3pt},
        star/.style={circle,fill=yellow,inner sep=0.75pt,outer sep=0.75pt},
        tvertex/.style={inner sep=1pt,font=\scriptsize},
        gap/.style={inner sep=0.5pt,fill=white}}

 \numberwithin{equation}{section}
\newtheorem*{pproposition}{Proposition}

\newtheorem{theorem}{Theorem}[section]

\newtheorem{proposition}[theorem]{Proposition}
\newtheorem{corollary}[theorem]{Corollary}
\newtheorem{conjecture}[theorem]{Conjecture}
\newtheorem{assumption}[theorem]{Assumption}

\theoremstyle{definition}
\newtheorem{definition}[theorem]{Definition}
\newtheorem{example}[theorem]{Example}

\theoremstyle{remark}
\newtheorem{remark}[theorem]{Remark}

\newcommand{\op}{\textnormal{op}}
\newcommand{\modcat}{\textnormal{mod}}

\newcommand{\E}{\mathcal{E}}

\DeclareMathOperator{\coh}{coh}

\DeclareMathOperator{\End}{End} 

\DeclareMathOperator{\Ext}{Ext}

\DeclareMathOperator{\GL}{GL} 
\DeclareMathOperator{\Mat}{Mat} 
 
\DeclareMathOperator{\Hom}{Hom}

\DeclareMathOperator{\Spec}{Spec}
\DeclareMathOperator{\Per}{Per}
\DeclareMathOperator{\SL}{SL}

\DeclareMathOperator{\rk}{rk}

\begin{document}

\title[Threefold flops via noncommutative algebras]{The length classification of threefold flops via noncommutative algebras}
\author{Joseph Karmazyn}

\address{School of Mathematics and Statistics,
University of Sheffield,
Hicks Building,
Hounsfield Road,
Sheffield,
S3 7RH.}
\email{j.h.karmazyn@sheffield.ac.uk}
\urladdr{http://www.jhkarmazyn.staff.shef.ac.uk/}
\date{\today}


\begin{abstract}
Smooth threefold flops with irreducible centres are classified by the length invariant, which takes values 1, 2, 3, 4, 5 or 6. This classification by Katz and Morrison identifies 6 possible partial resolutions of Kleinian singularities that can occur as generic hyperplane sections, and the simultaneous resolutions associated to such a partial resolution produce the \emph{universal flop of length $l$}.

In this paper we translate these ideas into noncommutative algebra. We introduce the \emph{universal flopping algebra of length $l$} from which the universal flop of length $l$ can be recovered by a moduli construction, and we present each of these algebras as the path algebra of a quiver with relations. This explicit realisation can then be used to construct examples of NCCRs associated threefold flops of any length as quiver with relations defined by superpotentials, to recover the matrix factorisation description of the universal flop conjectured by Curto and Morrison, and to realise examples of contraction algebras.
\end{abstract}

\maketitle

\tableofcontents

\section{Introduction}
A classification of simple threefold flops was produced by Katz and Morrison \cite{KatzMorrison} using the length invariant introduced by Kollar \cite[Page 95-96]{Length}. Under this classification the length invariant takes one of 6 possible values and is uniquely determined by the generic hyperplane section, which is a partial resolution of a Kleinian singularity. This length classification result is accompanied by a construction producing a \emph{universal flop of length $l$} such that, locally, any length $l$ flop can theoretically be constructed via a classifying map to the universal flop of length $l$.

Whilst this classification result provides 6 classes of simple threefold flops and the corresponding construction produces associated universal objects, it seems to be difficult to directly use this geometric description of universal flops to explicitly construct examples of simple threefold flops of all lengths.

As a solution, in this paper we investigate how the geometric classification and construction can be realised instead in the language of noncommutative algebra. The advantage of this translation is that it encodes the universal flop of length $l$ in an explicit and easy to present noncommutative algebra that we introduce: the \emph{universal flopping algebra of length $l$}, from which the geometric universal flop of length $l$ can be recovered by a moduli space construction. Having realised such an explicit presentation, an immediate application is the construction of specific examples of flops of all lengths, see Section \ref{Sec: Examples}.

This program of understanding universal flops via noncommutative algebra is also inspired by a conjectural program of Curto and Morrison that aimed to realise universal flops via matrix factorisations \cite{CurtoMorrison}. Using the universal flopping algebras, we explicitly recover Curto and Morrison's construction of the universal flop of length 2 as a matrix factorisation in Section \ref{Example: Length 2} and later in the introduction we outline how the matrix factorisations conjectured by Curto and Morrison can be found for all lengths.

In addition, there has recently been much work directed towards understanding threefold flops via homological techniques and noncommutative algebra, often building on earlier work of Bridgeland \cite{BridgelandFlops} and Van den Bergh \cite{3Dflops}.  See for example the work of Wemyss \cite{WemyssHMMP}, Brown and Wemyss \cite{GKVWemyssBrown}, Donovan and Wemyss \cite{WemyssDonovan}, Nolla de Celis and Sekiya \cite{CelisSekiya}, and Bodzenta and Bondal \cite{FlopsBodzentaBondal}.

One outcome of the work of Donovan and Wemyss \cite{WemyssDonovan} is the introduction of a new invariant of threefold flops: the contraction algebra. This is is an finite dimensional algebra, and it was shown by Hua that the $A_{\infty}$ version is a complete invariant of simple threefold flops up to analytic equivalence \cite{ZhengHua}.  An ongoing project of Brown and Wemyss aims to realise a full analytic classification of threefold flops via contraction algebras.

It is straightforward to recover the contraction algebras from examples produced from the universal flopping algebras, and so the results in this paper the program of realising contraction algebras by providing a tool to produce examples.

We now briefly outline the structure and main results of the paper, and we recap several related results.

\subsection*{Simple threefold flops} This paper considers the classification of simple threefold flops. A simple threefold flop is a flop of smooth threefolds with centre an irreducible curve. This curve is necessarily smooth and rational, so is isomorphic to $\mathbb{P}^1$. 

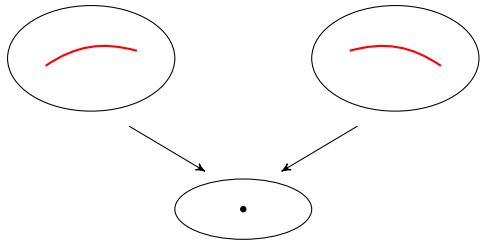
\[
\begin{tikzpicture}   
\draw[red, thick] (-2.6,-0.1) to [bend left=25] node[pos=0.48, right] {} (-1.4,0.1);
\draw[black] (-2,0) ellipse (1.1 and 0.7);
\draw[->] (-1.5,-0.9) -- (-0.5,-1.5);
\draw[red, thick] (1.4,0.1) to [bend left=25] node[pos=0.48, right] {} (2.6,-0.1);
\draw[black] (2,0) ellipse (1.1 and 0.7);
\draw[->] (1.5,-0.9) -- (0.5,-1.5);
\draw[black] (0,-2) ellipse (0.9 and 0.4);
\filldraw [black] (0,-2) circle (1pt);
\end{tikzpicture}
\]

Such a flop is determined by a small resolution of a threefold Gorenstein singularity that has an irreducible exceptional set; i.e. one side of the flop.  We refer to such a small resolution as a simple threefold flopping contraction. Locally, classifying simple flops is equivalent to classifying simple threefold flopping contractions as either side of a flop can be recovered from the other, see \cite{KollarFlops}.

In order to classify simple threefold flopping contractions, it is natural to associate invariants to the geometry. Various different invariants have been introduced to help classify simple threefold flopping contractions, and an obvious invariant is the normal bundle of the exceptional curve. The only normal bundles that can occur are type $(-1,-1)$, $(-2,0)$, or $(-3,1)$, so this invariant splits flops into three separate classes.

There are well known examples of threefold flops in all these classes: the Atiyah flop has normal bundle $(-1,-1)$, Reid's pagoda flops have normal bundle $(-2,0)$ \cite[Section 5]{ReidGeneral}, and a variety of examples have been constructed with normal bundle $(-3,1)$, see \cite{LauferFlop, PinkhamConstruction, ReidGeneral, MorrisonFlops, Curto, Ando}. Indeed, the Atiyah flop and Reid's pagoda flops locally classify all $(-1,-1)$ and $(-2,0)$ flops, and all other flops are in class $(-3,1)$. As such, while the normal bundle is a natural invariant it is rather coarse.

Another invariant is the length invariant introduced by Kollar \cite[Page 95-96]{Length}. For a small resolution $\pi: Y \rightarrow X= \Spec \, R$ contracting the irreducible curve $C$ to the point $p$, the length invariant is defined to be the generic rank of the cokernel of $\pi^* \mathcal{O}_p \rightarrow \mathcal{O}_Y$ on $C$. It was shown by Katz and Morrison \cite{KatzMorrison} (with an alternative, shorter proof later given by Kawamata \cite{KawamataHyperplane}) that any simple flopping contraction has length 1, 2, 3, 4, 5 or 6. 

\begin{theorem}[{\cite{KatzMorrison}}] \label{Theorem: Length Classification}(Length Classification)
Any simple threefold flopping contraction has length $1,2,3,4,5$ or $6$. The length is uniquely determined by the generic hyperplane section, and this hyperplane section is a partial resolution of a Kleinian surface singularity $\pi_{\Gamma_C}:Y_{\Gamma_C} \rightarrow X_{\Gamma}$ of type $\Gamma_{C}$ occurring in the following list:
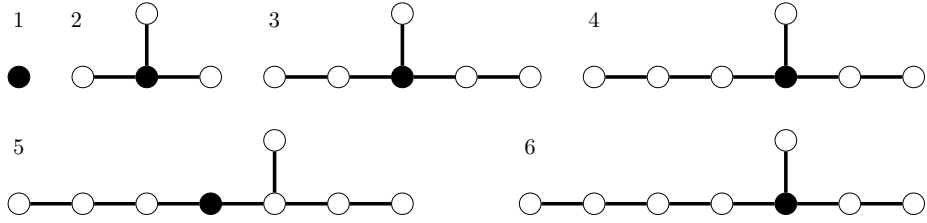
\[
\resizebox{350pt}{!}{ 
\begin{tikzpicture}[node distance=1cm, main node/.style={circle,fill=black!100,draw,font=\sffamily\Large\bfseries}]
\node (1Lab) at (0,0.9) {$1$};
  \node[circle,fill=black!100,draw] (11) at (0,0) {};
  
  \node (2Lab) at (0.9,0.9) {$2$};
    \node[main node] (21) at (2,0) {};
  \node[circle,draw] (22) at (1,0) {};
  \node[circle,draw] (23) at (3,0) {};
  \node[circle,draw] (24)  at (2,1) {};
 \draw [ultra thick] (21) to node {} (22);
 \draw [ultra thick] (21) to node {} (23);
 \draw [ultra thick] (21) to node {} (24);
 
\node (3Lab) at (4,0.9) {$3$};
  \node[main node] (31) at (6,0) {};
    \node[circle,draw] (32') at (4,0) {};
  \node[circle,draw] (32) at (5,0) {};
  \node[circle,draw] (33) at (7,0) {};
    \node[circle,draw] (33') at (8,0) {};
  \node[circle,draw] (34)  at (6,1) {};
 \draw [ultra thick] (31) to node {} (32);
 \draw [ultra thick] (31) to node {} (33);
 \draw [ultra thick] (31) to node {} (34);
  \draw [ultra thick] (32) to node {} (32');
   \draw [ultra thick] (33) to node {} (33');
 
 \node (4Lab) at (9,0.9) {$4$};
     \node[main node] (41) at (12,0) {};
      \node[circle,draw] (42'') at (9,0) {};
    \node[circle,draw] (42') at (10,0) {};
  \node[circle,draw] (42) at (11,0) {};
  \node[circle,draw] (43) at (13,0) {};
    \node[circle,draw] (43') at (14,0) {};
  \node[circle,draw] (44)  at (12,1) {};
 \draw [ultra thick] (41) to node {} (42);
 \draw [ultra thick] (41) to node {} (43);
 \draw [ultra thick] (41) to node {} (44);
  \draw [ultra thick] (42) to node {} (42');
    \draw [ultra thick] (42') to node {} (42'');
   \draw [ultra thick] (43) to node {} (43');
   
    \node (5Lab) at (0,-1.1) {$5$};
   \node[circle,draw] (51) at (4,-2) {};
        \node[circle,draw] (52''') at (0,-2) {};
      \node[circle,draw] (52'') at (1,-2) {};
    \node[circle,draw] (52') at (2,-2) {};
  \node[main node] (52) at (3,-2) {};
  \node[circle,draw] (53) at (5,-2) {};
    \node[circle,draw] (53') at (6,-2) {};
  \node[circle,draw] (54)  at (4,-1) {};
 \draw [ultra thick] (51) to node {} (52);
 \draw [ultra thick] (51) to node {} (53);
 \draw [ultra thick] (51) to node {} (54);
  \draw [ultra thick] (52) to node {} (52');
    \draw [ultra thick] (52') to node {} (52'');
        \draw [ultra thick] (52'') to node {} (52''');
   \draw [ultra thick] (53) to node {} (53');
   
          \node (6Lab) at (8,-1.1) {$6$};
               \node[main node] (61) at (12,-2) {};
        \node[circle,draw] (62''') at (8,-2) {};
      \node[circle,draw] (62'') at (9,-2) {};
    \node[circle,draw] (62') at (10,-2) {};
  \node[circle,draw] (62) at (11,-2) {};
  \node[circle,draw] (63) at (13,-2) {};
    \node[circle,draw] (63') at (14,-2) {};
  \node[circle,draw] (64)  at (12,-1) {};
   \draw [ultra thick] (61) to node {} (62);
 \draw [ultra thick] (61) to node {} (63);
 \draw [ultra thick] (61) to node {} (64);
  \draw [ultra thick] (62) to node {} (62');
    \draw [ultra thick] (62') to node {} (62'');
        \draw [ultra thick] (62'') to node {} (62''');
   \draw [ultra thick] (63) to node {} (63');

\end{tikzpicture}
}
\]
(notation for hyperplane sections in terms of coloured Dynkin diagrams $\Gamma_C$ is recalled in Section \ref{Subsection: Kleinian Singularities and deformations}).
\end{theorem}

The length invariant is a finer invariant than the normal bundle, providing 6 classes of flopping contractions: the length invariant is 1 if and only if the normal bundle is $(-1,-1)$ or $(-2,0)$ and the $(-3,1)$ flops are split into lengths 2, 3, 4, 5, and 6.

Moreover, the examples with normal bundle $(-3,1)$ appearing in \cite{LauferFlop, PinkhamConstruction, ReidGeneral, MorrisonFlops} have length 2, the examples appearing in \cite{Curto, Ando} have length 3, and it seems to be difficult to find explicit examples in the literature of flops of length $>3$. 

As well as being a finer invariant, a further advantage of the length classification is that it has an accompanying construction result. Associated to a partial resolution of a Kleinian singularity that occurs in the length classification are two simultaneous resolutions as recalled  Section \ref{Subsection: Kleinian Singularities and deformations}:

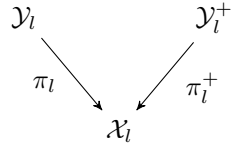
\begin{align*}
\begin{tikzpicture} [bend angle=0]
\node (Y) at (-1.25,1.5)  {$\mathcal{Y}_l$};
\node (Y+) at (1.25,1.5)  {$\mathcal{Y}_l^{+}$};
\node (X) at (0,0)  {$\mathcal{X}_l$};
\draw [->] (Y) to node[below left, pos=0.4]  {$\pi_l$} (X);
\draw [->] (Y+) to node[ below right, pos=0.3]  {$\pi_l^{+}$} (X);
\end{tikzpicture}
\end{align*}

We call this pair of simultaneous resolutions the \emph{universal flop of length $l$}. These are flat families of partial resolutions of Kleinian singularities over an affine base $\mathfrak{h}_l$, and the name is justified as it is a consequence of the classification results of \cite{KatzMorrison}, following the philosophy of Pinkham \cite{PinkhamConstruction}, that every simple threefold flop of length $l$
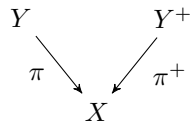
\begin{align*}
\begin{tikzpicture} [bend angle=0]
\node (Y) at (-1,1.25)  {$Y$};
\node (Y+) at (1,1.25)  {$Y^{+}$};
\node (X) at (0,0)  {$X$};
\draw [->] (Y) to node[below left, pos=0.4]  {$\pi$} (X);
\draw [->] (Y+) to node[ below right, pos=0.3]  {$\pi^{+}$} (X);
\end{tikzpicture}
\end{align*} has a flat morphism $\theta: Y \rightarrow \Delta$ to the unit disk $\Delta$ after shrinking $Y$ and a \emph{classifying map} $\nu: \Delta \rightarrow \mathfrak{h}_l$ 
such that the pullback via $\nu$ of $\mathcal{X}_l$ coincides with X, the pullback via $\nu$ of $\pi_l:\mathcal{Y}_l\rightarrow \mathcal{X}_l$ coincides with $\pi:Y\rightarrow X$, and the pullback via $\nu$ of $\pi^+_l:\mathcal{Y}_l^+\rightarrow \mathcal{X}_l$ coincides with $\pi^+:Y^+ \rightarrow X$.

This construction shows that the pullback of the morphism $\pi_l$ by a generic map from the formal disc to the base of the universal flop of length $l$ constructs a flopping contraction of length $l$ such that the other side of the flop can be recovered from the morphism $\pi_l^+$. By considering such classifying morphisms Katz and Morrison were able to show flopping contractions of all possible lengths exist.
\begin{theorem}[\cite{KatzMorrison}] \label{Theorem:Universal Flop} (Length Construction) Simple flopping contractions of lengths 1, 2, 3, 4, 5 and 6 all exist. Locally, any simple threefold flop of length $l$ can be be constructed as the pullback by a classifying morphism to the universal flop of length $l$. 
\end{theorem}

\subsection*{Universal flops and noncommutative algebras}  Whilst this classification and construction provides six classes of flops with associated universal objects, it appears to be difficult to construct explicit examples in each of these classes using just the given algebraic-geometric construction. In this paper we introduce the \emph{universal flopping algebra of length $l$}: $\mathcal{A}_l$. This is a derived equivalent noncommutative algebra with a chosen idempotent $e_0 \in \mathcal{A}_l$ from which the universal flop can be recovered by a moduli construction.

The fact that such an algebra exists largely follows from Van den Bergh's more general construction in \cite{3Dflops}, and we recall the construction of such an algebra, some of its properties, and the associated moduli constructions $\mathcal{M}_{A^{(\op)}_l}$ in Section \ref{Section: Noncommutative algebra construction}.

\begin{pproposition}[Proposition \ref{Prop:Summary}]
Let $e_0 \in \mathcal{A}_l$ denote the universal flopping algebra of length $l$. There is a derived equivalence 
\[
D^b( \coh \, \mathcal{Y}_{l}) \cong D^b(\mathcal{A}_{l}\text{-mod})
\]
Moreover, the universal flop of length $l$ can be recovered via a moduli space construction:
\begin{align*}
\begin{aligned}
\begin{tikzpicture} [bend angle=0]
\node (Y) at (-1.25,1.5)  {$\mathcal{Y}_l$};
\node (Y+) at (1.25,1.5)  {$\mathcal{Y}_l^{+}$};
\node (X) at (0,0)  {$\mathcal{X}_l$};
\draw [->] (Y) to node[below left, pos=0.4]  {$\pi_l$} (X);
\draw [->] (Y+) to node[ below right, pos=0.3]  {$\pi_l^{+}$} (X);
\end{tikzpicture}
\end{aligned}
\cong
\begin{aligned}
\begin{tikzpicture} [bend angle=0]
\node (Y) at (-1.25,1.5)  {$\mathcal{M}_{\mathcal{A}_l}$};
\node (Y+) at (1.25,1.5)  {$\mathcal{M}_{\mathcal{A}^{\op}_l}$};
\node (X) at (0,0)  {$\Spec \, e_0 \mathcal{A}_l e_0$};
\draw [->] (Y) to node[below left, pos=0.4]  {$\pi_l$} (X);
\draw [->] (Y+) to node[ below right, pos=0.3]  {$\pi_l^{+}$} (X);
\end{tikzpicture}
\end{aligned}
\end{align*}
where $\mathcal{R}_l \cong e_0 \mathcal{A}_l e_0 \cong e_0 \mathcal{A}_l^{\op}e_0$, $\Spec \, \mathcal{R}_l \cong \mathcal{X}_l$, and $\mathcal{M}_{\mathcal{A}_l} \cong \mathcal{Y}_l$ and $\mathcal{M}_{\mathcal{A}_l^{\op}} \cong \mathcal{Y}_l^{+}$ as $\mathcal{R}_l$-schemes.
\end{pproposition}

In order to use the universal flopping algebras to calculate examples of threefold flops it is necessary to obtain an explicit description of the algebras. Having introduced the universal flopping algebras, the next main result of this paper is to explicitly present these algebras as the path algebras of quivers with relations over polynomial rings $\mathbb{H}_l$.

\begin{theorem}[Proof in Section \ref{Section:Proof of Explicit}] \label{Theorem: Explicit presentations}
The universal flopping algebra of length $l$ can be presented as a path algebra over a polynomial ring $\mathbb{H}_l$ for a quiver and relations: 
\[
\mathcal{A}_l:= \frac{\mathbb{H}_lQ_l}{I_l}.
\]
For each each $l$ we list the polynomial ring $\mathbb{H}_l$, the quiver $Q_l$, and relations generating the two sided ideal $I_l$.
\begin{enumerate}
\item Length 1, $\mathbb{H}_1 \cong \mathbb{C}[t]$.
\begin{align*}
\begin{aligned}
\begin{tikzpicture} [bend angle=15, looseness=1]
\node (C0) at (0,0)  {$0$};
\node (C1) at (2.5,0)  {$1$};
\draw [->,bend left=35] (C0) to node[gap]  {\scriptsize{$a_0$}} (C1);
\draw [->,bend left] (C0) to node[gap]  {\scriptsize{$a_1^*$}} (C1);
\draw [->,bend left] (C1) to node[gap]  {\scriptsize{$a_0^*$}} (C0);
\draw [->,bend left=35] (C1) to node[gap]  {\scriptsize{$a_1$}} (C0);
\end{tikzpicture}
\end{aligned}
& \qquad
\begin{aligned}
&a_0a_0^*-a_1^*a_1=t e_0, \\
&a_1a_1^*-a_0^*a_0=-te_1.
\end{aligned}
\end{align*}
\item Length 2, $\mathbb{H}_2:= \mathbb{C}[t,T_0^\beta,T_0^\gamma,T_0^\delta]$.
\begin{align*}
\begin{aligned}
\begin{tikzpicture} [bend angle=15, looseness=1]
\node (C0) at (0,0)  {$0$};
\node (C1) at (2,0)  {$4$};
\draw [->,bend left] (C0) to node[above]  {\scriptsize{$a$}} (C1);
\draw [->,bend left] (C1) to node[below]  {\scriptsize{$a^*$}} (C0);
\draw [->, looseness=24, in=52, out=128,loop] (C1) to node[above] {$\scriptstyle{\beta}$} (C1);
\draw [->, looseness=24, in=-38, out=38,loop] (C1) to node[right] {$\scriptstyle{\delta}$} (C1);
\draw [->, looseness=24, in=-128, out=-52,loop] (C1) to node[below] {$\scriptstyle{\gamma}$} (C1);
\end{tikzpicture}
\end{aligned}
& \quad \quad
\begin{aligned}
&aa^*=te_0, \quad \beta^2=T_0^\beta e_4, \\ 
&\gamma^2=T_0^\gamma e_4, \quad \delta^2=T_0^\delta e_4, \\ 
&a^*a+\beta+\gamma + \delta = \frac{t}{2} \, e_4. \\ 
\end{aligned}
\end{align*}

\item Length 3, $\mathbb{H}_3 \cong \mathbb{C}[t,T^{\beta}_0,T^{\beta}_1,T^{\gamma}_0,T^{\gamma}_1,T_0^\delta ]$.
\begin{align*}
\begin{aligned}
\begin{tikzpicture} [bend angle=15, looseness=1]
\node (C0) at (0,0)  {$0$};
\node (C1) at (2,0)  {$6$};
\draw [->,bend left] (C0) to node[above]  {\scriptsize{$a$}} (C1);
\draw [->,bend left] (C1) to node[below]  {\scriptsize{$a^*$}} (C0);
\draw [->, looseness=24, in=52, out=128,loop] (C1) to node[above] {$\scriptstyle{\beta}$} (C1);
\draw [->, looseness=24, in=-38, out=38,loop] (C1) to node[right] {$\scriptstyle{\delta}$} (C1);
\draw [->, looseness=24, in=-128, out=-52,loop] (C1) to node[below] {$\scriptstyle{\gamma}$} (C1);
\end{tikzpicture}
\end{aligned}
\qquad
\begin{aligned}
& \delta a^*=a^* t, & & a \delta = t a, \\ 
&aa^*=(t^2 -T_0^{\delta})e_0, & &  a^*a=\delta^2 - T_0^{\delta}e_4,  \\ 
&\beta^3= T^{\beta}_1 \beta + T^{\beta}_0e_4, & &\gamma^3=T^{\gamma}_1 \gamma + T^{\gamma}_0e_4, \\ 
&\beta+\gamma + \delta = \frac{t}{3} \, e_4. \\ 
\end{aligned}
\end{align*}

\item Length 4, $\mathbb{H}_4:=\mathbb{C}[t,T^{\beta}_0,T^{\gamma}_0,T^{\gamma}_1,T^{\gamma}_2,T_0^\delta,T_1^\delta]$.
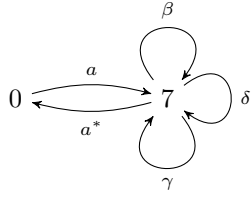
\begin{align*}
\begin{aligned}
\begin{tikzpicture} [bend angle=15, looseness=1]
\node (C0) at (0,0)  {$0$};
\node (C1) at (2,0)  {$7$};
\draw [->,bend left] (C0) to node[above]  {\scriptsize{$a$}} (C1);
\draw [->,bend left] (C1) to node[below]  {\scriptsize{$a^*$}} (C0);
\draw [->, looseness=24, in=52, out=128,loop] (C1) to node[above] {$\scriptstyle{\beta}$} (C1);
\draw [->, looseness=24, in=-38, out=38,loop] (C1) to node[right] {$\scriptstyle{\delta}$} (C1);
\draw [->, looseness=24, in=-128, out=-52,loop] (C1) to node[below] {$\scriptstyle{\gamma}$} (C1);
\end{tikzpicture}
\end{aligned}
\qquad
\begin{aligned}
& \delta a^*=a^* t, & & a \delta = ta, \\
&aa^*=(t^3-T^{\delta}_1 t - T^{\delta}_0) e_0, & &  a^*a=\delta^3- T^{\delta}_1 \delta  - T^{\delta}_0 e_7,  \\ 
&\beta^2=T^\beta_0 e_7, & &\gamma^4=T^{\gamma}_2 \gamma^2 + T^{\gamma}_1 \gamma + T^{\gamma}_0 e_7, \\ 
&\beta+\gamma +\delta = \frac{t}{4} \, e_7. \\ 
\end{aligned}
\end{align*}

\item Length 5, $\mathbb{H}_5:=\mathbb{C}[t,T_0^\delta,T_1^\delta,T_2^\delta,T_0,T_1,T_2,T_3]$.
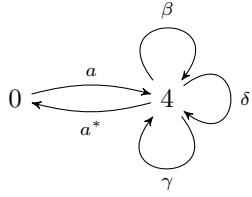
\begin{align*}
\begin{aligned}
\begin{tikzpicture} [bend angle=15, looseness=1]
\node (C0) at (0,0)  {$0$};
\node (C1) at (2,0)  {$4$};
\draw [->,bend left] (C0) to node[above]  {\scriptsize{$a$}} (C1);
\draw [->,bend left] (C1) to node[below]  {\scriptsize{$a^*$}} (C0);
\draw [->, looseness=24, in=52, out=128,loop] (C1) to node[above] {$\scriptstyle{\beta}$} (C1);
\draw [->, looseness=24, in=-38, out=38,loop] (C1) to node[right] {$\scriptstyle{\delta}$} (C1);
\draw [->, looseness=24, in=-128, out=-52,loop] (C1) to node[below] {$\scriptstyle{\gamma}$} (C1);
\end{tikzpicture}
\end{aligned}
& \quad \quad
\begin{aligned}
&a\delta= ta, \qquad \delta a^* = a^*t,\\
&aa^*=(t^4 - T^{\delta}_2t^2 - T_1^{\delta} t -T_0^{\delta})e_0, \\
&a^*a=\delta^4 - T^{\delta}_2\delta^2 - T_1^{\delta} \delta -T_0^{\delta}e_4, \\
&\gamma \beta \gamma +\gamma^2 \beta +\gamma \beta^3=-T_3\gamma \beta -T_2 \gamma -T_0e_4 \\
&
(\gamma+ \beta^2)^2+\beta \gamma \beta = -T_3(\gamma+\beta^2)  -T_2\beta -T_1e_4,\\
&\delta - \beta =\frac{t}{5}e_4.
\end{aligned}
\end{align*}

\item Length 6, $\mathbb{H}_6:=\mathbb{C}[t,T^{\beta}_0,T^{\gamma}_0,T^{\gamma}_1,T_0^\delta,T_1^\delta,T_2^\delta,T_3^\delta]$.

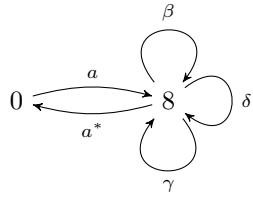
\begin{align*}
\begin{aligned}
\begin{tikzpicture} [bend angle=15, looseness=1]
\node (C0) at (0,0)  {$0$};
\node (C1) at (2,0)  {$8$};
\draw [->,bend left] (C0) to node[above]  {\scriptsize{$a$}} (C1);
\draw [->,bend left] (C1) to node[below]  {\scriptsize{$a^*$}} (C0);
\draw [->, looseness=24, in=52, out=128,loop] (C1) to node[above] {$\scriptstyle{\beta}$} (C1);
\draw [->, looseness=24, in=-38, out=38,loop] (C1) to node[right] {$\scriptstyle{\delta}$} (C1);
\draw [->, looseness=24, in=-128, out=-52,loop] (C1) to node[below] {$\scriptstyle{\gamma}$} (C1);
\end{tikzpicture}
\end{aligned}
& \quad \quad
\begin{aligned}
& \delta a^*=a^* t, \qquad a \delta = t a, \\ 
&aa^*=(t^5-T^\delta_3t^3-T^\delta_2t^2-t T^\delta_1 -T^\delta_0)e_0,  \\ 
& a^*a=\delta^5 -T^\delta_3\delta^3-T^\delta_2\delta^2- T^\delta_1 \delta-T^\delta_0 e_8,\\ 
&\beta^2=T^{\beta}_0, \qquad\gamma^3=T^{\gamma}_1 \gamma + T^{\gamma}_0 e_8, \\ 
&\beta+\gamma + \delta = \frac{t}{6} \, e_8. \\ 
\end{aligned}
\end{align*}
\end{enumerate}
\end{theorem}

This quiver with relations description is both explicit and compact and so well suited for calculations. One immediate application of such an explicit presentation is to construction examples of noncommutative crepant resolutions corresponding to simple threefold flops of all lengths by considering various classifying maps. Such examples are constructed as paths algebras with quivers with relations defined by a superpotential in Section \ref{Sec: Examples}. Alongside constructing these examples, we outline several further applications below.

\subsection*{Universal flops and matrix factorisations.}
The universal flopping algebras can also be used to resolve conjectures of Curto and Morrison on realising universal flops as matrix factorisations. They conjecture that particular matrix factorisations exist from which the universal flop can be recovered by a process they call "Grassman blowup", see \cite{CurtoMorrison}. 

\begin{conjecture}[{\cite[Conjecture 3]{CurtoMorrison}}] \label{conj:CurtoMorrison2}
For a partial resolution of a rational double point corresponding
to a single vertex in the Dynkin diagram with coefficient ℓ in the maximal root, the versal deformation $\mathcal{X}$ over $\text{PRes}$ has two matrix factorizations of size $2l \times 2l$, such that the two simultaneous partial resolutions can be obtained as the Grassmann blowups of the corresponding Cohen-Macaulay modules. Moreover, the matrices take the special form $x I_{2l} - \Xi$ and $x I_{2l} +\Xi$ in appropriate coordinates.
\end{conjecture}

In particular, this conjecture proposes that the two simultaneous resolutions occurring in the universal flop of length $l$ can be recovered from a corresponding $2l \times 2l$ matrix factorisation of the hypersurface defining $\mathcal{X}_l:= \Spec \mathcal{R}_l$. Using the universal flopping algebras, we can verify this claim.

Firstly, the universal flopping algebras introduced in this paper also encode the data of the commutative algebra $\mathcal{R}_l$ and two rank $l$ MCM $\mathcal{R}_l$-modules $\mathcal{N}_l$ and $\mathcal{N}^{+}_l$.

\begin{pproposition}[See Proposition \ref{Prop:MCM Equivalence}]
There are isomorphisms $e_0 \mathcal{A}_l e_0 \cong \mathcal{R}_l \cong e_0 \mathcal{A}_l^{\op}e_0$. The $\mathcal{R}_l$-modules $\mathcal{N}_l:=e_0\mathcal{A}_l (1-e_0)$ and $\mathcal{N}_l^{+}:= (1-e_0) \mathcal{A}_l e_0$ are maximal Cohen Macaulay such that $\End_{\mathcal{R}_l}(\mathcal{R}_l \oplus \mathcal{N}_l)^{\op} \cong \mathcal{A}_l$ and $\End_{\mathcal{R}_l}(\mathcal{R}_l \oplus \mathcal{N}_l^{+})^{\op} \cong \mathcal{A}_l^{\op}$.
\end{pproposition}
These MCM modules are in correspondence with matrix factorisations, see \cite{EisenbudMF} for the relationship between MCMs and matrix factorisations. These matrix factorisations can be explicitly recovered from the presentations of universal flopping algebras in Theorem \ref{Theorem: Explicit presentations} using the calculations outlined in Appendix \ref{Appendix}.

Secondly, moving from a matrix factorisation description to a noncommutative algebra description the Grassman blowup construction of Curto and Morrison is replaced by the moduli of representations construction, and Proposition \ref{Prop:Summary} shows that the universal flop can be recovered from the universal flopping algebra by this moduli construction. (Indeed,  the Grassman blowup of the $R$-module $N$ can be seen to be the moduli of 0-generated dim $(1,\rk \, N)$ representations of $\End_{R}(R \oplus N)^{\op}$.)

Finally, the existence of a matrix factorisation of the form $(x I_{2l} - \Xi)(x I_{2l} + \Xi)$ proposed by Curto and Morrison in Conjecture \ref{conj:CurtoMorrison2} follows by direct calculation from the universal flopping algebras, see Appendix \ref{Appendix}.

This interpretation of the universal flopping algebras as $\mathcal{A}_l:=\End_{\mathcal{R}_l}(\mathcal{R}_l \oplus \mathcal{N}_l)^{\op}$, the moduli construction given in Proposition \ref{Prop:Summary}, and the explicit calculations outlined in Appendix \ref{Appendix} confirm Curto and Morrison's Conjecture \ref{conj:CurtoMorrison2} for all diagrams occurring in the length classification from a noncommutative algebra perspective. 

\begin{example}(Length 2)
The main result of Curto and Morrison's work, \cite[Section 4, Main Theorem]{CurtoMorrison}, shows that Conjecture \ref{conj:CurtoMorrison2} holds in lengths 1 and 2.  
We explicitly recover the length 2 example from the universal flopping algebra of length 2 in Example \ref{Example: Length 2}.

In particular, Curto and Morrison realise the universal flop of length 2, \cite[Section 5]{CurtoMorrison}, via the data of the commutative algebra $\mathcal{R}_2:=\mathbb{C}[x,y,z,u,w,v,t]/f$ where
\begin{equation} \label{Equation:Morrison+Curto}
f=x^2+uy^2+2vyz+wz^2+(uw-v^2)t^2.
\end{equation} and a matrix factorisation 
\begin{equation} \label{Equation:MatrixFactorisation}
\Phi_2=\begin{pmatrix}
x-vt & y & z & t \\
-uy-2zv & x+vt & -ut & z \\
-wz & wt & x-vt & -y \\
-uwt & -wz & uy+2vz & x+vt 
\end{pmatrix}
\end{equation}
and
\begin{equation}
\Phi^+_2=\begin{pmatrix}
x+vt & -y & -z & -t \\
uy+2zv & x-vt & ut & -z \\
wz & -wt & x+vt & y \\
uwt & wz & -uy-2vz & x-vt 
\end{pmatrix},
\end{equation}
such that $\Phi_2 \Phi^+_2 = f \textnormal{I}_4 = \Phi^+_2 \Phi_2$. that encodes a pair of maximal Cohen-Macaulay (MCM) $\mathcal{R}_2$-modules $\mathcal{N}_2$ and $\mathcal{N}_2^{+}$. The corresponding "Grassman blowup" is discussed in detail in \cite[Section 5]{CurtoMorrison}. 
\end{example}

\begin{remark}
While the universal flopping algebras are presented in a compact manner, the corresponding hypersurface equations and matrix factorisations for lengths $>2$ are too large to be explicitly recorded in this paper. The calculations outlined in Appendix \ref{Appendix} allow the interested reader to construct these equations and matrix factorisations for themselves.
\end{remark}

\subsection*{Universal flops and contraction algebras}

A further advantage of the universal flopping algebras are ready-built in the language of the homological minimal model program. 

An invariant arising from this noncommutative approach to understanding the minimal model program is the contraction algebra $A_{\text{con}}$ introduced by Donovan and Wemyss, \cite{WemyssDonovan}. This is a finite dimensional algebra, and Hua has shown that (the $A_{\infty}$ enhanced) contraction algebra is a complete invariant of simple flops: a threefold flop is classified up to analytic isomorphism by the $A_{\infty}$-structure of the contraction algebra \cite{ZhengHua}. As such it should be expected that any invariant of simple threefold flopping contractions can be recovered from the contraction algebra.

It was shown by Toda and Hua that the contraction algebra encodes the length invariant and Gopakumar-Vafa invariants in a straightforward manner, see \cite{TodaGVInvariants} and \cite{HuaToda}.

\begin{proposition}[{\cite[Theorem 4.2]{HuaToda}}] \label{Prop:GKV}
Let $A_{\textnormal{con}}$ be the contraction algebra associated to a simple threefold flopping contraction. There exists a deformation of $A_{\textnormal{con}}$ with a fibre that is a semisimple algebra, and for any such deformation and fibre the semisimple algebra decomposes as $\bigoplus_{i=1}^{6} \textnormal{Mat}_{i\times i}(\mathbb{C})^{\oplus n_i}$ for $n_i \in \mathbb{N}$ the Gopakumar-Vafa invariants associated to the flopping contraction. As the dimension of fibres is constant under in flat families it follows that 
\[ \dim A_{\textnormal{con}} = \sum n_i \, i^2 
\quad \text{ and } \quad
\dim \frac{A_{\textnormal{con}}}{\langle [a,b]: a,b \in A_{\textnormal{con}} \rangle} = n_1.
\]
Moreover, the largest $i$ such that $n_i$ is nonzero is the length of the flop. 
\end{proposition}

In particular, Brown and Wemyss have managed to give greater insight into these Gopakumar-Vafa curve invariants using explicit calculations with the contraction algebra \cite{GKVWemyssBrown}.

 The universal flopping algebras produce examples of threefolds flops in such a way that their contraction algebras are already explicitly encoded: for an algebra $A$ produced from a simple threefold flopping contraction by Van den Bergh's construction, see Section \ref{Sec: Translation}, there is a chosen idempotent $e_0 \in A$ such that the contraction algebra is $A_{\textnormal{con}}:= \frac{A}{A e_0 A}$, see \cite[Definition 2.9]{WemyssDonovan}. 

As such, the universal flopping algebra and results in this paper show how explicit examples of threefold flopping contractions with associated contraction algebras can be produced. This supplies a tool that can be used to construct explicit examples of contraction algebras. Calculating the contraction algebra (and a semisimple deformation) associated to an NCCR can verify the length of the corresponding flopping contraction by Proposition \ref{Prop:GKV}.
 
\subsection*{Acknowledgements} The author is supported by EPSRC grant EP/M017516/1. the author would also like to thank Marco Fazzi and Michael Wemyss for interesting discussions and helpful comments.

\section{Kleinian singularities and simultaneous resolutions} \label{Subsection: Kleinian Singularities and deformations}

As discussed in the introduction, to understand Katz and Morrison's classification of simple threefold flops one must work with Kleinian surface singularities, and we set up the required notation in this section.

\subsection{Kleinian singularities and resolutions}
\label{subsection:Kleinian singularities and resolutions}
Kleinian singularities, also known as Du Val or rational double point singularities, are Gorenstein surface singularities. They are classified into types $A_n \, ( n \ge 1), \, D_n \, (n \ge 4), \, E_6, \, E_7$, and $E_8$. An isolated surface singularity is said to be a Kleinian singularity of type $A_n$, $D_n$ or $E_n$ if it analytically isomorphic to the corresponding a hypersurface singularity defined by one of the following equations in $\mathbb{C}[x,y,z]$:\begin{enumerate}
\item Type $A_n$: $f=xy-z^{n+1}$,
\item Type $D_n$: $f=x^2-zy^2-z^{n-1}$,
\item Type $E_6$: $f=x^2+y^3+z^4$,
\item Type $E_7$: $f=x^2+y^3+yz^3$, or
\item Type $E_8$: $f=x^2+y^3+z^5$.
\end{enumerate}
There are corresponding type $A_n, \, D_n, \, E_n$ Dynkin diagrams, and  we let $\Gamma$ denote such a Dynkin diagram. Associated to such a diagram $\Gamma$ is  a Kleinian singularity $X_{\Gamma}=\Spec \, R_{\Gamma}$ and, as $X_{\Gamma}$ is a surface, a minimal resolution $\pi_{\Gamma}: Y_{\Gamma} \rightarrow X_{\Gamma}$ exists and is unique. Moreover, the exceptional divisor of the minimal resolution is a collection of rational curves whose dual graph forms a Dynkin diagram of type $\Gamma$. 

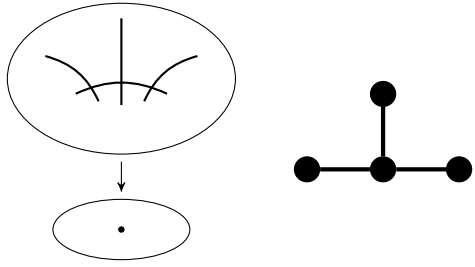
\begin{align*}
\begin{aligned}
\begin{tikzpicture}   
\draw[black, thick] (-0.6,-0.2) to [bend left=25] node[pos=0.48, right] {} (0.6,-0.2);
\draw[black, thick] (0,-0.35) to [bend right=0] node[pos=0.48, right] {} (0,0.8);
\draw[black,thick] (-1,0.3) to [bend left=25] node[pos=0.48, right] {} (-0.3,-0.3);
\draw[black,thick] (1,0.3) to [bend right=25] node[pos=0.48, right] {} (0.3,-0.3);
\draw[black] (0,0) ellipse (1.5 and 1);
\draw[->] (0,-1.1) -- (0,-1.5);
\draw[black] (0,-2) ellipse (0.9 and 0.4);
\filldraw [black] (0,-2) circle (1pt);
\end{tikzpicture}
\end{aligned}
&
\qquad
\begin{aligned}
\begin{tikzpicture}[node distance=1cm, main node/.style={circle,fill=black!100,draw,font=\sffamily\Large\bfseries}]
  \node[main node] (1) at (0,0) {};
  \node[main node] (2) at (-1,0) {};
  \node[main node] (3) at (1,0) {};
  \node[main node] (4)  at (0,1) {};
 \draw [ultra thick] (1) to node {} (2);
 \draw [ultra thick] (1) to node {} (3);
 \draw [ultra thick] (1) to node {} (4);
\end{tikzpicture}
\end{aligned}
\end{align*}
There is a one to one correspondence between exceptional curves in the minimal resolution $Y_{\Gamma} \rightarrow X_{\Gamma}$ and vertices in the corresponding Dynkin diagram $\Gamma$. Associated to a $\circ$/$\bullet$ colouring $C$ of the vertices of $\Gamma$ is a coloured diagram $\Gamma_C$ and a partial resolution $\pi_{\Gamma_C}:Y_{\Gamma_C} \rightarrow X_{\Gamma}$ that contracts the curves corresponding to $\circ$ vertices in $\Gamma_C$.

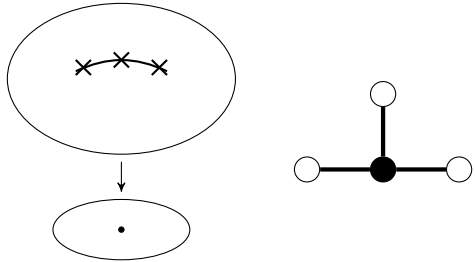
\begin{align*}
\begin{aligned}
\begin{tikzpicture}   
\draw[black, thick] (-0.6,0.1) to [bend left=25] node[pos=0.48, right] {} (0.6,0.1);
\draw[black, thick] (0.1,0.15) to [bend right=0] node[pos=0.48, right] {} (-0.1,0.35);
\draw[black, thick] (-0.1,0.15) to [bend right=0] node[pos=0.48, right] {} (0.1,0.35);
\draw[black, thick] (-0.4,0.05) to [bend right=0] node[pos=0.48, right] {} (-0.6,0.25);
\draw[black, thick] (-0.6,0.05) to [bend right=0] node[pos=0.48, right] {} (-0.4,0.25);
\draw[black, thick] (0.4,0.05) to [bend right=0] node[pos=0.48, right] {} (0.6,0.25);
\draw[black, thick] (0.6,0.05) to [bend right=0] node[pos=0.48, right] {} (0.4,0.25);
\draw[black] (0,0) ellipse (1.5 and 1);
\draw[->] (0,-1.1) -- (0,-1.5);
\draw[black] (0,-2) ellipse (0.9 and 0.4);
\filldraw [black] (0,-2) circle (1pt);
\end{tikzpicture}
\end{aligned}
&
\qquad
\begin{aligned}
\begin{tikzpicture}[node distance=1cm, main node/.style={circle,fill=black!100,draw,font=\sffamily\Large\bfseries}]
  \node[main node] (1) at (0,0) {};
  \node[circle, draw] (2) at (-1,0) {};
  \node[circle, draw] (3) at (1,0) {};
  \node[circle, draw] (4)  at (0,1) {};
 \draw [ultra thick] (1) to node {} (2);
 \draw [ultra thick] (1) to node {} (3);
 \draw [ultra thick] (1) to node {} (4);
\end{tikzpicture}
\end{aligned}
\end{align*}

\begin{remark} \label{Remark:Group Quotient}
Each singularity $X_\Gamma$ is isomorphic to a quotient singularity $\mathbb{C}^2/G$ for a finite subgroup $G< \SL_2(\mathbb{C})$, or equivalently each $R_{\Gamma}$ is isomorphic to an invariant ring $\mathbb{C}[u,v]^G$. In particular, each $R_\Gamma$ has a $\mathbb{Z}$-grading given by $u,v$ having degree 1, this induces a $\mathbb{C}^*$-action on $X_{\Gamma}$ making $R$ a complete local graded $\mathbb{C}$-algebra, and the partial resolution is $\mathbb{C}^*$-equivariant. In particular, $R_{\Gamma}$ is a complete local graded $\mathbb{C}$-algebra with a unique graded maximal ideal $\mathfrak{m}$. 
\end{remark}
\begin{remark} \label{Remark:grading}
For such a complete local graded $\mathbb{C}$-algebra $R$ there is an induced $\mathbb{C}^*$-action on the (ungraded) completion of $R$ with respect to unique graded maximal ideal $\mathfrak{m}$,  $\widehat{R}$, and the algebra $\widehat{R}$ is a flat $R$ algebra. The graded ring $R$ can be recovered from $\widehat{R}$ as the direct sum of the $\mathbb{C}^*$-action eigenspaces, which are indexed by the eigenvalues corresponding the characters $\mathbb{Z}$ of $\mathbb{C}^*$, see e.g. \cite[Appendix]{Namikawa} for the explicit details of such a construction.
\end{remark}

\subsection{Diagram combinatorics} \label{Sec:Lie Alg Com} We now recall some further combinatorics attached to the Dynkin diagram $\Gamma$. We let $|\Gamma|$ denote an indexing set for the vertices of $\Gamma$, which we label by positive integers. For each Dynkin diagram $\Gamma$ we denote the associated extended Dynkin diagram by $\widetilde{\Gamma}$ and denote the additional vertex by $0$.

The Dynkin diagram $\Gamma$ has an associated semisimple Lie algebra. This defines a vector space $\mathfrak{h}_{\Gamma}$ of dimension the size of $|\Gamma|$, the corresponding Cartan subalgebra.  We let $\mathbb{H}_\Gamma:= \textnormal{sym}^{\bullet}(\mathfrak{h}^\vee)$ denote the corresponding polynomial ring with generators $t_i$, $i \in |\Gamma|$, and we call this the Cartan polynomial algebra. 

For notational reasons, we choose to present $\mathbb{H}_{\Gamma}$ as having generators $t_i$ for $i \in |\widetilde{\Gamma}|$ which satisfy a single linear relation
\[
\sum \omega_i t_i=0
\]
where $\omega_i$ are the dual Coxeter labels for vertices of the Dynkin diagram; the multiplicity occuring at that vertex in the longest root (equivalently they are the dimensions of the corresponding irreducible representations of $G$ under the McKay correspondence.) For completeness we include this labelling for type $A_1$, $D_4$, $E_6$, $E_7$ and $E_8$ Dynkin diagrams below.
\vspace{-1cm}
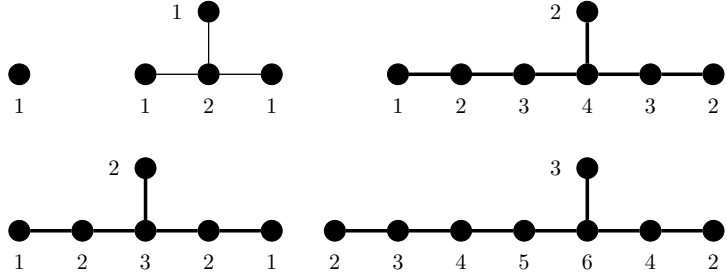
\begin{center}
\[
\resizebox{275pt}{!}{ 
\begin{tikzpicture}[node distance=1cm, main node/.style={circle,fill=black!100,draw,font=\sffamily\Large\bfseries}]
\node (11Lab) at (0,-0.5) {$1$};
  \node[circle,fill=black!100,draw] (11) at (0,0) {};

    \node[main node] (21) at (3,0) {};
    \node (21Lab) at (3,-0.5) {$2$};
  \node[main node] (22) at (2,0) {};
  \node (22Lab) at (2,0-0.5) {$1$};
  \node[main node] (23) at (4,0) {};
  \node (23Lab) at (4,-0.5) {$1$};
  \node[main node] (24)  at (3,1) {};
  \node (24Lab) at (2.5,1) {$1$};
 \draw [main node] (21) to node {} (22);
 \draw [main node] (21) to node {} (23);
 \draw [main node] (21) to node {} (24);

     \node[main node] (41) at (9,0) {};
           \node (21Lab) at (9,-0.5) {$4$};
      \node[main node] (42'') at (6,0) {};
            \node (21Lab) at (6,-0.5) {$1$};
    \node[main node] (42') at (7,0) {};
          \node (21Lab) at (7,-0.5) {$2$};
  \node[main node] (42) at (8,0) {};
        \node (21Lab) at (8,-0.5) {$3$};
  \node[main node] (43) at (10,0) {};
        \node (21Lab) at (10,-0.5) {$3$};
    \node[main node] (43') at (11,0) {};
          \node (21Lab) at (11,-0.5) {$2$};
  \node[main node] (44)  at (9,1) {};
        \node (21Lab) at (8.5,1) {$2$};
 \draw [ultra thick] (41) to node {} (42);
 \draw [ultra thick] (41) to node {} (43);
 \draw [ultra thick] (41) to node {} (44);
  \draw [ultra thick] (42) to node {} (42');
    \draw [ultra thick] (42') to node {} (42'');
   \draw [ultra thick] (43) to node {} (43');

     \node[main node] (31) at (2,-2.5) {};
      \node (21Lab) at (2,-3) {$3$};
    \node[main node] (32') at (0,-2.5) {};
        \node (21Lab) at (0,-3) {$1$};
  \node[main node] (32) at (1,-2.5) {};
      \node (21Lab) at (1,-3) {$2$};
  \node[main node] (33) at (3,-2.5) {};
      \node (21Lab) at (3,-3) {$2$};
    \node[main node] (33') at (4,-2.5) {};
        \node (21Lab) at (4,-3) {$1$};
  \node[main node] (34)  at (2,-1.5) {};
      \node (21Lab) at (1.5,-1.5) {$2$};
 \draw [ultra thick] (31) to node {} (32);
 \draw [ultra thick] (31) to node {} (33);
 \draw [ultra thick] (31) to node {} (34);
  \draw [ultra thick] (32) to node {} (32');
   \draw [ultra thick] (33) to node {} (33');

   \node[main node] (51) at (9,-2.5) {};
             \node (21Lab) at (9,-3) {$6$};
        \node[main node] (52''') at (5,-2.5) {};
                  \node (21Lab) at (5,-3) {$2$};
      \node[main node] (52'') at (6,-2.5) {};
                \node (21Lab) at (6,-3) {$3$};
    \node[main node] (52') at (7,-2.5) {};
              \node (21Lab) at (7,-3) {$4$};
  \node[main node] (52) at (8,-2.5) {};
            \node (21Lab) at (8,-3) {$5$};
  \node[main node] (53) at (10,-2.5) {};
            \node (21Lab) at (10,-3) {$4$};
    \node[main node] (53') at (11,-2.5) {};
              \node (21Lab) at (11,-3) {$2$};
  \node[main node] (54)  at (9,-1.5) {};
                \node (21Lab) at (8.5,-1.5) {$3$};
 \draw [ultra thick] (51) to node {} (52);
 \draw [ultra thick] (51) to node {} (53);
 \draw [ultra thick] (51) to node {} (54);
  \draw [ultra thick] (52) to node {} (52');
    \draw [ultra thick] (52') to node {} (52'');
        \draw [ultra thick] (52'') to node {} (52''');
   \draw [ultra thick] (53) to node {} (53');

\end{tikzpicture}
}
\]
\end{center}
\vspace{-0.1cm}
The additional vertex in the extended Dynkin diagram is always labelled by 1. We also choose to realise $\mathbb{H}_{\Gamma}$ as a graded algebra with each generator $t_i$ in degree 2.

The action of the Weyl group $W_{\Gamma}$ on the Cartan $\mathfrak{h}_{\Gamma}$ defines an action of $W_{\Gamma}$ on $\mathbb{H}_{\Gamma}$. The Weyl group is the Coxeter group associated to the diagram $\Gamma$, with generators $s_i$ for $i \in |\Gamma|$. An action of $W_{\Gamma}$ on $\mathbb{H}_{\Gamma}$ is defined by
\[
s_i(t_j):= \left\{ \begin{array}{c c} 
-t_i & \text{if $i=j$} \\
t_j+kt_i & \text{if vertices $i$ and $j$ joined by $k$ edges} \\
t_j & \text{otherwise}
\end{array} \right.
\]
where we consider vertices and edges in the extended Dynkin diagram $\widetilde{\Gamma}$. A coloured Dynkin diagram $\Gamma_C$ defines a subgroup $W_{C} \subset W_{\Gamma}$ generated by the simple reflections $s_i$ corresponding to uncoloured vertices $\circ$ in $C$, and this also acts on $\mathbb{H}_{\Gamma}$. 

\subsection{Simultaneous partial resolutions of Kleinian singularities} \label{Subsection:Deformations of Kleinian singularities and resolutions}
In order to recall the simultaneous resolutions that determine the universal flop of length $l$, we briefly recap the deformation theory of Kleinian surface singularities and their partial resolutions as developed by Brieskorn \cite{Brieskorn66,Brieskorn68,Brieskorn70} and Tjurina \cite{Tjurina70}. We refer the reader to the book of Slodowy \cite{Slodowy} for a readable and more complete discussion of these topics.

Firstly, versal deformations of the schemes $X_{\Gamma}$, $Y_{\Gamma}$ and $Y_{\Gamma_C}$ defined in the previous section exist, and we denote these versal deformations by the flat morphisms
\[
\mathcal{X}_{\Gamma} \rightarrow  \textnormal{Def}_{\Gamma}, \quad
\mathcal{Y}_{\Gamma} \rightarrow  \textnormal{Res}_{\Gamma}, \quad \text{and} \quad
\mathcal{Y}_{\Gamma_C} \rightarrow  \textnormal{PRes}_{\Gamma_C}
\]
respectively. As versal deformations, any formal deformation of $X_{\Gamma}, \, Y_{\Gamma},$ or $Y_{\Gamma_C}$ has a (non-unique) morphism to the corresponding versal deformation with a uniquely defined differential between the Zariski tangent spaces of the bases. 

What we are choosing to call versal deformations in this paper are called semi-universal deformations in \cite{KatzMorrison} and \cite{Slodowy}; the additional condition we include is the uniqueness of the differential induced between Zariski tangent spaces at a chosen closed point.

The definition of a deformation involves a chosen closed point in the base identifying the fibre being deformed. The $\mathbb{C}^*$-action on $X_\Gamma$ recalled in Remark \ref{Remark:Group Quotient} induces $\mathbb{C}^*$-actions on all the versal deformations, their bases, and the morphisms between them and there is a unique $\mathbb{C}^*$-equivariant closed point in the base, which is the chosen point. In particular, all varieties and morphisms occurring in a universal flop are $\mathbb{C}^*$-equivariant.  A $\mathbb{C}^*$-versal deformation is versal for both deformations with $\mathbb{C}^*$-actions and arbitrary deformations, see \cite{Pinkham}, and so we implicitly always consider $\mathbb{C}^*$-versal deformations where possible.

Secondly, we recall the relationship between $\textnormal{Def}_{\Gamma}, \, \textnormal{Res}_{\Gamma},$ and $\textnormal{PRes}_{\Gamma_C}$. As noted above, a type $\Gamma$ Dynkin diagram has an associated Cartan $\mathfrak{h}_{\Gamma}$ and Weyl group $W_{\Gamma}$, and a coloured diagram $\Gamma_C$ defines a subgroup $W_{C}$ corresponding to the Weyl group associated to the $\circ$ vertices of the Dynkin diagram. Then the bases of these versal deformations can be realised as
\[
\textnormal{Def}_{\Gamma}\cong \mathfrak{h}_{\Gamma}/W_{\Gamma}, \quad
\textnormal{Res}_{\Gamma} \cong \mathfrak{h}_{\Gamma}, \quad \text{and} \quad
\textnormal{PRes}_{\Gamma_C} \cong  \mathfrak{h}_{\Gamma}/W_{C}
\]
and these are related by the obvious quotient maps. As $\mathfrak{h}_{\Gamma}$ is affine space and the Weyl group actions are generated by reflections all these bases are isomorphic to affine space, although with different associated $\mathbb{C}^*$-actions induced from that on $\mathfrak{h}_l$. Further, these quotient maps induce projective morphisms $\mathcal{Y}_\Gamma \rightarrow \mathcal{X}_{\Gamma} \times_{\mathfrak{h}_{\Gamma}/W_{\Gamma}} \mathfrak{h}_{\Gamma}$ and $\mathcal{Y}_{\Gamma_C}  \rightarrow \mathcal{X}_{\Gamma} \times_{\mathfrak{h}_{\Gamma}/W_{\Gamma}} \mathfrak{h}_{\Gamma}/W_{C}$.

Having recalled the existence of these deformations, a \emph{simultaneous partial resolution} of the flat morphism $\mathcal{X}_{\Gamma} \rightarrow \mathfrak{h}_{\Gamma}/W_{\Gamma}$ over $\mathfrak{h}_{\Gamma}/W_{C}$ consists of a commutative diagram
\[
\begin{tikzpicture} [bend angle=0]
\node (X) at (-2,0)  {$\mathfrak{h}_{\Gamma}/W_{C}$};
\node (X') at (2,0)  {$\mathfrak{h}_{\Gamma}/W_{\Gamma}$};
\node (X1) at (-2,1.5)  {$\mathcal{Y}$};
\node (X2) at (2,1.5)  {$\mathcal{X}_{\Gamma}$};
\draw [->] (X1) to node[above]  {$\scriptstyle{\phi}$} (X2);
\draw [->] (X1) to node[left]  {$\scriptstyle{\theta}$} (X);
\draw [->] (X2) to node[left]  {$\scriptstyle{}$} (X');
\draw [->] (X) to node[above]  {$\scriptstyle{\psi}$} (X');
\end{tikzpicture}
\]
where $\theta$ is flat, $\psi$ is the quotient morphism, $\phi$ is projective, and for all closed points $p \in \mathfrak{h}_{\Gamma}/W_{C}$ the map on fibres ${\mathcal{Y}}_p \rightarrow {\mathcal{X}_{\Gamma}}_p$ is a partial resolution of the Kleinian singularity ${\mathcal{X}_{\Gamma}}_p$ dominated by the minimal resolution. 

If we let $\mathcal{X}':= \mathcal{X}_{\Gamma} \times_{\mathfrak{h}_{\Gamma}/W_{\Gamma}} \mathfrak{h}_{\Gamma}/W_{C}$, then is is clear that by the Cartesian property that such a simultaneous resolution determines morphisms
\[
\mathcal{Y} \xrightarrow{\pi} \mathcal{X}' \xrightarrow{\alpha} \mathfrak{h}_{\Gamma}/W_{C},
\] 
such that $\pi$ is projective and $\alpha$ and $\alpha \circ \pi$ are flat. Vice versa, such a flat morphism $\mathcal{Y} \rightarrow \mathfrak{h}_{\Gamma}/W_{C}$ factoring through $\mathcal{X}'$ by a projective morphism determines a simultaneous resolution. Indeed, the versal deformation of $\mathcal{Y}_{\Gamma_C} \rightarrow \mathfrak{h}_{\Gamma}/W_{C}$ of $Y_{\Gamma_C}$ factors through $\mathcal{X}'$ and so provides a simultaneous partial resolution of $\mathcal{X}_{\Gamma} \rightarrow \mathfrak{h}_{\Gamma}/W_{\Gamma}$. 

Simultaneous partial resolutions are not unique, but any simultaneous partial resolution is determined by the blow up of a element of the class group of $\mathcal{X}'$, see \cite[3.1]{Brieskorn68}, and there are a finite number enumerated by Pinkham's result \cite[Theorem 3]{PinkhamConstruction}.

In fact, for partial resolutions $\mathcal{Y}_{\Gamma_C}$ with a single $\bullet$ curve in $\Gamma_C$, such as those in the length classification, Pinkham's result \cite[Theorem 3]{PinkhamConstruction} shows that there are exactly two simultaneous resolutions. As such, associated to the diagram $\Gamma_C$ of length $l$ occurring in Theorem \ref{Theorem: Length Classification} are a pair of simultaneous resolutions producing morphisms 
\begin{align*}
\begin{aligned}
\begin{tikzpicture} [bend angle=0]
\node (Y) at (-1.25,1.5)  {$\mathcal{Y}_l$};
\node (Y+) at (1.25,1.5)  {$\mathcal{Y}_l^{+}$};
\node (X) at (0,0)  {$\mathcal{X}_l$};
\node (H) at (0,-1.5)  {$\mathfrak{h}_l$};
\draw [->] (Y) to node[below left, pos=0.4]  {$\pi_l$} (X);
\draw [->] (Y+) to node[ below right, pos=0.3]  {$\pi_l^{+}$} (X);
\draw [->] (X) to node[ below right, pos=0.3]  {$\alpha$} (H);
\end{tikzpicture}
\end{aligned}
\qquad 
\begin{aligned}
&\text{where:} \\
&\text{(1) $\alpha$, $\alpha \circ \pi_l$, and $\alpha \circ \pi_l^+$ are flat morphisms,} \\
&\text{(2) $\mathcal{X}_l:=\mathcal{X}_{\Gamma} \times_{\mathfrak{h}_{\Gamma}/W_{\Gamma}} \mathfrak{h}_{\Gamma}/W_{C}$,} \\
&\text{(3) Both $\mathcal{Y}_l$ and $\mathcal{Y}_l^{+}$ are isomorphic to $\mathcal{Y}_{\Gamma_C}$ as $\mathbb{C}$-schemes,} \\
&\text{(4) $\mathfrak{h}_l \cong \mathfrak{h}_{\Gamma}/W_{C}$,} \\
&\text{(5) $\mathcal{X}_l =:\Spec \, \mathcal{R}_l$ and $\mathfrak{h}_l =: \Spec \, \mathbb{H}_l$ are affine.}
\end{aligned}
\end{align*}
The varieties $\mathcal{Y}_l^{(+)}$ and $\mathcal{X}_l$ are higher dimensional varieties, and the fibres of $\pi_l$ and $\pi_l^+$ are partial resolutions of Kleinian surface singularities: i.e. the fibres of $\alpha$ are morphisms of two dimensional varieties.

In particular, we call this pair of simultaneous resolutions $\pi_l^{(+)}: \mathcal{Y}_l^{(+)} \rightarrow \mathcal{X}_l$ the \emph{universal flop of length $l$}. Katz and Morrison use this deformation theory in the classification of simple threefold flops, \cite{KatzMorrison}, as discussed in the introduction.

\section{Construction of the universal flopping algebras} \label{Section: Noncommutative algebra construction}

In this section we recall the techniques we will require to translate from geometry to noncommutative algebra and construct the universal flopping algebras.

\subsection{Translation to noncommutative algebra} \label{Sec: Translation}
We translate the classification and construction of simple threefold flops by length from algebraic geometry into noncommutative algebra using general techniques introduced by Van den Bergh in \cite{3Dflops}. These techniques work under the following assumption, which holds for partial resolutions of Kleinian singularies, simple threefold flopping contractions, and the simultaneous partial resolutions of Kleinian singularities considered in Section \ref{Subsection: Kleinian Singularities and deformations}. 

\begin{assumption} \label{ass:fibres}
We suppose that $\pi:Y \rightarrow X \cong \Spec \, R$ is a projective morphism of Noetherian schemes over $\mathbb{C}$ such that 
\begin{enumerate}
\item The fibres of $\pi$ have dimension $\le 1$, and
\item $\mathbf{R}\pi_* \mathcal{O}_Y \cong \mathcal{O}_X$.
\end{enumerate}
\end{assumption}
In this situation the following abelian categories were introduced by Bridgeland \cite{BridgelandFlops} and interpreted more explicitly as module categories by Van den Bergh.

\begin{definition}[{\cite[Section 3]{3Dflops}}] \label{Perverse} Define $\mathfrak{C}$ to be the abelian subcategory of $\coh \, Y$ consisting of $\mathcal{F} \in \coh\, Y$ such that $\mathbb{R}\pi_*\mathcal{F} \cong 0$. For $i=0,1$ the abelian category $^{-i}\Per(Y/X)$ is defined to contain $\mathcal{E} \in D^b(Y)$ which satisfy the following conditions:
\begin{enumerate}
\item The only non-vanishing cohomology of $\E$ lies in degrees $-1$ and $0$.
\item $\pi_*\mathcal{H}^{-1}(\E)=0$ and  $\mathbb{R}^1 \pi_* \mathcal{H}^0(\E)=0$, where $\mathcal{H}^j$ denotes taking the $j^{th}$ cohomology sheaf.
\item For $i=0$, $\Hom_Y(C,\mathcal{H}^{-1}(\E))=0$ for all $C \in \mathfrak{C}$.
\item For $i=1$, $\Hom_Y(\mathcal{H}^0(\E),C)=0$ for all $C \in \mathfrak{C}$.
\end{enumerate}
\end{definition}
It was shown by Van den Bergh that these abelian categories have projective generators so can be interpreted as module categories for noncommutative algebras.

\begin{theorem}[\cite{3Dflops}] \label{Theorem:Van den Bergh Tilting} For $\pi:Y \rightarrow X$ satisfying Assumptions \ref{ass:fibres}:
\begin{enumerate}
\item There is a vector bundle $T=\mathcal{O}_Y \oplus T'$ on $Y$ that is a projective generator of ${^{0}\Per(Y/X)}$, so
\[
{^{0}\Per(Y/X)} \cong A\textnormal{-mod}
\]
where $A:= \End_Y(T)^{\op}$.
\item The vector bundle $T^{\vee}=\mathcal{O}_Y \oplus T'^{\vee}$ on $Y$ is a projective generator of ${^{-1}\Per(Y/X)}$, so
\[
{^{-1}\Per(Y/X)} \cong A^{\op}\textnormal{-mod}
\]
where $A^{\op}:= \End_Y(T^{\vee})^{\op} \cong \End_Y(T)$.
\item There are derived equivalences 
\[
D^b(A\textnormal{-mod}) \cong D^b( \coh \, Y) \cong D^b(A^{\op} \textnormal{-mod}).
\]
\end{enumerate}
\end{theorem}

\begin{remark} \label{Remark:Flat Base Change} We remark that the construction of ${^{-i}\Per(Y/X)}$ commutes with flat base change in $X$, \cite[Proposition 3.1.4.]{3Dflops}.
\end{remark}

\begin{remark} \label{Remark:KrullSchmidt}
In general an algebra $A$ produced by Theorem \ref{Theorem:Van den Bergh Tilting} is only defined up to Morita equivalence; it is the abelian category $A$-$\modcat \cong {^0}\Per(Y/X)$ that is determined up to equivalence. However, when the category of finitely generated projective $A$-modules has the Krull-Schmidt property then we can ask for $A$ to be a basic copy of the algebra; i.e. ask $T$ to be a basic projective generator of ${^0}\Per(Y/X)$ so $A$ decomposes into one copy of each indecomposable projective module up to isomorphism.

For example, when $R$ is a complete local $\mathbb{C}$-algebra with closed point $p \in X  \cong \Spec R$ corresponding to maximal ideal $\mathfrak{m}_p$ such that $R/\mathfrak{m}_p \cong \mathbb{C}$, then the category of projective $A$-modules has the Krull-Schmidt property. In particular, the non-free indecomposable projective modules are in one to one correspondence with the irreducible curves in $\pi^{-1}(p)$, see \cite[Section 3.4]{3Dflops}. 

As in Remark \ref{Remark:grading}, this also holds in the graded setting. When $R$ is a complete local graded $\mathbb{C}$-algebra, $\pi$ is $\mathbb{C}^*$-equivariant, and the projective generator $T$ has a $\mathbb{C}^*$-linearisation then we can also ask for a basic graded algebra $A$. This can be seen by completing the algebra $R$ with respect to the unique graded maximal ideal, realising the basic projective generator $\widehat{T}$ in the complete setting, and then noting there is a corresponding projective generator $T$ such that $T \otimes_R \widehat{R} \cong \widehat{T}$, see the techniques of \cite[Prop A.6]{Namikawa}. Then $A:=\End_Y(T)^{\op}$. 

As in Remark \ref{Remark:grading}, this also holds in the graded setting. When $R$ is a complete local graded $\mathbb{C}$-algebra, $\pi$ is $\mathbb{C}^*$-equivariant, and the projective generator $T$ has a $\mathbb{C}^*$-linearisation then we can also ask for a basic graded algebra $A$. This can be seen by completing the algebra $R$ with respect to the unique graded maximal ideal, realising the basic projective generator $\widehat{T}$ in the complete setting, and then noting there is a corresponding projective generator $T$ such that $T \otimes_R \widehat{R} \cong \widehat{T}$, see the techniques of \cite[Prop A.6]{Namikawa}. Then $A:=\End_Y(T)^{\op}$. 

\end{remark}

For a flopping contraction with 1-dimension fibres (see \cite[Section 4.4]{3Dflops} for the conditions required to be a flopping contraction) there is a further corollary.

\begin{corollary}[{\cite[Theorem 4.4.2]{3Dflops}}]\label{Corollary: derived flops}
If in addition the map $\pi:Y \rightarrow X$ is a flopping contraction with 1-dimensional fibres with flop $\pi^{+}:Y^{+} \rightarrow X$ then ${^{-1}\Per(Y/X)} \cong {^{0}\Per(Y^{+}/X)}$. In particular
$D^b(Y) \cong D^b(Y^+)$.
\end{corollary}

Van den Bergh's results provide an algebraic interpretation of Bridgeland's proof that both sides of a threefold flop are derived equivalent, \cite{BridgelandFlops}, by showing that the abelian categories of perverse sheaves considered by Bridgeland can be reinterpreted as module categories of noncommutative algebras. In addition, Bridgeland constructed both sides of the flop as moduli of stable objects in two t-structures, and in this noncommutative algebra setting this moduli construction can be translated to the moduli of 0-generated modules which we discuss next.

\subsection{Moduli construction} \label{Sec:Moduli}
We now recall a moduli construction that can recover the geometry $\pi: Y \rightarrow X$ from an algebra $A$ produced in Theorem \ref{Theorem:Van den Bergh Tilting}.

The algebra $A:=\End_{Y}(T)^{\op}$ comes equipped with a chosen idempotent $e_0 \in A$ correponding to the trivial summand in the chosen splitting $T=\mathcal{O}_Y \oplus T'$. Then a finite dimensional left $A$-module $V$ splits into two vector spaces $V_0:=e_0V$ and $V_1:=(1-e_0)V$, and we say the dimension of such a module is $(\dim V_0, \dim V_1)$. An $A$-module of dimension $(1, \rk T')$ is 0-generated if the action of $A$ on $V_0$ generates the whole of $V$. We let $\mathcal{M}_A$ denote the fine moduli space of 0-generated $A$-modules with dimension vector $(1,\rk T')$. Such a moduli space can be constructed via King's GIT construction for algebras presented as quivers with relations, \cite{KingQGIT}, or see Van den Bergh's moduli construction \cite[Section 6.2]{NCCRs} for an alternate version. 

Indeed, when $Y$ is a crepant resolution of $X$ the algebras $A$ and $A^{\op}$ are noncommutative crepant resolutions (NCCRs) of $R$ in the sense of \cite{NCCRs}, see \cite[Corollary 3.2.11]{3Dflops}, and Van den Bergh has shown that for any generic stability condition the moduli of representations of an NCCR $A$ in dimension 3 constructs a crepant resolution of $R$, see \cite[Theorem 6.3.1]{NCCRs}, although this may not be $Y$ itself.

The following result applies in the setting of Assumption \ref{ass:fibres}, where there is no requirement of $Y$ being crepant or a resolution, and shows that it is $Y$ itself that is recovered as $\mathcal{M}_A$ for the algebra $A$ produced by Theorem \ref{Theorem:Van den Bergh Tilting}.

\begin{proposition}[{\cite[Corollary 5.2.4]{Karmazyn2017}}]  \label{proposition: Moduli space construction}
In the notation of Theorem \ref{Theorem:Van den Bergh Tilting}, there is an isomorphism $Y \cong \mathcal{M}_A$, and the tautological object under this moduli construction is $T^{\vee}$.
\end{proposition}

This has a further corollary when $\pi:Y \rightarrow X$ is a flopping contraction, where we recall that ${^{-1}\Per(Y/X)} \cong {^{0}\Per(Y^+/X)}$, so $A^+$ is Morita equivalence to $A^{\op}$. 

\begin{corollary}{\cite[Corollary 5.3.2]{Karmazyn2017}}\label{Corollary: Moduli space flops}
If in addition the map $\pi:Y \rightarrow X$ is a flopping contraction then $\mathcal{M}_{A^{\op}} \cong \mathcal{M}_{A^{+}} \cong Y^{+}.$
\end{corollary}

The map $\pi: Y \rightarrow X= \Spec \, R$ defines $Y$ as an $R$-scheme. The moduli space $\mathcal{M}_A$ comes equipped with this morphism $\pi: \mathcal{M}_A \cong Y \rightarrow \Spec R$ corresponding to restricting an $A$-module to an $e_0 A e_0$-module, where $e_0 A e_0 \cong R$, as 
\begin{align*}
e_0 A e_0 &\cong  \mathbf{R} \Hom_{Y}(\mathcal{O}_Y,\mathcal{O}_Y) \\
&\cong  \mathbf{R} \Hom_{X}(\mathcal{O}_X,\mathbf{R}\pi_*\mathcal{O}_Y) \tag{$\mathcal{O}_Y \cong \mathbf{L}\pi^* \mathcal{O}_X$ and adjunction}\\
&\ \cong  \mathbf{R} \Hom_{X}(\mathcal{O}_X,\mathcal{O}_X) \tag{Assumption \ref{ass:fibres} (2)} \\ 
& \cong R. \tag{$X= \Spec \, R$ affine}
\end{align*}
As such the construction in Proposition \ref{proposition: Moduli space construction} and Corollary \ref{Corollary: Moduli space flops} recovers both morphisms $\pi: Y \cong \mathcal{M}_A \rightarrow X$ and $\pi^{+}: Y^{+} \cong \mathcal{M}_{A^{\op}}\rightarrow X$, where $e_0^{\op} A^{\op} e_0^{\op} \cong (e_0 A e_0)^{\op} \cong R^{\op} \cong R$ as $R$ is commutative. This can be rephrased by saying that the moduli construction realises $Y$ and $Y^+$ as $R$-schemes.

\subsection{Universal flopping algebras} \label{Subsection: Universal flopping algebras}
We now apply these techniques to the universal flops of Katz and Morrison appearing in Theorem \ref{Theorem:Universal Flop} and recapped in Section \ref{Subsection:Deformations of Kleinian singularities and resolutions}. The universal flop $\pi^{(+)}:\mathcal{Y}_l^{(+)}\rightarrow \mathcal{X}_l$ of length $l$ satisfies Assumption \ref{ass:fibres}, and we define the \emph{universal flopping algebra of length $l$}
\[
\mathcal{A}_l:= \End_{\mathcal{Y}_l}(\mathcal{O}_{\mathcal{Y}_l} \oplus \mathcal{T}')^{\op}
\]
to be the basic algebra produced from $\pi_l:\mathcal{Y}_l \rightarrow \mathcal{X}_l \cong \Spec \mathcal{R}_l$ by Theorem \ref{Theorem:Van den Bergh Tilting} such that $\mathcal{A}_l$-$\modcat \cong {^{0}\Per(\mathcal{Y}_l/\mathcal{X}_l)}$. It comes equipped with a chosen idempotent $e_0 \in \mathcal{A}_l$ corresponding to the summand $\mathcal{O}_{\mathcal{Y}_l}$. The universal flopping algebra $\mathcal{A}_l$ is a finitely generated $\mathcal{R}_l$-algebra, and both both the algebras $\mathcal{A}_l$ and $\mathcal{R}_l \cong e_0 \mathcal{A}_l e_0$ are flat $\mathbb{H}_l$-algebras where $\Spec\, \mathbb{H}_l := \mathfrak{h}_l$ and $\Spec \, \mathcal{R}_l \cong \mathcal{X}_l$.

The following results are immediate consequences of applying Van den Bergh's construction recalled Theorem \ref{Theorem:Van den Bergh Tilting} and the moduli constructions of Section \ref{Sec:Moduli} to the universal flop of length $l$.

\begin{proposition} \label{Prop:Summary}
There is a derived equivalence 
\[
D^b( \coh \, \mathcal{Y}_{l}) \cong D^b(\mathcal{A}_{l}\text{-mod})
\]
Moreover, the universal flop can be recovered via a moduli space construction:
\begin{align*}
\begin{aligned}
\begin{tikzpicture} [bend angle=0]
\node (Y) at (-1.25,1.5)  {$\mathcal{Y}_l$};
\node (Y+) at (1.25,1.5)  {$\mathcal{Y}_l^{+}$};
\node (X) at (0,0)  {$\mathcal{X}_l$};
\draw [->] (Y) to node[below left, pos=0.4]  {$\pi_l$} (X);
\draw [->] (Y+) to node[ below right, pos=0.3]  {$\pi_l^{+}$} (X);
\end{tikzpicture}
\end{aligned}
\cong
\begin{aligned}
\begin{tikzpicture} [bend angle=0]
\node (Y) at (-1.25,1.5)  {$\mathcal{M}_{\mathcal{A}_l}$};
\node (Y+) at (1.25,1.5)  {$\mathcal{M}_{\mathcal{A}^{\op}_l}$};
\node (X) at (0,0)  {$\Spec \, e_0 \mathcal{A}_l e_0$};
\draw [->] (Y) to node[below left, pos=0.4]  {$\pi_l$} (X);
\draw [->] (Y+) to node[ below right, pos=0.3]  {$\pi_l^{+}$} (X);
\end{tikzpicture}
\end{aligned}
\end{align*}
where $\mathcal{R}_l \cong e_0 \mathcal{A}_l e_0 \cong e_0 \mathcal{A}_l^{\op}e_0$, $\Spec \, \mathcal{R}_l \cong \mathcal{X}_l$, and $\mathcal{M}_{\mathcal{A}_l} \cong \mathcal{Y}_l$ and $\mathcal{M}_{\mathcal{A}_l^{\op}} \cong \mathcal{Y}_l^{+}$ as $\mathcal{R}_l$-schemes.
\end{proposition}

The algebra $\mathcal{A}_l$ can also be interpreted on the level of $\mathcal{R}_l$ via MCM modules.
\begin{proposition} \label{Prop:MCM Equivalence}
Consider the $\mathcal{R}_l$-modules $\mathcal{N}_l:=e_0\mathcal{A}_l(1-e_0)$ and $\mathcal{N}_l^{+} := (1-e_0)\mathcal{A}_l e_0 $. These are MCM $\mathcal{R}_l$-modules and there are isomorphisms
\[ \mathcal{A}_l \cong \End_{\mathcal{R}_l}(\mathcal{R}_l \oplus \mathcal{N}_l)^{\op} \quad \text{and} \quad  \mathcal{A}_l^{\op} \cong \End_{\mathcal{R}_l}(\mathcal{R}_l \oplus \mathcal{N}_l^{+})^{\op}. \]
\end{proposition}
\begin{proof} Recall the splitting $T=\mathcal{O}_{\mathcal{Y}_l} \oplus T'$. Then by  definition $\mathbf{R}\pi_*(T') \cong \mathbf{R}\Hom_Y(\mathcal{O}_Y,T') \cong e_0 \mathcal{A}_l(1-e_0) \cong  \mathcal{N}_2$ and $\mathbf{R}\pi_*(T'^{\vee})\cong \mathbf{R}\Hom_Y(T',\mathcal{O}_Y) \cong (1-e_0)\mathcal{A}_l e_0 \cong \mathcal{N}_2^{+}$. These are left and right MCM $\mathcal{R}_l$-modules respectively by  \cite[Lemma 3.2.9]{3Dflops}, and as $\End_{\mathcal{Y}_l}(T) \cong \End_{\mathcal{X}_l}(\mathbf{R}\pi_* T)$ and $\End_{\mathcal{Y}_l}(T^{\vee}) \cong \End_{\mathcal{X}_l}(\mathbf{R}\pi_* T^{\vee})$   it follows that 
\[ \mathcal{A}_l \cong \End_{\mathcal{R}_l}(\mathcal{R}_l \oplus \mathcal{N}_l)^{\op} \quad \text{and} \quad  \mathcal{A}_l^{\op} \cong \End_{\mathcal{R}_l}(\mathcal{R}_l \oplus \mathcal{N}_l^{+})^{\op} \] by \cite[Proposition 3.2.10.]{3Dflops} 
(or see the proof of  \cite[Lemma 3.6]{WemyssGL2}).
\end{proof}

\begin{remark}
As $(\mathcal{R}, \mathfrak{m})$ is a complete graded local $\mathbb{C}$-algebra with $\mathcal{R}/\mathfrak{m} \cong \mathbb{C}$ there is a unique non-free indecomposable MCM $\mathcal{R}_l$-module by Remark \ref{Remark:KrullSchmidt} in the graded setting. As we assume that $\mathcal{A}_l \cong \End_{\mathcal{R}_l}(\mathcal{R}_l \oplus \mathcal{N}_l)^{\op}$ is basic, we assume that  $\mathcal{N}_l$ is the unique non-free indecomposable MCM $\mathcal{R}_l$-module. Moreover, over the central fibre of $\mathfrak{h}_l$ the module $\mathcal{N}_l$ restricts to an MCM module on the Kleinian singularity $R_{\Gamma}$. Such modules are classified, see \cite{ArtinVerdier}, and it can be seen that this module has rank $l$.
\end{remark}

\section{Explicit realisations of the universal flopping algebras}

Whilst Theorem \ref{Theorem:Van den Bergh Tilting} shows noncommutative algebras modelling chosen geometric situations exist, it does not address how to determine explicit presentations of these algebras. We now recall several classes of morphisms $\pi:Y \rightarrow \Spec \, R$ satisfying Assumption \ref{ass:fibres} where presentations of algebras $A$ produced by Theorem \ref{Theorem:Van den Bergh Tilting} are known. This progression through examples will build towards finding relatively simple presentations of the universal flopping algebras and proving Theorem \ref{Theorem: Explicit presentations}.

\subsection{Resolutions of Kleinian singularities}\label{sec:Kleinian presentations} We first consider the minimal resolutions of Kleinian singularities
\[
\pi_{\Gamma}:Y_{\Gamma} \rightarrow X_{\Gamma}
\]
recapped in Section \ref{subsection:Kleinian singularities and resolutions}. The singularity $X_{\Gamma}$ is isomorphic to a quotient singularity $\mathbb{C}^2/G$ for some finite group $G<\SL_2(\mathbb{C})$, see Remark \ref{Remark:Group Quotient}. It is known from the McKay correspondence that the abelian category ${^{0}\Per(Y_{\Gamma}/X_{\Gamma})}$ is equivalent to $G$-equivariant coherent sheaves on $\mathbb{C}^2$, which in turn is equivalent to the category of finitely generated modules over the skew group algebra $\mathbb{C}[u,v] \rtimes G$. 

In this situation the conclusions of Theorem \ref{Theorem:Van den Bergh Tilting} were previously known from the derived $\SL_2(\mathbb{C})$ McKay correspondence, where the derived equivalence between $Y_{\Gamma}$ and $\mathbb{C}[u,v] \rtimes G$ was constructed by Kapranov and Vasserot \cite{KapranovVasserot} and the moduli construction $\mathcal{M}_{A_{\Gamma}}$ of the minimal resolution is the G Hilbert scheme construction of Ito and Nakamura \cite{ItoNakamura}.

However, the skew group algebra is not basic, and we let $A_{\Gamma}$ denote the basic noncommutative algebra constructed by Theorem \ref{Theorem:Van den Bergh Tilting} which is isomorphic to the preprojective algebra of extended Dynkin type $\widetilde{\Gamma}$.

\begin{definition} We recall the \emph{preprojective algebra of extended Dynkin type $\widetilde{\Gamma}$} is defined as the path algebra of a quiver with relations as follows.

\begin{enumerate}
\item  Let $Q=(Q_0, Q_1)$ be quiver corresponding to some choice of orientation on the arrows of the extended Dynkin diagram $\widetilde{\Gamma}$.

\item Label the arrows. 

\item Produce the double quiver $\widetilde{Q}=(Q_0, Q_1 \cup Q_1^*)$ by introducing an arrow $a^*:j \rightarrow i$ for each arrow $a:i \rightarrow j \in Q_1$. 

\item The preprojective algebra is defined to be the path algebra over $\mathbb{C}$ of the quiver $\widetilde{Q}$ with relations defined by $\sum_{a \in Q_1} [a,a^*]$.
\end{enumerate}
\end{definition}

We note that the relation $\sum_{a \in Q_1} [a,a^*]$ breaks up into one relation $e_i (\sum_{a \in Q_1} [a,a^*]) e_i$ for each vertex $i$, and we also note that these preprojective algebras are graded with arrows in degree $1$ and generating relations in degree 2.

\begin{proposition}[McKay Correspondence]
There is an isomorphism of algebras.
\[
A_{\Gamma} \cong \frac{\mathbb{C} \widetilde{Q}}{\langle \sum_{a \in Q_1} [a,a^*] \rangle }.
\] 
\end{proposition}
 
We now list presentations of the preprojective algebras of types $A_1$, $D_4$, $E_6$, $E_7$, and $E_8$.
\begin{example}(Type $A_1)$
The algebra $A_{\Gamma}$ can be presented as the path algebra over $\mathbb{C}$ of the following quiver with relations
\begin{align*}
\begin{aligned}
\begin{tikzpicture} [bend angle=15, looseness=1]
\node (C0) at (0,0)  {$0$};
\node (C1) at (2,0)  {$1$};
\draw [->,bend left=35] (C0) to node[gap]  {\scriptsize{$a_0$}} (C1);
\draw [->,bend left] (C0) to node[gap]  {\scriptsize{$a_1^*$}} (C1);
\draw [->,bend left] (C1) to node[gap]  {\scriptsize{$a_0^*$}} (C0);
\draw [->,bend left=35] (C1) to node[gap]  {\scriptsize{$a_1$}} (C0);
\end{tikzpicture}
\end{aligned}
& \quad \quad
\begin{aligned}
&a_0a_0^*-a_1^*a_1=0 \\
&a_1a_1^*-a_0^*a_0=0 
\end{aligned}
\end{align*}
In this case $A_{\Gamma} \cong A_{\Gamma_C}$, so this is also a presentation of $A_{\Gamma_C}$.
\end{example}

\begin{example}(Type $D_4$)
The algebra $A_{\Gamma}$ can be presented as the path algebra over $\mathbb{C}$ for the following quiver and relations.
\begin{align*}
\begin{aligned}
\begin{tikzpicture} [bend angle=15, looseness=1]
\node (C0) at (-1.3,-1.3)  {$0$};
\node (C1) at (-1.3,1.3)  {$1$};
\node (C2) at (1.3,1.3)  {$2$};
\node (C3) at (1.3,-1.3)  {$3$};
\node (C4) at (0,0)  {$4$};
\draw [->,bend left] (C0) to node[gap]  {\scriptsize{$a_0$}} (C4);
\draw [->,bend left] (C4) to node[gap]  {\scriptsize{$a_0^*$}} (C0);
\draw [->,bend left] (C1) to node[gap]  {\scriptsize{$a_1$}} (C4);
\draw [->,bend left] (C4) to node[gap]  {\scriptsize{$a_1^*$}} (C1);
\draw [->,bend left] (C2) to node[gap]  {\scriptsize{$a_2$}} (C4);
\draw [->,bend left] (C4) to node[gap]  {\scriptsize{$a_2^*$}} (C2);
\draw [->,bend left] (C3) to node[gap]  {\scriptsize{$a_3$}} (C4);
\draw [->,bend left] (C4) to node[gap]  {\scriptsize{$a_3^*$}} (C3);
\end{tikzpicture}
\end{aligned}
& \quad \quad
\begin{aligned}
&a_0a_0^*=0, \quad a_1a_1^*=0, \\ 
&a_2a_2^*=0, \quad a_3a_3^*=0, \\ 
&a_0^*a_0+a_1^*a_1+a_2^*a_2+a_3^*a_3=0.\\ 
\end{aligned}
\end{align*}
\end{example}

\begin{example}(Type $E_6$)
The algebra $A_{\Gamma}$ of can be presented as the path algebra over $\mathbb{C}$ of the following path algebra with relations.
\begin{align*}
\begin{aligned}
\begin{tikzpicture} [bend angle=15, looseness=1]
  \node (6) at (0,0) {$6$};
  \node (3) at (-1.4,0) {$3$};
   \node (2) at (-2.8,0) {$2$};
  \node (1) at (0,1.4) {$1$};
  \node (0) at (-1.4,1.4) {$0$};
  \node (5)  at (1.4,0) {$5$};
    \node (4)  at (2.8,0) {$4$};
\draw [->,bend left] (0) to node[above]  {\scriptsize{$a_0$}} (1);
\draw [->,bend left] (1) to node[right]  {\scriptsize{$a_1$}} (6);
\draw [->,bend left] (2) to node[above]  {\scriptsize{$a_2$}} (3);
\draw [->,bend left] (3) to node[above]  {\scriptsize{$a_3$}} (6);
\draw [->,bend left] (4) to node[below]  {\scriptsize{$a_4$}} (5);
\draw [->,bend left] (5) to node[below]  {\scriptsize{$a_5$}} (6);
\draw [->,bend left] (1) to node[below]  {\scriptsize{$a_0^*$}} (0);
\draw [->,bend left] (6) to node[left]  {\scriptsize{$a_1^*$}} (1);
\draw [->,bend left] (3) to node[below]  {\scriptsize{$a_2^*$}} (2);
\draw [->,bend left] (6) to node[below]  {\scriptsize{$a_3^*$}} (3);
\draw [->,bend left] (5) to node[above]  {\scriptsize{$a_4^*$}} (4);
\draw [->,bend left] (6) to node[above]  {\scriptsize{$a_5^*$}} (5);
\end{tikzpicture}
\end{aligned}
& \quad \quad
\begin{aligned}
&a_0a_0^*=0, \quad  a_1a_1^*-a_0^*a_0=0,  \\
&a_2a_2^*=0, \quad a_3a_3^* - a_2^*a_2=0,  \\
&a_4a_4^*=0, \quad a_5a_5^* - a_4^*a_4=0, \\
&a_1^*a_1+a_3^*a_3+a_5^*a_5=0.
\end{aligned}
\end{align*}
\end{example}

\begin{example}(Type $E_7$) 
The algebra $A_{\Gamma}$ can be presented as the path algebra over $\mathbb{C}$ of the following quiver with relations.
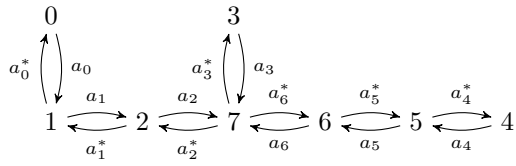
\begin{align*}
\begin{aligned}
\begin{tikzpicture} [bend angle=15, looseness=1]
  \node (7) at (0,0) {$7$};
  \node (3) at (0,1.4) {$3$};
  \node (2) at (-1.2,0) {$2$};
  \node (1) at (-2.4,0) {$1$};
  \node (0) at (-2.4,1.4) {$0$};
  \node (6)  at (1.2,0) {$6$};
  \node (5)  at (2.4,0) {$5$};
  \node (4)  at (3.6,0) {$4$};
\draw [->,bend left] (0) to node[right]  {\scriptsize{$a_0$}} (1);
\draw [->,bend left] (1) to node[above]  {\scriptsize{$a_1$}} (2);
\draw [->,bend left] (2) to node[above]  {\scriptsize{$a_2$}} (7);
\draw [->,bend left] (3) to node[right]  {\scriptsize{$a_3$}} (7);
\draw [->,bend left] (4) to node[below]  {\scriptsize{$a_4$}} (5);
\draw [->,bend left] (5) to node[below]  {\scriptsize{$a_5$}} (6);
\draw [->,bend left] (6) to node[below]  {\scriptsize{$a_6$}} (7);
\draw [->,bend left] (1) to node[left]  {\scriptsize{$a_0^*$}} (0);
\draw [->,bend left] (2) to node[below]  {\scriptsize{$a_1^*$}} (1);
\draw [->,bend left] (7) to node[below]  {\scriptsize{$a_2^*$}} (2);
\draw [->,bend left] (7) to node[left]  {\scriptsize{$a_3^*$}} (3);
\draw [->,bend left] (5) to node[above]  {\scriptsize{$a_4^*$}} (4);
\draw [->,bend left] (6) to node[above]  {\scriptsize{$a_5^*$}} (5);
\draw [->,bend left] (7) to node[above]  {\scriptsize{$a_6^*$}} (6);
\end{tikzpicture}
\end{aligned}
& \quad \quad
\begin{aligned}
&a_0a_0^*=0, \quad a_1a_1^*-a_0^*a_0=0,  \\
&a_3a_3^*=0, \quad a_2a_2^*-a_1^*a_1=0\\
&a_4a_4^*=0, \quad a_5a_5^*-a_4^*a_4=0, \\
&a_6a_6^*-a_5^*a_5=0, \\
&a_2^*a_2+a_3^*a_3+a_5^*a_5=0.
\end{aligned}
\end{align*}
\end{example}

\begin{example}(Type $E_8$) 
The algebra $\mathbb{A}_\Gamma$ can be presented as a path algebra over $\mathbb{C}$ for the following quiver and relations.
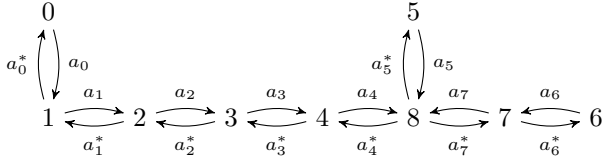
\begin{align*}
\begin{aligned}
\begin{tikzpicture} [bend angle=15, looseness=1]
  \node (8) at (0,0) {$8$};
  \node (5) at (0,1.4) {$5$};
  \node (4) at (-1.2,0) {$4$};
  \node (3) at (-2.4,0) {$3$};
  \node (2) at (-3.6,0) {$2$};
   \node (1) at (-4.8,0) {$1$};
   \node (0) at (-4.8,1.4) {$0$};
  \node (7)  at (1.2,0) {$7$};
    \node (6)  at (2.4,0) {$6$};
\draw [->,bend left] (0) to node[right]  {\scriptsize{$a_0$}} (1);
\draw [->,bend left] (1) to node[above]  {\scriptsize{$a_1$}} (2);
\draw [->,bend left] (2) to node[above]  {\scriptsize{$a_2$}} (3);
\draw [->,bend left] (3) to node[above]  {\scriptsize{$a_3$}} (4);
\draw [->,bend left] (4) to node[above]  {\scriptsize{$a_4$}} (8);
\draw [->,bend left] (5) to node[right]  {\scriptsize{$a_5$}} (8);
\draw [->,bend right] (6) to node[above]  {\scriptsize{$a_6$}} (7);
\draw [->,bend right] (7) to node[above]  {\scriptsize{$a_7$}} (8);
\draw [->,bend left] (1) to node[left]  {\scriptsize{$a_0^*$}} (0);
\draw [->,bend left] (2) to node[below]  {\scriptsize{$a_1^*$}} (1);
\draw [->,bend left] (3) to node[below]  {\scriptsize{$a_2^*$}} (2);
\draw [->,bend left] (4) to node[below]  {\scriptsize{$a_3^*$}} (3);
\draw [->,bend left] (8) to node[below]  {\scriptsize{$a_4^*$}} (4);
\draw [->,bend left] (8) to node[left]  {\scriptsize{$a_5^*$}} (5);
\draw [->,bend right] (7) to node[below]  {\scriptsize{$a_6^*$}} (6);
\draw [->,bend right] (8) to node[below]  {\scriptsize{$a_7^*$}} (7);
\end{tikzpicture}
\end{aligned}
& \quad \quad
\begin{aligned}
&a_0a_0^*=0, \quad a_1a_1^*-a_0a_0^*=0, \\
&a_5a_5^*=0,\quad  a_2a_2^*-a_1a_1^*=0, \\
&a_6a_6^*=0, \quad a_7a_7^*-a_6a_6^*=0,\\
&a_3a_3^*-a_2a_2^*=0, \\
&a_4a_4^*-a_3a_3^*=0, \\
&a_4^*a_4+a_5^*a_5+a_7^*a_7=0.
\end{aligned}
\end{align*}
\end{example}

\subsection{Partial resolutions of Kleinian singularities} \label{Section: Partial resolutions presentations}

The construction in Theorem \ref{Theorem:Van den Bergh Tilting} does not just apply just to smooth situations, for example Assumption \ref{ass:fibres} also applies to the partial resolutions of Kleinian singularities $\pi_{\Gamma_C}:Y_{\Gamma_C} \rightarrow X_{\Gamma}$ considered in Section \ref{subsection:Kleinian singularities and resolutions}. 

For a partial resolution $Y_{\Gamma_C} \rightarrow X_{\Gamma}$ we let $A_{\Gamma_C}$ denote the associated basic noncommutative algebra produced by Theorem \ref{Theorem:Van den Bergh Tilting}. For such partial resolutions it was determined by Kalck, Iyama, Wemyss, and Yang \cite{KIWY} that there is an idempotent $e_C \in A_{\Gamma}$ producing $A_{\Gamma_C}$ from $A_{\Gamma}$. 

We recall that the non-free projective modules are in correspondence with the indecomposable curves in the exceptional divisor, see Remark \ref{Remark:KrullSchmidt}.

\begin{proposition}[{\cite[Theorem 4.6.]{KIWY}}]
Let $e_C$ denote the sum of the idempotents corresponding to the extending vertex and the coloured vertices $\bullet$ in $\Gamma_C$, then there is an isomorphism
\[
A_{\Gamma_C} \cong e_C A_{\Gamma} e_C.
\]
\end{proposition}

For each coloured diagram $\Gamma_C$ occurring in the length classification of flops with $C=\{0,c\} \subset |\widetilde{\Gamma}|$ we present the algebra $e_C \mathcal{A}_{\Gamma} e_C \cong A_{\Gamma_C}$ as the path algebra over $\mathbb{C}$ of a path algebra with relations below. To do this we use the explicit presentations of $A_{\Gamma}$ as quiver with relations given in Section \ref{sec:Kleinian presentations} immediately above and calculate the projective resolutions of various simple objects. 

In particular, the first 3 terms in a minimal graded projective resolution of the simple modules determine the number of arrows and relations required in a minimal presentation as the path algebra of a quiver with relations. Let $P_i^{(d)}$ and $s_i^{(d)}$ denote the projective and simple modules at vertex $i$ in degree $d$ respectively. Then if the minimal projective resolution has the form
 \[
 \dots \rightarrow \bigoplus_{j\in \widetilde{\Gamma}} \bigoplus_{l=1}^{v_{ij}} P_j^{(d_l^{ij})} \rightarrow \bigoplus_{j \in \widetilde{\Gamma}} \bigoplus_{l=1}^{u_{ij}} P_j^{(d_l^{ij})} \rightarrow P_i^{(0)} \rightarrow s_i^{(0)} \rightarrow 0
 \]
then the number of arrows from vertex $i$ to $j$ is $v_{ij}$ with degrees $d^{ij}_1 \dots d^{ij}_{v_{ij}}$ and the number of generating relations with head at vertex $i$ and tail at vertex $j$ is $u_{ij}$ with degree $d^{ij}_1 \dots d^{ij}_{u_{ij}}$. See \cite[Theorem 7.3]{MinimalButlerKing}  for such a description of the minimal resolution or see \cite[Section 3]{BuanIyamaReitenSmith} for a description of the number of generators and relations required for the presentation of an algebra in terms of $\dim \Ext^k_{e_CA_{\Gamma}e_C}(s_0,s_c)$ and $ \dim \Ext^k_{e_CA_{\Gamma}e_C}(s_0,s_c)$ for $k=1,2$ which can be calculated from these resolutions (although this description is given in the complete case it also applies in the graded case with arrows in degree $>0$).

In particular, in our examples each algebra $A_{\Gamma}$ is $\mathbb{Z}$-graded, and each graded simple module $s_i^{(d)}$ has a graded projective resolution
 \[
 0 \rightarrow P_i^{(d+2)} \rightarrow \bigoplus_{j\in \widetilde{\Gamma}} (P_j^{(d+1)})^{\oplus \gamma_{ij}} \rightarrow P_i^{(d)} \rightarrow s_i^{(d)} \rightarrow 0 
 \]
where $\gamma_{ij}$ equals the number of edges joining $i$ and $j$ in the corresponding extended Dynkin diagram $\widetilde{\Gamma}$.

Starting with the resolutions of $s_c^{(0)}$ and $s_0^{(0)}$ by repeated approximation these resolutions can be used to construct two left-unbounded graded complexes consisting of $P_0^{(d)}$ and $P_c^{(d)}$, with $s_c^{(0)}$ and $s_0^{(0)}$ on the right, and with cohomology sheaves direct sums of $s_j^{(d)}$ for $j \neq 0,c$. After restriction to $e_C A_{\Gamma} e_C$ such complexes become acyclic and (unbounded) projective resolutions of $s_0^{(0)}$ and $s_c^{(0)}$, and hence these resolutions determine the presentation.

For the diagram $\Gamma_C$ associated to each length $l$ below we list the first 3 terms in this graded projective resolution, we give a presentation $B:=\mathbb{C}Q/I$ with an implicit surjective graded morphism $ B \rightarrow e_C A_{\Gamma} e_C$, and as this presentation has the correct number of arrows and relations in the correct degrees it is in fact an isomorphism. 

\begin{example}(Type $A_1)$
In this case $A_{\Gamma} \cong A_{\Gamma_C}$, so the algebra $ A_{\Gamma_C}$ the resolutions of the simples $s_0^{(0)}$ and $s_1^{(0)}$ in $e_C A_{\Gamma} e_C$ are 
for \[
0 \rightarrow P_0^{(2)} \rightarrow (P_1^{(1)})^{\oplus 2} \rightarrow P_0^{(0)} \rightarrow s_0^{(0)} \quad \text{and} \quad
0 \rightarrow P_1^{(2)} \rightarrow (P_0^{(1)})^{\oplus 2} \rightarrow P_1^{(0)} \rightarrow s_1^{(0)}
\]
and $A_{\Gamma} \cong A_{\Gamma_C}$ can be presented as the path algebra over $\mathbb{C}$ of the following quiver with relations
\begin{align*}
\begin{aligned}
\begin{tikzpicture} [bend angle=15, looseness=1]
\node (C0) at (0,0)  {$0$};
\node (C1) at (2,0)  {$1$};
\draw [->,bend left=35] (C0) to node[gap]  {\scriptsize{$a_0$}} (C1);
\draw [->,bend left] (C0) to node[gap]  {\scriptsize{$a_1^*$}} (C1);
\draw [->,bend left] (C1) to node[gap]  {\scriptsize{$a_0^*$}} (C0);
\draw [->,bend left=35] (C1) to node[gap]  {\scriptsize{$a_1$}} (C0);
\end{tikzpicture}
\end{aligned}
& \quad \quad
\begin{aligned}
&a_0a_0^*-a_1^*a_1=0 \\
&a_1a_1^*-a_0^*a_0=0 
\end{aligned}
\end{align*}
\end{example}

\begin{example}(Type $D_4$) \label{Example:length 2 singular} 
For the idempotent $e_C:= e_0 + e_4 \in A_{\Gamma}$, the beginning of the resolutions of the simples $s_0^{(0)}$ and $s_4^{(0)}$ in $e_C A_{\Gamma} e_C$ are 
\[
P_0^{(2)} \rightarrow P_4^{(1)} \rightarrow P_0^{(0)} \rightarrow s_0^{(0)}.
\]
and 
\[
(P_4^{(4)})^{\oplus 3} \rightarrow (P_4^{(2)})^{\oplus 2} \oplus P_0^{(1)} \rightarrow P_4^{(0)} \rightarrow s_4^{(0)}.
\]
Then the algebra $e_C A_{\Gamma} e_C$ can be presented as 
\begin{align*}
\begin{aligned}
\begin{tikzpicture} [bend angle=15, looseness=1]
\node (C0) at (0,0)  {$0$};
\node (C1) at (2,0)  {$4$};
\draw [->,bend left] (C0) to node[above]  {\scriptsize{$a_0$}} (C1);
\draw [->,bend left] (C1) to node[below]  {\scriptsize{$a_0^*$}} (C0);
\draw [->, looseness=24, in=52, out=128,loop] (C1) to node[above] {$\scriptstyle{b}$} (C1);
\draw [->, looseness=24, in=-128, out=-52,loop] (C1) to node[below] {$\scriptstyle{c}$} (C1);
\end{tikzpicture}
\end{aligned}
& \quad \quad
\begin{aligned}
&a_0a_0^*=0 & & (a_0^*a_0+b+c)^2=0  \\ 
&b^2=0  & &  c^2=0
\end{aligned}
\end{align*}
where $b=a_1^*a_1,$ and $c=a_2^* a_2$. 
\end{example}

\begin{example}(Type $E_6$)
For the idempotent $e_C:=e_0+e_6 \in A_{\Gamma}$. The beginning of the resolutions of the simples $s_0^{(0)}$ and $s_6^{(0)}$ in $e_C A_{\Gamma} e_C$ are 
\[
P_6^{(4)} \oplus P_0^{(4)}   \rightarrow P_6^{(2)} \rightarrow P_0^{(0)} \rightarrow s_0^{(0)}
\]
and 
\[
P_0^{(4)} \oplus P_6^{(4)} \oplus (P_6^{(6)})^{\oplus 2} \rightarrow P_0^{(2)}\oplus (P_6^{(2)})^{\oplus 2} \rightarrow P_6^{(0)} \rightarrow s_6^{(0)}.
\]
The algebra $e_C A_\Gamma e_C$ can be presented as the path algebra over $\mathbb{C}$ of the following quiver with relations.
\begin{align*}
\begin{aligned}
\begin{tikzpicture} [bend angle=15, looseness=1]
\node (C0) at (0,0)  {$0$};
\node (C1) at (2,0)  {$6$};
\draw [->,bend left] (C0) to node[above]  {\scriptsize{$a$}} (C1);
\draw [->,bend left] (C1) to node[below]  {\scriptsize{$a^*$}} (C0);
\draw [->, looseness=24, in=52, out=128,loop] (C1) to node[above] {$\scriptstyle{b}$} (C1);
\draw [->, looseness=24, in=-128, out=-52,loop] (C1) to node[below] {$\scriptstyle{c}$} (C1);
\end{tikzpicture}
\end{aligned}
& \quad \quad
\begin{aligned}
& aa^*=0, & & a^*a=(b+c)^2, \\ 
& (b+c) a^*=0, & & a (b+c) = 0,  \\ 
&b^3=0, & & c^3=0 & &
\end{aligned}
\end{align*}
where
\[
a=a_0a_1, \quad a^*=a_1^* a_0^*, \quad b=a_3^*a_3, \quad \text{and} \quad c=a_5^*a_5.
\]
\end{example}

\begin{example}(Type $E_7$) 
For $e_C=e_0+e_7 \in A_{\Gamma}$ the beginning of the resolutions of the simples $s_0^{(0)}$ and $s_7^{(0)}$ in $e_C A_{\Gamma} e_C$ are 
\[
P_0^{(6)} \oplus P_7^{(5)}    \rightarrow P_7^{(3)} \rightarrow P_0^{(0)} \rightarrow s_0^{(0)}
\]
and 
\[
P_0^{(5)} \oplus P_7^{(4)} \oplus P_7^{(6)} \oplus P_7^{(8)} \rightarrow  P_0^{(3)} \oplus (P_7^{(2)})^{\oplus 2} \rightarrow P_7^{(0)} \rightarrow s_7^{(0)}.
\]
Then the algebra $e_C A_{\Gamma} e_C$ can be presented as the path algebra over $\mathbb{C}$ of the following quiver with relations
\begin{align*}
\begin{aligned}
\begin{tikzpicture} [bend angle=15, looseness=1]
\node (C0) at (0,0)  {$0$};
\node (C1) at (2,0)  {$7$};
\draw [->,bend left] (C0) to node[above]  {\scriptsize{$a$}} (C1);
\draw [->,bend left] (C1) to node[below]  {\scriptsize{$a^*$}} (C0);
\draw [->, looseness=24, in=52, out=128,loop] (C1) to node[above] {$\scriptstyle{b}$} (C1);
\draw [->, looseness=24, in=-128, out=-52,loop] (C1) to node[below] {$\scriptstyle{c}$} (C1);
\end{tikzpicture}
\end{aligned}
& \quad \quad
\begin{aligned}
& aa^*=0, && a^*a=-(b+c)^3,  \\
& (b+c) a^*=0,  && a (b+c) = 0,  \\ 
&b^2=0, && c^4=0
\end{aligned}
\end{align*}
where
\[
a=a_0a_1a_2, \quad a^*=a_2^*a_1^* a_0^*, \quad b=a_3^*a_3, \quad \text{and} \quad c=a_6^*a_6.
\]
\end{example}

\begin{example}(Type $E_8$) 
Let $e_C=e_0+e_4 \in A_\Gamma$. The beginning of the resolutions of the simples $s_0^{(0)}$ and $s_4^{(0)}$ in $e_C A_{\Gamma} e_C$ are 
\[
P_0^{(8)} \oplus P_4^{(6)}    \rightarrow P_4^{(4)} \rightarrow P_0^{(0)} \rightarrow s_0^{(0)}
\]
and 
\[
P_0^{(6)} \oplus (P_4^{(8)})^{\oplus 2} \oplus P_4^{(10)} \rightarrow  P_0^{(4)} \oplus P_4^{(2)} \oplus P_4^{(4)} \rightarrow P_4^{(0)} \rightarrow s_4^{(0)}.
\]
Then $e_C A_{\Gamma} e_C$ can be presented as the path algebra over $\mathbb{C}$ of the following  quiver with relations.

\begin{align*}
\begin{aligned}
\begin{tikzpicture} [bend angle=15, looseness=1]
\node (C0) at (0,0)  {$0$};
\node (C1) at (2,0)  {$4$};
\draw [->,bend left] (C0) to node[above]  {\scriptsize{$a$}} (C1);
\draw [->,bend left] (C1) to node[below]  {\scriptsize{$a^*$}} (C0);
\draw [->, looseness=24, in=52, out=128,loop] (C1) to node[above] {$\scriptstyle{b}$} (C1);
\draw [->, looseness=24, in=-128, out=-52,loop] (C1) to node[below] {$\scriptstyle{c}$} (C1);
\end{tikzpicture}
\end{aligned}
& \quad \quad
\begin{aligned}
&aa^*=0,  & &a^*a=b^4, \\
&a b =0, & & b a^* =0, \\
& c b c +c^2 b +c b^3=0, & & (c+ b^2)^2+b c b = 0 \\
\end{aligned}
\end{align*}
where
\[
a=a_0a_1a_2a_3, \quad a^*=a_3^*a_2^*a_1^*a_0^*,  \quad b=a_4a_4^*, \quad \text{and} \quad  c=a_4a_5^*a_5a_4^*.
\]

Let $e_C:=e_0+e_8 \in A_{\Gamma}$. The beginning of the resolutions of the simples $s_0^{(0)}$ and $s_8^{(0)}$ in $e_C A_{\Gamma} e_C$ are 
\[
P_0^{(10)} \oplus P_8^{(7)}    \rightarrow P_8^{(5)} \rightarrow P_0^{(0)} \rightarrow s_0^{(0)}
\]
and 
\[
P_0^{(7)} \oplus P_8^{(4)} \oplus P_8^{(6)} \oplus P_8^{(10)} \rightarrow   P_0^{(5)} \oplus (P_8^{(2)})^{\oplus 2} \rightarrow P_8^{(0)} \rightarrow s_8^{(0)}.
\]
Then the algebra $e_C \mathcal{A}_\Gamma e_C$ can be presented as a path algebra over $\mathbb{C}$ of the following quiver with relations.
\begin{align*} 
\begin{aligned}
\begin{tikzpicture} [bend angle=15, looseness=1]
\node (C0) at (0,0)  {$0$};
\node (C1) at (2,0)  {$8$};
\draw [->,bend left] (C0) to node[above]  {\scriptsize{$a$}} (C1);
\draw [->,bend left] (C1) to node[below]  {\scriptsize{$a^*$}} (C0);
\draw [->, looseness=24, in=52, out=128,loop] (C1) to node[above] {$\scriptstyle{b}$} (C1);
\draw [->, looseness=24, in=-128, out=-52,loop] (C1) to node[below] {$\scriptstyle{c}$} (C1);
\end{tikzpicture}
\end{aligned}
& \quad \quad
\begin{aligned}
& aa^*=0, & & a^*a=-(b+c)^5, && \\ 
& (b+c) a^*=0, & & a  (b+c) = 0, \\ 
&b^2=0, & &  c^3=0 &&
\end{aligned}
\end{align*}
where 
\[
a=a_0a_1a_2a_3a_4,\quad a^*=a_4^*a_3^*a_2^*a_1^*a_0^*, \quad b=a_5^*a_5, \quad \text{and} \quad c=a_7^*a_7.
\]
\end{example}

\begin{remark}
We note that the quotients $A_{\Gamma_C}/(A_{\Gamma_C}e_0 A_{\Gamma_C})$ are the finite dimensional algebras $\Delta_{con}$ considered by Wemyss and Donovan \cite[Remark 3.17]{WemyssDonovan}, where the dimensions are calculated as 1, 4, 9, 24, 40 and 60. This dimension can easily be calculated from the presentations above using the methods outlined in Appendix \ref{App:Dimension of contraction algebras}.
\end{remark}

\subsection{Simultaneous resolutions of deformations of Kleinian singularities} \label{sec:Simultaneous Presentations} Similarly, the construction in Theorem \ref{Theorem:Van den Bergh Tilting} does not just apply to surfaces. For example, Assumption \ref{ass:fibres} also holds for the higher  dimensional schemes realised as deformations of surfaces and recalled in Section \ref{Subsection:Deformations of Kleinian singularities and resolutions}; we reuse the notation from that section. 

In particular, a simultaneous resolution $\mathcal{Y}_{\Gamma} \rightarrow \mathcal{X}_{\Gamma} \times_{\mathfrak{h}_{\Gamma}/W_{\Gamma}} \mathfrak{h}_{\Gamma}$ satisfies Assumption \ref{ass:fibres} and is flat over $\mathfrak{h}_{\Gamma}$. In this case we let $\mathcal{A}_{\Gamma}$ denote the associated basic algebra  produced by Theorem \ref{Theorem:Van den Bergh Tilting}, and this is a  a flat $\mathbb{H}_{\Gamma}=\textnormal{Sym}^{\bullet}(\mathfrak{h}_{\Gamma}^{\vee})$-algebra with central fibre $A_{\Gamma}$. We recall that $\mathbb{H}_{\Gamma}$ is defined as a polynomial ring in Section \ref{Sec:Lie Alg Com}.

A presentation of $\mathcal{A}_\Gamma$ was determined by Crawley-Boevey and Holland, \cite{CBH}, by calculating deformations of the projective algebra $A_{\Gamma}$ as a noncommutative algebra and determining the geometric deformations. We recall the definition of the quiver $\widetilde{Q}$ and preprojective algebra from Section \ref{sec:Kleinian presentations}.

\begin{proposition}[Crawley-Boevey and Holland \cite{CBH}] \label{Prop:Deformation Preprojective}
There is an isomorphism of algebras
\[
\mathcal{A}_{\Gamma} \cong \frac{\mathbb{H}_{\Gamma}\widetilde{Q}}{\langle \sum_{a \in Q_1} [a,a^*] - \sum_{i \in Q_0} t_i e_i \rangle}.
\]
\end{proposition}

These algebras are also positively graded, with the arrows in degree 1 and the polynomial generators $t_i$ in degree 2.The algebra $\mathcal{A}_\Gamma$ is also a flat graded, and hence free, $\mathbb{H}_{\Gamma}$-module.

We now list the algebras corresponding to Dynkin types $A_1, \, D_4, \, E_6, \, E_7$, and $E_8$.

\begin{example}(Type $A_1$)
The algebra $\mathcal{A}_{\Gamma}$ can be presented as the path algebra over
\[
\mathbb{H}_{\Gamma}:=\frac{\mathbb{C}[t_0,t_1]}{(t_0+t_1)} \cong \mathbb{C}[t].\]
of the following quiver with relations
\begin{align*}
\begin{aligned}
\begin{tikzpicture} [bend angle=15, looseness=1]
\node (C0) at (0,0)  {$0$};
\node (C1) at (2,0)  {$1$};
\draw [->,bend left=35] (C0) to node[gap]  {\scriptsize{$a_0$}} (C1);
\draw [->,bend left] (C0) to node[gap]  {\scriptsize{$a_1^*$}} (C1);
\draw [->,bend left] (C1) to node[gap]  {\scriptsize{$a_0^*$}} (C0);
\draw [->,bend left=35] (C1) to node[gap]  {\scriptsize{$a_1$}} (C0);
\end{tikzpicture}
\end{aligned}
& \quad \quad
\begin{aligned}
&a_0a_0^*-a_1^*a_1=t e_0, \\
&a_1a_1^*-a_0^*a_0=-te_1.
\end{aligned}
\end{align*}
\end{example}

\begin{example}(Type $D_4$)  The algebra $\mathcal{A}_\Gamma$ can be presented as the path algebra over
\[
\mathbb{H}_{\Gamma}:= \frac{\mathbb{C}[t_0,t_1,t_2,t_3,t_4]}{(t_0+t_1+t_2+t_3+2t_4)}
\]
of the following quiver with relations.

\begin{align*}
\begin{aligned}
\begin{tikzpicture} [bend angle=15, looseness=1]
\node (C0) at (-1.3,-1.3)  {$0$};
\node (C1) at (-1.3,1.3)  {$1$};
\node (C2) at (1.3,1.3)  {$2$};
\node (C3) at (1.3,-1.3)  {$3$};
\node (C4) at (0,0)  {$4$};
\draw [->,bend left] (C0) to node[gap]  {\scriptsize{$a_0$}} (C4);
\draw [->,bend left] (C4) to node[gap]  {\scriptsize{$a_0^*$}} (C0);
\draw [->,bend left] (C1) to node[gap]  {\scriptsize{$a_1$}} (C4);
\draw [->,bend left] (C4) to node[gap]  {\scriptsize{$a_1^*$}} (C1);
\draw [->,bend left] (C2) to node[gap]  {\scriptsize{$a_2$}} (C4);
\draw [->,bend left] (C4) to node[gap]  {\scriptsize{$a_2^*$}} (C2);
\draw [->,bend left] (C3) to node[gap]  {\scriptsize{$a_3$}} (C4);
\draw [->,bend left] (C4) to node[gap]  {\scriptsize{$a_3^*$}} (C3);
\end{tikzpicture}
\end{aligned}
& \quad \quad
\begin{aligned}
&a_0a_0^*=t_0 e_0, \quad a_1a_1^*=t_1e_1, \\ 
&a_2a_2^*=t_2e_2, \quad a_3a_3^*=t_3e_3, \\ 
&a_0^*a_0+a_1^*a_1+a_2^*a_2+a_3^*a_3=-t_4e_4.\\ 
\end{aligned}
\end{align*}
\end{example}

\begin{example}(Type $E_6$)
The algebra $\mathcal{A}_{\Gamma}$ can be presented as the path algebra over
\[
\mathbb{H}_{\Gamma}:=\frac{\mathbb{C}[t_0,t_1,t_2,t_3,t_4,t_5,t_6]}{(t_0+2t_1+t_2+2t_3+t_4+2t_5+3t_6)}
\]
of the following quiver with relations.
\begin{align*}
\begin{aligned}
\begin{tikzpicture} [bend angle=15, looseness=1]
  \node (6) at (0,0) {$6$};
  \node (3) at (-1.4,0) {$3$};
   \node (2) at (-2.8,0) {$2$};
  \node (1) at (0,1.4) {$1$};
  \node (0) at (-1.4,1.4) {$0$};
  \node (5)  at (1.4,0) {$5$};
    \node (4)  at (2.8,0) {$4$};
\draw [->,bend left] (0) to node[above]  {\scriptsize{$a_0$}} (1);
\draw [->,bend left] (1) to node[right]  {\scriptsize{$a_1$}} (6);
\draw [->,bend left] (2) to node[above]  {\scriptsize{$a_2$}} (3);
\draw [->,bend left] (3) to node[above]  {\scriptsize{$a_3$}} (6);
\draw [->,bend left] (4) to node[below]  {\scriptsize{$a_4$}} (5);
\draw [->,bend left] (5) to node[below]  {\scriptsize{$a_5$}} (6);
\draw [->,bend left] (1) to node[below]  {\scriptsize{$a_0^*$}} (0);
\draw [->,bend left] (6) to node[left]  {\scriptsize{$a_1^*$}} (1);
\draw [->,bend left] (3) to node[below]  {\scriptsize{$a_2^*$}} (2);
\draw [->,bend left] (6) to node[below]  {\scriptsize{$a_3^*$}} (3);
\draw [->,bend left] (5) to node[above]  {\scriptsize{$a_4^*$}} (4);
\draw [->,bend left] (6) to node[above]  {\scriptsize{$a_5^*$}} (5);
\end{tikzpicture}
\end{aligned}
& \quad \quad
\begin{aligned}
&a_0a_0^*=t_0e_0, \quad  a_1a_1^*-a_0^*a_0=t_1e_1,  \\
&a_2a_2^*=t_2e_2, \quad a_3a_3^* - a_2^*a_2=t_2e_3,  \\
&a_4a_4^*=t_4e_4, \quad a_5a_5^* - a_4^*a_4=t_3e_5, \\
&a_1^*a_1+a_3^*a_3+a_5^*a_5=-t_6e_6.
\end{aligned}
\end{align*}
\end{example}

\begin{example}(Type $E_7$) The algebra $\mathcal{A}_\Gamma$ can be presented as the path algebra over
\[
\mathbb{H}_{\Gamma}:=\frac{\mathbb{C}[t_0,t_1,t_2,t_3,t_4,t_5,t_6,t_7]}{(t_0+2t_1+4t_2+3t_3+2t_4+3t_5+2t_6+ t_7)}
\]
of the following quiver with relations.

\begin{align*}
\begin{aligned}
\begin{tikzpicture} [bend angle=15, looseness=1]
  \node (7) at (0,0) {$7$};
  \node (3) at (0,1.4) {$3$};
  \node (2) at (-1.2,0) {$2$};
  \node (1) at (-2.4,0) {$1$};
  \node (0) at (-2.4,1.4) {$0$};
  \node (6)  at (1.2,0) {$6$};
  \node (5)  at (2.4,0) {$5$};
  \node (4)  at (3.6,0) {$4$};
\draw [->,bend left] (0) to node[right]  {\scriptsize{$a_0$}} (1);
\draw [->,bend left] (1) to node[above]  {\scriptsize{$a_1$}} (2);
\draw [->,bend left] (2) to node[above]  {\scriptsize{$a_2$}} (7);
\draw [->,bend left] (3) to node[right]  {\scriptsize{$a_3$}} (7);
\draw [->,bend left] (4) to node[below]  {\scriptsize{$a_4$}} (5);
\draw [->,bend left] (5) to node[below]  {\scriptsize{$a_5$}} (6);
\draw [->,bend left] (6) to node[below]  {\scriptsize{$a_6$}} (7);
\draw [->,bend left] (1) to node[left]  {\scriptsize{$a_0^*$}} (0);
\draw [->,bend left] (2) to node[below]  {\scriptsize{$a_1^*$}} (1);
\draw [->,bend left] (7) to node[below]  {\scriptsize{$a_2^*$}} (2);
\draw [->,bend left] (7) to node[left]  {\scriptsize{$a_3^*$}} (3);
\draw [->,bend left] (5) to node[above]  {\scriptsize{$a_4^*$}} (4);
\draw [->,bend left] (6) to node[above]  {\scriptsize{$a_5^*$}} (5);
\draw [->,bend left] (7) to node[above]  {\scriptsize{$a_6^*$}} (6);
\end{tikzpicture}
\end{aligned}
& \quad \quad
\begin{aligned}
&a_0a_0^*=t_0e_0, \quad a_1a_1^*-a_0^*a_0=t_1e_1,  \\
&a_3a_3^*=t_3e_3, \quad a_2a_2^*-a_1^*a_1=t_2e_2\\
&a_4a_4^*=t_4e_4, \quad a_5a_5^*-a_4^*a_4=t_5e_5, \\
&a_6a_6^*-a_5^*a_5=t_6e_6, \\
&a_2^*a_2+a_3^*a_3+a_5^*a_5=-t_7e_7.
\end{aligned}
\end{align*}
\end{example}

\begin{example}(Type $E_8$) The algebra $\mathcal{A}_\Gamma$ can be presented as the path algebra over
\[
\mathbb{H}_{\Gamma}:=\frac{\mathbb{C}[t_0,t_1,t_2,t_3,t_4,t_5,t_6,t_7,t_8]}{(t_0+2t_1+3t_2+4t_3+5t_4+3t_5+2t_6+4t_7+6t_8)}
\]
of the following quiver with relations.
\begin{align*}
\begin{aligned}
\begin{tikzpicture} [bend angle=15, looseness=1]
  \node (8) at (0,0) {$8$};
  \node (5) at (0,1.4) {$5$};
  \node (4) at (-1.1,0) {$4$};
  \node (3) at (-2.2,0) {$3$};
  \node (2) at (-3.3,0) {$2$};
   \node (1) at (-4.4,0) {$1$};
   \node (0) at (-4.4,1.4) {$0$};
  \node (7)  at (1.1,0) {$7$};
    \node (6)  at (2.2,0) {$6$};
\draw [->,bend left] (0) to node[right]  {\scriptsize{$a_0$}} (1);
\draw [->,bend left] (1) to node[above]  {\scriptsize{$a_1$}} (2);
\draw [->,bend left] (2) to node[above]  {\scriptsize{$a_2$}} (3);
\draw [->,bend left] (3) to node[above]  {\scriptsize{$a_3$}} (4);
\draw [->,bend left] (4) to node[above]  {\scriptsize{$a_4$}} (8);
\draw [->,bend left] (5) to node[right]  {\scriptsize{$a_5$}} (8);
\draw [->,bend right] (6) to node[above]  {\scriptsize{$a_6$}} (7);
\draw [->,bend right] (7) to node[above]  {\scriptsize{$a_7$}} (8);
\draw [->,bend left] (1) to node[left]  {\scriptsize{$a_0^*$}} (0);
\draw [->,bend left] (2) to node[below]  {\scriptsize{$a_1^*$}} (1);
\draw [->,bend left] (3) to node[below]  {\scriptsize{$a_2^*$}} (2);
\draw [->,bend left] (4) to node[below]  {\scriptsize{$a_3^*$}} (3);
\draw [->,bend left] (8) to node[below]  {\scriptsize{$a_4^*$}} (4);
\draw [->,bend left] (8) to node[left]  {\scriptsize{$a_5^*$}} (5);
\draw [->,bend right] (7) to node[below]  {\scriptsize{$a_6^*$}} (6);
\draw [->,bend right] (8) to node[below]  {\scriptsize{$a_7^*$}} (7);
\end{tikzpicture}
\end{aligned}
& \quad 
\begin{aligned}
&a_0a_0^*=t_0e_0, \quad a_1a_1^*-a_0a_0^*=t_1e_1, \\
&a_5a_5^*=t_5e_5,\quad  a_2a_2^*-a_1a_1^*=t_2e_2, \\
&a_6a_6^*=t_6e_6, \quad a_7a_7^*-a_6a_6^*=t_7e_7,\\
&a_3a_3^*-a_2a_2^*=t_3e_3, \\
&a_4a_4^*-a_3a_3^*=t_4e_4, \\
&a_4^*a_4+a_5^*a_5+a_7^*a_7=-t_8e_8.
\end{aligned}
\end{align*}
\end{example}

\subsection{Universal flopping algebras}
Making use of the presentations we have already determined, we now find a presentation of the the universal floppping algebras $\mathcal{A}_l$ and prove Theorem \ref{Theorem: Explicit presentations}.

We recall the universal flop of length $l$ recapped in Section \ref{Subsection:Deformations of Kleinian singularities and resolutions} with associated morphisms $\mathcal{Y}_l \xrightarrow{\pi_l} \mathcal{X}_l \rightarrow \Spec \, \mathbb{H}_l$. We let $\mathcal{A}_l$ denote the corresponding basic algebra produced from the morphism $\pi_l:\mathcal{Y}_l \rightarrow \mathcal{X}_l$ by Theorem \ref{Theorem:Van den Bergh Tilting} and discussed in Section \ref{Subsection: Universal flopping algebras}. Just as the versal deformation $\mathcal{Y}_{\Gamma_C} \cong \mathcal{Y}_l$ can be produced from the versal deformation $\mathcal{Y}_{\Gamma}$ by base change and blow down, we will realise the universal flopping algebra $\mathcal{A}_l$ from the algebra $\mathcal{A}_{\Gamma}$.

We do this by noting a presentation of the $\mathbb{H}_{\Gamma}$-algebra $\mathcal{A}_{l} \otimes_{\mathbb{H}_l} \mathbb{H}_{\Gamma}$ before recovering the $\mathbb{H}_l$-algebra $\mathcal{A}_l$. We recall from Remark \ref{Remark:KrullSchmidt} that the non-free projective $\mathcal{A}_{\Gamma}$-modules are in one to one correspondence with the irreducible curves in the preimage of the closed point of the completion of $\mathcal{R}$. This is exactly the set of contracting curves occurring in the hyperplane section defined by $\Gamma$. As such there is one non-free projective module corresponding to each vertex in $\Gamma$.

\begin{proposition} \label{Prop:Basechange}
Let $e_C:=e_0+ e_{\bullet}$ be the idempotent corresponding to the 0 vertex and the $\bullet$ curve in $\Gamma_C$. Then 
\[
e_C A_{\Gamma} e_C\text{-}\modcat  \cong {^0}\Per(\mathcal{Y}_{\Gamma_C} \times_{\mathfrak{h}_{\Gamma}/W_{C}}\mathfrak{h}_{\Gamma})
\] 
and
\[
e_C \mathcal{A}_{\Gamma} e_C \otimes_{\mathbb{H}_{\Gamma}} \mathbb{H}_l\text{-}\modcat  \cong {^0}\Per(\mathcal{Y}_l/\mathcal{X}_l).
\]
That is, there is an isomorphism of algebras $\mathcal{A}_l \cong e_C \mathcal{A}_{\Gamma} e_C \otimes_{\mathbb{H}_{\Gamma}} \mathbb{H}_l$.
\end{proposition}
\begin{proof}
Following the proof of \cite[Theorem 4.6]{KIWY} the idempotent $e_C \in \mathcal{A}_{\Gamma}$ corresponding to the 0 vertex and $\bullet$ curve in $\Gamma_C$ defines the algebra $e_C \mathcal{A}_{\Gamma} e_C$ corresponding to contracting all $\circ$ curves and as such $e_C \mathcal{A}_{\Gamma} e_C$-$\modcat \cong {^0}\Per(\mathcal{Y}_{\Gamma_C} \times_{\mathfrak{h}_{\Gamma}/W_{C}}\mathfrak{h}_{\Gamma}).$ Then $e_C\mathcal{A}_{\Gamma}e_C$ is a $\mathbb{H}_{\Gamma}$-algebra, and there is a $W_{C}$ action on $\mathbb{H}_{\Gamma}$ such that $\mathbb{H}_{\Gamma}^{W_{C}} \cong \mathbb{H}_l$ and $\mathbb{H}_{\Gamma}$ is a flat $\mathbb{H}_l$ module. In particular, since ${^0}\Per(Y/X)$ commutes with flat base change in $X$ as noted in Remark \ref{Remark:Flat Base Change} and $\mathcal{Y}_{\Gamma_C} \cong \mathcal{Y}_l$ it follows that 
\[
e_C \mathcal{A}_{\Gamma} e_C \text{-mod} \cong {^0}\Per(\mathcal{Y}_{\Gamma_C} \times_{\mathfrak{h}_{\Gamma}/W_{C}}\mathfrak{h}_{\Gamma}) \cong \mathcal{A}_l \otimes_{\mathbb{H}_l}\mathbb{H}_{\Gamma} \text{-mod}
\]
and hence 
\[
e_C \mathcal{A}_{\Gamma} e_C \otimes_{\mathbb{H}_{\Gamma}} \mathbb{H}_l \text{-mod} \cong \mathcal{A}_l \text{-mod} \cong {^0}\Per(\mathcal{Y}_l/\mathcal{X}_l)
\]
as
$
\mathcal{A}_l \otimes_{\mathbb{H}_l}\mathbb{H}_{\Gamma} \otimes_{\mathbb{H}_{\Gamma}} \mathbb{H}_l \cong \mathcal{A}_l.
$ By transferring the basic projective generators across this abelian equivalence we can deduce that 
\[
\mathcal{A}_l \cong e_C\mathcal{A}_{\Gamma}e_C \otimes_{\mathbb{H}_{\Gamma}} \mathbb{H}_l.
\]
\end{proof}

This discussion allows us to produce the explicit presentations of Theorem \ref{Theorem: Explicit presentations}.
\subsection{Proof of Theorem \ref{Theorem: Explicit presentations}} \label{Section:Proof of Explicit}

In order to prove Theorem \ref{Theorem: Explicit presentations}, in the next 6 examples we calculate a presentation of the universal flopping algebra $\mathcal{A}_{l}$ for each of the cases occurring in the classification of 3-fold flops. For each length $l$ with diagram $\Gamma_C$ we record:
\begin{enumerate}
\item The coloured Dynkin diagram $\Gamma_C$.
\item The Cartan polynomial algebra $\mathbb{H}_{\Gamma}$.
\item The Weyl subgroup $W_C$.
\item The invariant algebra $\mathbb{H}_l=\mathbb{H}^{W_C}_{\Gamma}$.
\item The algebra $e_C \mathcal{A}_{\Gamma} e_C$.
\item The universal flopping algebra $\mathcal{A}_l$.
\end{enumerate}

The presentation of the universal flopping algebra is produced by first realising a presentation of the algebra $e_C \mathcal{A}_{\Gamma} e_C$ as a path algebra over $\mathbb{H}_\Gamma$ of a quiver with relations by using the presentation of $\mathcal{A}_{\Gamma}$ given in Section \ref{sec:Simultaneous Presentations} immediately above. In each case there is a clear homomorphism from the proposed path algebra with relations to $e_C \mathcal{A}_{\Gamma} e_C$. Then both these are flat over $\mathbb{H}_l$  and the homomorphism restricts to an isomorphism over the central fibre, where both algebras are isomorphic to the algebra $A_{\Gamma_C}$ considered in Section \ref{Section: Partial resolutions presentations}. This implies the homomorphism is in fact an isomorphism by Nakayama's lemma (for graded modules) as we can view both algebras as graded $\mathbb{H}_{\Gamma}$ modules and restrict to the unique graded maximal ideal of $\mathbb{H}_{\Gamma}$. Then, having found $e_C \mathcal{A}_{\Gamma} e_C$, we are then able to use Proposition \ref{Prop:Basechange} to realise the algebra $\mathcal{A}_l$ by change of base ring.

\begin{example}(Length 1)
The coloured Dynkin diagram $\Gamma_C$ associated to length 1 is
\[
\begin{tikzpicture}[node distance=1cm, main node/.style={circle,fill=black!100,draw,font=\sffamily\Large\bfseries},bend angle=5, looseness=1]
 \node[main node] (1) at (0,0) {};
\end{tikzpicture}
\]
with associated Cartan polynomial algebra
\[
\mathbb{H}_{\Gamma}:=\frac{\mathbb{C}[t_0,t_1]}{(t_0+t_1)} \cong \mathbb{C}[t].
\]
In this case the universal flopping algebra is actually isomorphic to the algebra $\mathcal{A}_\Gamma$ of Proposition \ref{Prop:Deformation Preprojective}, $e_C=1$, and $\mathcal{A}_1$ can be presented as the path algebra over $\mathbb{H}_{\Gamma}$ of the following quiver with relations
\begin{align*}
\begin{aligned}
\begin{tikzpicture} [bend angle=15, looseness=1]
\node (C0) at (0,0)  {$0$};
\node (C1) at (2,0)  {$1$};
\draw [->,bend left=35] (C0) to node[gap]  {\scriptsize{$a_0$}} (C1);
\draw [->,bend left] (C0) to node[gap]  {\scriptsize{$a_1^*$}} (C1);
\draw [->,bend left] (C1) to node[gap]  {\scriptsize{$a_0^*$}} (C0);
\draw [->,bend left=35] (C1) to node[gap]  {\scriptsize{$a_1$}} (C0);
\end{tikzpicture}
\end{aligned}
& \quad \quad
\begin{aligned}
&a_0a_0^*-a_1^*a_1=t e_0, \\
&a_1a_1^*-a_0^*a_0=-te_1.
\end{aligned}
\end{align*}
\end{example}

The Weyl group $W_{\Gamma} \cong S_2$ is generated by the simple reflection $s_1$ that maps $t$ to $-t$. The subgroup $W_C$ is trivial, so $\mathbb{H}_1:=\mathbb{H}_{\Gamma} \cong \mathbb{C}[t]$.

\begin{example}(Length 2) \label{Example:length 2} The coloured Dynkin diagram $\Gamma_C$ associated to length 2 is
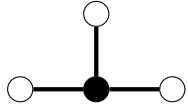
\[
\begin{tikzpicture}[node distance=1cm, main node/.style={circle,fill=black!100,draw,font=\sffamily\Large\bfseries}]
  \node[main node] (1) at (0,0) {};
  \node[circle, draw] (2) at (-1,0) {};
  \node[circle, draw] (3) at (1,0) {};
  \node[circle, draw] (4)  at (0,1) {};
 \draw [ultra thick] (1) to node {} (2);
 \draw [ultra thick] (1) to node {} (3);
 \draw [ultra thick] (1) to node {} (4);
\end{tikzpicture}
\]
with associated Cartan polynomial algebra
\[
\mathbb{H}_{\Gamma}:= \frac{\mathbb{C}[t_0,t_1,t_2,t_3,t_4]}{(t_0+t_1+t_2+t_3+2t_4)}.
\]

The action of the generators of the Weyl group $W_{\Gamma}$ on $\mathbb{H}_{\Gamma}$ is given by the simple reflections $s_1, \,s_2, \, s_3, \, s_4$:
\[
s_i(t_j)= \left\{ \begin{array}{cc} -t_i & \text{if $i=j$}, \\
t_j & \text{otherwise}.
\end{array} \right.
\] The subgroup $W_{C} \cong S_2 \times S_2 \times S_2$ is generated by the three elements $s_1,s_2,s_3$. Hence the elements
\[
t:=t_0, \quad
T_0^{\beta}:=t_1^2/4, \quad
T_0^{\gamma}:=t_2^2/4, \quad \text{and }
T_0^{\delta}:=t_3^2/4 
\]
generate the invariant polynomial ring \[
\mathbb{H}_2:=\mathbb{C}[t,T_0^{\beta},T_0^{\gamma},T_0^{\delta}].
\]
Choosing the idempotent $e_C:= e_0 + e_4$, the algebra $e_C \mathcal{A}_{\Gamma} e_C$ can be presented as 
\begin{align*}
\begin{aligned}
\begin{tikzpicture} [bend angle=15, looseness=1]
\node (C0) at (0,0)  {$0$};
\node (C1) at (2,0)  {$4$};
\draw [->,bend left] (C0) to node[above]  {\scriptsize{$a_0$}} (C1);
\draw [->,bend left] (C1) to node[below]  {\scriptsize{$a_0^*$}} (C0);
\draw [->, looseness=24, in=52, out=128,loop] (C1) to node[above] {$\scriptstyle{b}$} (C1);
\draw [->, looseness=24, in=-38, out=38,loop] (C1) to node[right] {$\scriptstyle{d}$} (C1);
\draw [->, looseness=24, in=-128, out=-52,loop] (C1) to node[below] {$\scriptstyle{c}$} (C1);
\end{tikzpicture}
\end{aligned}
& \quad \quad
\begin{aligned}
&a_0a_0^*=t_0e_0 & &b(b-t_1)=0 \\ 
&c(c-t_2)=0  & &d(d-t_3)=0\\ 
&a_0^*a_0+b+c + d +t_4 e_4=0 \\ 
\end{aligned}
\end{align*}
where $b=a_1^*a_1, \, c=a_2^* a_2$, and $d= a_3^* a_3$. The change of basis
\[
\beta = b- t_1/2, \quad
\gamma = c- t_1/2, \quad \text{and }
\delta = d- t_1/2 
\]
rewrites the quiver and relations as
\begin{align*}
\begin{aligned}
\begin{tikzpicture} [bend angle=15, looseness=1]
\node (C0) at (0,0)  {$0$};
\node (C1) at (2,0)  {$4$};
\draw [->,bend left] (C0) to node[above]  {\scriptsize{$a_0$}} (C1);
\draw [->,bend left] (C1) to node[below]  {\scriptsize{$a_0^*$}} (C0);
\draw [->, looseness=24, in=52, out=128,loop] (C1) to node[above] {$\scriptstyle{\beta}$} (C1);
\draw [->, looseness=24, in=-38, out=38,loop] (C1) to node[right] {$\scriptstyle{\delta}$} (C1);
\draw [->, looseness=24, in=-128, out=-52,loop] (C1) to node[below] {$\scriptstyle{\gamma}$} (C1);
\end{tikzpicture}
\end{aligned}
& \quad \quad
\begin{aligned}
&a_0a_0^*=te_0, & & \beta^2=T_0^{\beta} e_4, \\ 
&\gamma^2=T_0^{\gamma} e_4, & & \delta^2=T_0^{\delta}e_4, \\ 
&a_0^*a_0+\beta+\gamma + \delta = \frac{t}{2} \, e_4. \\ 
\end{aligned}
\end{align*}
This presents the algebra $\mathcal{A}_2:=(e_C \mathcal{A}_{\Gamma} e_C)\otimes_{\mathbb{H}_{\Gamma}} \mathbb{H}_2$ as  claimed in Theorem \ref{Theorem: Explicit presentations}.
\end{example}

\begin{example}(Length 3)
The coloured Dynkin diagram $\Gamma_C$ associated to length 3 is
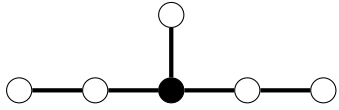
\[
\begin{tikzpicture}[node distance=1cm, main node/.style={circle,fill=black!100,draw,font=\sffamily\Large\bfseries}]
  \node[main node] (1) at (0,0) {};
  \node[circle, draw] (2) at (0,1) {};
  \node[circle, draw] (3) at (-1,0) {};
  \node[circle, draw] (3') at (-2,0) {};
  \node[circle, draw] (4)  at (1,0) {};
    \node[circle, draw] (4')  at (2,0) {};
 \draw [ultra thick] (1) to node {} (2);
 \draw [ultra thick] (1) to node {} (3);
 \draw [ultra thick] (1) to node {} (4);
  \draw [ultra thick] (4) to node {} (4');
   \draw [ultra thick] (3) to node {} (3');
\end{tikzpicture}
\]
with associated Cartan polynomial algebra
\[
\mathbb{H}_{\Gamma}:=\frac{\mathbb{C}[t_0,t_1,t_2,t_3,t_4,t_5,t_6]}{(t_0+2t_1+t_2+2t_3+t_4+2t_5+3t_6)}.
\]

The Weyl group $W_{\Gamma}$ is generated by the simple reflections $s_i$ for $1 \le i \le 6$. The Weyl subgroup $W_{C} \cong S_2 \times S_3 \times S_3$ is generated by $S_2\cong \{s_1 \}, \, S_{3}\cong \{s_2,s_3\}, \,S_3 \cong \{s_4, s_5\}$.
The action of $W_{C} \cong S_2 \times S_3 \times S_3$ fixes $t:=2t_0+t_1$ and acts by the permutation representation on the set
\[
\{ \tau_1^{\beta}, \tau_2^{\beta}, \tau_3^{\beta} \} \times \{\tau_1^{\gamma}, \tau_2^{\gamma}, \tau_3^{\gamma}  \}\times  \{\tau_1^{\delta}, \tau_2^{\delta} \} 
\] where
\small
\begin{align*}
\begin{aligned}
\tau_1^{\beta}&=(t_2+2t_3)/3,  \\
\tau_2^{\beta}&=(t_2-t_3)/3, \\
\tau_3^{\beta}&=-(2t_2+t_3)/3, 
\end{aligned}
\qquad
\begin{aligned}
\tau_1^{\gamma}&=(t_4+2t_5)/3,  \\
\tau_2^{\gamma}&=(t_4-t_5)/3, \\
\tau_3^{\gamma}&=-(2t_4+t_5)/3,
\end{aligned}
\qquad
\begin{aligned}
\tau_1^{\delta}&=t_1/2,  \\
\tau_2^{\delta}&=-t_1/2. 
\end{aligned}
\end{align*}
\normalsize

Let $\sigma_i(x_1, \dots, x_n)$ denote the degree $i$ elementary symmetric polynomial in the elements $x_1, \dots, x_n$, and then we choose the following set of generators for the invariant algebra $\mathbb{H}^{W_C}_{\Gamma}$:
\small
\begin{align*}
\begin{aligned}
t & :=(2t_0+t_1)/2, \\
T_0^{\delta} &:= -\sigma_2(\tau_1^{\delta}, \tau_2^{\delta}),\\ 
\end{aligned}
\qquad
\begin{aligned}
T^\beta_1 &:=- \sigma_2(\tau_1^{\beta},\tau_2^{\beta},\tau_3^{\beta}), \\ 
T^\beta_0  &:=- \sigma_3(\tau_1^{\beta},\tau_2^{\beta},\tau_3^{\beta}), 
\end{aligned}
\qquad
\begin{aligned}
T^\gamma_1 &:=-\sigma_2(\tau_1^{\gamma},\tau_2^{\gamma},\tau_3^{\gamma}), \\
T^\gamma_0 &:= -\sigma_3(\tau_1^{\gamma},\tau_2^{\gamma},\tau_3^{\gamma}).
\end{aligned}
\end{align*}
\normalsize
Hence
\[
\mathbb{H}_3:= \mathbb{H}^{W_{C}}_{\Gamma} \cong \mathbb{C}[t,T_0^\beta,T_1^\beta,T_0^\gamma,T_1^\gamma,T_0^{\delta}].
\]

Take the idempotent $e_C:=e_0+e_6$. The algebra $e_C \mathcal{A}_\Gamma e_C$ can be presented as the following path algebra with relations over $\mathbb{H}_{\Gamma}$
\begin{align*}
\begin{aligned}
\begin{tikzpicture} [bend angle=15, looseness=1]
\node (C0) at (0,0)  {$0$};
\node (C1) at (2,0)  {$6$};
\draw [->,bend left] (C0) to node[above]  {\scriptsize{$a$}} (C1);
\draw [->,bend left] (C1) to node[below]  {\scriptsize{$a^*$}} (C0);
\draw [->, looseness=24, in=52, out=128,loop] (C1) to node[above] {$\scriptstyle{b}$} (C1);
\draw [->, looseness=24, in=-38, out=38,loop] (C1) to node[right] {$\scriptstyle{d}$} (C1);
\draw [->, looseness=24, in=-128, out=-52,loop] (C1) to node[below] {$\scriptstyle{c}$} (C1);
\end{tikzpicture}
\end{aligned}
& \quad \quad
\begin{aligned}
& aa^*=t_0(t_0 +t_1)e_0, & & a^*a=d ( d - t_1), \\ 
& \delta a^*=a^* (t_0 +t_1), & & a d = (t_0 +t_1)a,  \\ 
&b(b-t_3)(b-(t_2+t_3))=0, & & \\
& c(c-t_5)(c-(t_4+t_5))=0, & &\\ 
&b+c + d + t_6e_6=0. \\ 
\end{aligned}
\end{align*}
where
\[
a=a_0a_1, \quad a^*=a_1^* a_0^*, \quad b=a_3^*a_3, \quad c=a_5^*a_5, \quad \text{and} \quad d=a_1^*a_1.
\]
The base change
\[
\beta := b-(2t_3+t_2)/3, \quad
\gamma := c-(2t_5+t_4)/3, \quad
\delta := d -t_1/2 
\]
rewrites the quiver and relations as 
\begin{align*}
\begin{aligned}
\begin{tikzpicture} [bend angle=15, looseness=1]
\node (C0) at (0,0)  {$0$};
\node (C1) at (2,0)  {$6$};
\draw [->,bend left] (C0) to node[above]  {\scriptsize{$a$}} (C1);
\draw [->,bend left] (C1) to node[below]  {\scriptsize{$a^*$}} (C0);
\draw [->, looseness=24, in=52, out=128,loop] (C1) to node[above] {$\scriptstyle{\beta}$} (C1);
\draw [->, looseness=24, in=-38, out=38,loop] (C1) to node[right] {$\scriptstyle{\delta}$} (C1);
\draw [->, looseness=24, in=-128, out=-52,loop] (C1) to node[below] {$\scriptstyle{\gamma}$} (C1);
\end{tikzpicture}
\end{aligned}
\qquad
\begin{aligned}
&aa^*=(t^2 -T_0^{\delta})e_0, & & \\ 
& a^*a=\delta^2 - T_0^{\delta}e_4, & & \\ 
& \delta a^*=a^* t, & & a \delta = t a, \\ 
&\beta^3- T^{\beta}_1 \beta - T^{\beta}_0e_4=0, & &\gamma^3-T^{\gamma}_1 \gamma - T^{\gamma}_0e_4=0, \\ 
&\beta+\gamma + \delta = \frac{t}{3} \, e_4. \\ 
\end{aligned}
\end{align*}
which presents it as an algebra over $\mathbb{H}_3$.
This is the presentation of $\mathcal{A}_3:=(e_C \mathcal{A}_{\Gamma} e_C)\otimes_{\mathbb{H}} \mathbb{H}_3$ required in Theorem \ref{Theorem: Explicit presentations}.

\end{example}

\begin{example}(Length 4) The coloured Dynkin diagram $\Gamma_C$ associated to length 4 is
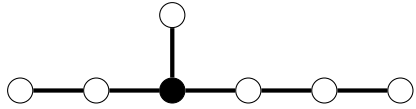
\[
\begin{tikzpicture}[node distance=1cm, main node/.style={circle,fill=black!100,draw,font=\sffamily\Large\bfseries}]
  \node[main node] (1) at (0,0) {};
  \node[circle, draw] (2) at (0,1) {};
  \node[circle, draw] (3) at (1,0) {};
  \node[circle, draw] (3') at (2,0) {};
    \node[circle, draw] (3'') at (3,0) {};
  \node[circle, draw] (4)  at (-1,0) {};
    \node[circle, draw] (4')  at (-2,0) {};
 \draw [ultra thick] (1) to node {} (2);
 \draw [ultra thick] (1) to node {} (3);
 \draw [ultra thick] (1) to node {} (4);
  \draw [ultra thick] (4) to node {} (4');
   \draw [ultra thick] (3) to node {} (3');
    \draw [ultra thick] (3') to node {} (3'');
\end{tikzpicture}
\]
with associated Cartan polynomial algebra
\[
\mathbb{H}_{\Gamma}:=\frac{\mathbb{C}[t_0,t_1,t_2,t_3,t_4,t_5,t_6,t_7]}{(t_0+2t_1+4t_2+3t_3+2t_4+3t_5+2t_6+ t_7)}.
\]

The subgroup $W_{C} \cong S_3 \times S_2 \times S_4$ is generated by $\{s_1,s_2\} \times \{s_3 \} \times \{s_4,s_5,s_6 \}$. Then $W_{C}$ fixes the element
\[
3t_0+2t_1+t_2
\] and acts by the permutation representation on 
\[
 \{ \tau_1^{\beta}, \tau_2^{\beta} \} \times \{\tau_1^{\gamma}, \tau_2^{\gamma}, \tau_3^{\gamma}, \tau_4^{\gamma} \} \times \{\tau_1^{\delta}, \tau_2^{\delta}, \tau_3^{\delta} \}
\] where
\small
\begin{align*}
\begin{aligned}
\tau_1^{\beta}&=(t_3)/2,  \\
\tau_2^{\beta}&=-(t_3)/2, 
\end{aligned}
\qquad
\begin{aligned}
\tau_1^{\gamma}&=(t_4+2t_5+3t_6)/4,  \\
\tau_2^{\gamma}&=(t_4+2t_5-t_6)/4, \\
\tau_3^{\gamma}&=( t_4-2t_5-t_6)/4, \\
\tau_4^{\gamma}&=-(3t_4+2t_5+t_6)/4, 
\end{aligned}
\qquad
\begin{aligned}
\tau_1^{\delta}&=(t_1+2t_2)/3, \\
\tau_2^{\delta}&=(t_1-t_2)/3, \\
\tau_3^{\delta}&=-(2t_1+t_2)/3.
\end{aligned}
\end{align*}
\normalsize

Let $\sigma_i(x_1, \dots, x_n)$ denote the $i^{th}$ elementary symmetric polynomial in the elements $x_1, \dots, x_n$, and then we choose the following set of generators for the invariant algebra $\mathbb{H}^{W_C}_{\Gamma}$:
\small
\begin{align*}
\begin{aligned}
t & :=(3t_0+2t_1+t_2)/3, \\
T^\beta_0  &:=- \sigma_2(\tau_1^{\beta},\tau_2^{\beta}), 
\end{aligned}
\qquad
\begin{aligned}
T^\gamma_2 &:=-\sigma_2(\tau_1^{\gamma},\tau_2^{\gamma},\tau_3^{\gamma},\tau_4^{\gamma}),\\
T^\gamma_1 &:=-\sigma_3(\tau_1^{\gamma},\tau_2^{\gamma},\tau_3^{\gamma},\tau_4^{\gamma}), \\
T^\gamma_0 &:= -\sigma_4(\tau_1^{\gamma},\tau_2^{\gamma},\tau_3^{\gamma},\tau_4^{\gamma}),
\end{aligned}
\qquad
\begin{aligned}
T_0^{\delta} &:= -\sigma_3(\tau_1^{\delta},\tau_2^{\delta},\tau_3^{\delta}),\\ 
T_1^{\delta} &:= -\sigma_2(\tau_1^{\delta},\tau_2^{\delta},\tau_3^{\delta}).\\ 
\end{aligned}
\end{align*}
\normalsize

Hence 
\[
\mathbb{H}_4 \cong \mathbb{C}[t,T_0^{\delta},T_1^{\delta},T^\beta_0,T^\gamma_2,T^\gamma_1,T^\gamma_0].
\]

Take $e_C=e_0+e_7$, then the algebra $e_C \mathcal{A}_{\Gamma} e_C$ can be presented as the $\mathbb{H}_{\Gamma}$ path algebra with relations
\begin{align*}
\begin{aligned}
\begin{tikzpicture} [bend angle=15, looseness=1]
\node (C0) at (0,0)  {$0$};
\node (C1) at (2,0)  {$7$};
\draw [->,bend left] (C0) to node[above]  {\scriptsize{$a$}} (C1);
\draw [->,bend left] (C1) to node[below]  {\scriptsize{$a^*$}} (C0);
\draw [->, looseness=24, in=52, out=128,loop] (C1) to node[above] {$\scriptstyle{b}$} (C1);
\draw [->, looseness=24, in=-38, out=38,loop] (C1) to node[right] {$\scriptstyle{d}$} (C1);
\draw [->, looseness=24, in=-128, out=-52,loop] (C1) to node[below] {$\scriptstyle{c}$} (C1);
\end{tikzpicture}
\end{aligned}
& \quad \quad
\begin{aligned}
& aa^*=t_0(t_0 + t_1)(t_0+(t_1+t_2)), \\
& a^*a=d(d-t_2)(d-(t_1+t_2)),  \\
& d a^*=a^*( t_0 +(t_1+t_2)),  \qquad a d = (t_0 + (t_1+t_2)),  \\ 
&b(b-t_3)=0, \\
& c(c-t_6)(c-(t_6+t_5))(c-(t_6+t_5+t_4))=0,\\ 
&b+c + d + t_7e_7=0. 
\end{aligned}
\end{align*}
where
\[
a=a_0a_1a_2, \quad a^*=a_2^*a_1^* a_0^*, \quad b=a_3^*a_3, \quad c=a_6^*a_6, \quad \text{and} \quad d=a_2^*a_2.
\]
The base change 
\[
\delta := d -(t_1+2t_2)/3, \quad
\beta := b-(t_3)/2, \quad \text{ and }
\gamma := c-(t_4+2t_5+3t_6)/4.
\]
rewrites the quiver and relations as
\begin{align*}
\begin{aligned}
\begin{tikzpicture} [bend angle=15, looseness=1]
\node (C0) at (0,0)  {$0$};
\node (C1) at (2,0)  {$7$};
\draw [->,bend left] (C0) to node[above]  {\scriptsize{$a$}} (C1);
\draw [->,bend left] (C1) to node[below]  {\scriptsize{$a^*$}} (C0);
\draw [->, looseness=24, in=52, out=128,loop] (C1) to node[above] {$\scriptstyle{\beta}$} (C1);
\draw [->, looseness=24, in=-38, out=38,loop] (C1) to node[right] {$\scriptstyle{\delta}$} (C1);
\draw [->, looseness=24, in=-128, out=-52,loop] (C1) to node[below] {$\scriptstyle{\gamma}$} (C1);
\end{tikzpicture}
\end{aligned}
\qquad
\begin{aligned}
& \delta a^*=a^* t, & & a \delta = ta, \\ 
&aa^*=(t^3-T^{\delta}_1 t - T^{\delta}_0) e_0, & &  a^*a=\delta^3- T^{\delta}_1 \delta  - T^{\delta}_0 e_7, \\ 
&\beta^2=T^\beta_0, & &\gamma^4=T^{\gamma}_2 \gamma^2 + T^{\gamma}_1 \gamma + T^{\gamma}_0, \\ 
&\beta+\gamma +\delta = \frac{t}{4} \, e_7. \\ 
\end{aligned}
\end{align*}
This is the required presentation of $\mathcal{A}_4:=(e_C \mathcal{A}_{\Gamma} e_C)\otimes_{\mathbb{H}_{\Gamma}} \mathbb{H}_4$ as a path algebra with relations over $\mathbb{H}_4$ from Theorem \ref{Theorem: Explicit presentations}.
\end{example}

\begin{example}(Length 5) The coloured Dynkin diagram $\Gamma_C$ associated to length 5 is
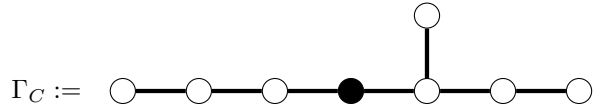
\[
\begin{tikzpicture}[node distance=1cm, main node/.style={circle,fill=black!100,draw,font=\sffamily\Large\bfseries}]
 \node (Label) at (-5,0) {$\Gamma_C:=$};
  \node[circle, draw] (1) at (0,0) {};
  \node[circle, draw] (2) at (0,1) {};
  \node[main node] (3) at (-1,0) {};
  \node[circle, draw] (3') at (-2,0) {};
    \node[circle, draw] (3'') at (-3,0) {};
        \node[circle, draw] (3''') at (-4,0) {};
  \node[circle, draw] (4)  at (1,0) {};
    \node[circle, draw] (4')  at (2,0) {};
 \draw [ultra thick] (1) to node {} (2);
 \draw [ultra thick] (1) to node {} (3);
 \draw [ultra thick] (1) to node {} (4);
  \draw [ultra thick] (4) to node {} (4');
   \draw [ultra thick] (3) to node {} (3');
    \draw [ultra thick] (3') to node {} (3'');
        \draw [ultra thick] (3'') to node {} (3''');
\end{tikzpicture}
\]
with associated Cartan polynomial algebra
\[
\mathbb{H}_{\Gamma}:=\frac{\mathbb{C}[t_0,t_1,t_2,t_3,t_4,t_5,t_6,t_7,t_8]}{(t_0+2t_1+3t_2+4t_3+5t_4+3t_5+2t_6+4t_7+6t_8)}.
\]

The subgroup $W_{C} \cong S_4 \times S_5$ is generated by the simple reflections $\{s_1, s_2,s_3 \} \times \{s_5,s_6,s_7,s_8\}$. In particular the action fixes the element $(4t_0+3t_1+2t_2+t_3)$ and acts by its permutation representation on
\[
\{ \tau_1^{\delta}, \tau_2^{\delta}, \tau_3^{\delta}, \tau_4^{\delta} \} \times \{ \tau_1, \tau_2, \tau_3, \tau_4, \tau_5 \}
\]
where
\small
\begin{align*}
\begin{aligned}
\tau_1^{\delta}&:=(t_1+2t_2+3t_3)/4, \\ 
\tau_2^{\delta}&:= (t_1+2t_2-t_3)/4, \\ 
\tau_3^{\delta}&:=(t_1-2t_2-t_3)/4, \\ 
\tau_4^{\delta}&:= -(3t_1+2t_2+t_3)/4,
\end{aligned}
\qquad
\begin{aligned}
\tau_1 &:= (t_5+2t_8+3t_7+4t_6)/5, \\ 
\tau_2 &:= (t_5+2t_8+3t_7-t_6)/5, \\ 
\tau_3 &:= (t_5+2t_8-2t_7-t_6)/5, \\ 
\tau_4 &:= (t_5-3t_8-2t_7-t_6)/5, \\ 
\tau_5 &:= (-4t_5-3t_8-2t_7-t_6)/5.
\end{aligned}
\normalsize
\end{align*}
Hence the following elementary symmetric polynomials are a generating set for the invariant ring
\small
\begin{align*}
\begin{aligned}
&t:= (4t_0+3t_1+2t_2+t_3)/4, \\
&T^{\delta}_0:= -\sigma_4(\tau_1^{\delta},\tau_2^{\delta},\tau_3^{\delta},\tau_4^{\delta}), \\
&T^{\delta}_1:=  -\sigma_3(\tau_1^{\delta},\tau_2^{\delta},\tau_3^{\delta},\tau_4^{\delta}), \\
&T^{\delta}_2:= - \sigma_2(\tau_1^{\delta},\tau_2^{\delta},\tau_3^{\delta},\tau_4^{\delta}),
\end{aligned}
\qquad
\begin{aligned}
&T_3:=- \sigma_2(\tau_1, \tau_2,\tau_3,\tau_4,\tau_5), \\
&T_2:=- \sigma_3(\tau_1, \tau_2,\tau_3,\tau_4,\tau_5), \\
&T_1:=-\sigma_4(\tau_1, \tau_2,\tau_3,\tau_4,\tau_5),  \\
&T_0:= - \sigma_5(\tau_1, \tau_2,\tau_3,\tau_4,\tau_5).
\end{aligned}
\end{align*}
\normalsize

In order to calculate the universal flopping algebra of length 5 we first consider the intermediate case corresponding to the diagram
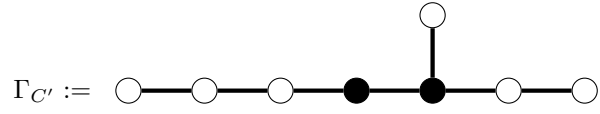
\[
\begin{tikzpicture}[node distance=1cm, main node/.style={circle,fill=black!100,draw,font=\sffamily\Large\bfseries}]
 \node (Label) at (-5,0) {$\Gamma_{C'}:=$};
  \node[main node] (1) at (0,0) {};
  \node[circle,draw] (2) at (0,1) {};
  \node[main node] (3) at (-1,0) {};
  \node[circle,draw] (3') at (-2,0) {};
    \node[circle,draw] (3'') at (-3,0) {};
        \node[circle,draw] (3''') at (-4,0) {};
  \node[circle,draw] (4)  at (1,0) {};
    \node[circle,draw] (4')  at (2,0) {};
 \draw [ultra thick] (1) to node {} (2);
 \draw [ultra thick] (1) to node {} (3);
 \draw [ultra thick] (1) to node {} (4);
  \draw [ultra thick] (4) to node {} (4');
   \draw [ultra thick] (3) to node {} (3');
    \draw [ultra thick] (3') to node {} (3'');
        \draw [ultra thick] (3'') to node {} (3''');
\end{tikzpicture}
\]
and the subgroup $S_4 \times S_3 \times S_2  < S_4 \times S_5 $. The subgroup $S_3 \times S_2 < S_5$ is generated by the simple reflections $\{s_6,s_7\} \times  \{ s_5 \} $. Then this acts by it's natural permutation representation on the set $\{ \tau_1, \tau_2, \tau_3 \} \times \{ \tau_4, \tau_5 \}$.  

The elements 
\[
t:=(4t_0+3t_1+2t_2+t_3)/5  \text{  and  }  R_1:=(3t_5+6t_8+4t_7+2t_6)/5=-(t_0+2t_1+3t_2+4t_3+5t_4)/5
\]
are invariant under the $S_2 \times S_3$ action, as are the elementary symmetric polynomials
\small
\begin{align*}
\begin{aligned}
R_1&:= \tau_1+\tau_2+\tau_3, \\
R_2&:= \tau_1 \tau_2+\tau_2 \tau_3+ \tau_3 \tau_1, \\
R_3&:= \tau_1 \tau_2 \tau_3, 
\end{aligned}
\qquad
\begin{aligned}
S_1&:= \tau_4+\tau_5, \\
S_2&:= \tau_4 \tau_5,
\end{aligned}
\end{align*}
\normalsize
The elementary symmetric polynomials for the $S_5$ action can be realised as
\small
\begin{align*}
T_3&=R_2+S_2+R_1S_1 =R_2+S_2-R_1^2\\
T_2&=R_3+R_2S_1+R_1S_2 =R_3-R_1R_2+R_1S_2\\
T_1&=-R_3R_1+R_2S_2 \\
T_0&=R_3S_2.
\end{align*}
\normalsize

Let $e_{C'}:=e_0+e_4+e_8$. Then the algebra $e_{C'} \mathcal{A}_{\Gamma} e_{C'}$ can be presented as the path algebra over $\mathbb{H}_{\Gamma}$ of a quiver with relations as follows:
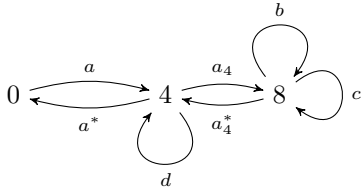
\begin{align*}
\begin{aligned}
\begin{tikzpicture} [bend angle=15, looseness=1]
\node (C0) at (0,0)  {$0$};
\node (C1) at (2,0)  {$4$};
\node (C2) at (3.5,0)  {$8$};
\draw [->,bend left] (C0) to node[above]  {\scriptsize{$a$}} (C1);
\draw [->,bend left] (C1) to node[below]  {\scriptsize{$a^*$}} (C0);
\draw [->,bend left] (C1) to node[above]  {\scriptsize{$a_4$}} (C2);
\draw [->,bend left] (C2) to node[below]  {\scriptsize{$a_4^*$}} (C1);
\draw [->, looseness=24, in=52, out=128,loop] (C2) to node[above] {$\scriptstyle{b}$} (C2);
\draw [->, looseness=24, in=-38, out=38,loop] (C2) to node[right] {$\scriptstyle{c}$} (C2);
\draw [->, looseness=24, in=-128, out=-52,loop] (C1) to node[below] {$\scriptstyle{d}$} (C1);
\end{tikzpicture}
\end{aligned}
& \quad \quad
\begin{aligned}
&aa^*=t_0(t_0+t_1)(t_0+t_1+t_2)(t_0+t_1+t_2+t_3)e_0 \\
&a^*a=d(d-t_3)(d-t_3-t_2)(d-t_3-t_2-t_1) \\
&t_0 a = a(d - t_1-t_2-t_3) \\
&da^* = a^*(t_0+t_1+t_2+t_3) \\
&a_4a_4^*-d=t_4e_4 \\
&b(b-t_5)=0 \\
&c(c-t_7)(c-t_6-t_7) = 0 \\
&a_4^*a_4 + b + c + t_8e_8 =0 
\end{aligned}
\end{align*}
where
\[
a:=a_0a_1a_2a_3, \quad a^*:=a_3^*a_2^*a_1^*a_0^*, \quad d:=a_3^*a_3, \quad b:=a_5^* a_5, \text{ and }  c:=a_7^*a_7.
\]
The base change 
\[
\delta = d- \frac{t_1+2t_2+3t_3}{4}, \quad
B = b + \tau_3, \quad \text{ and }
C = c + \tau_5
\]
rewrites the relations as
\begin{align*}
\begin{aligned}
&t a = a \delta,
 \qquad \delta  a^* = a^* t, \\
&aa^*=(t^4 - T^{\delta}_2t^2 - T_1^{\delta} t -T_0^{\delta})e_0, \qquad a^*a=\delta ^4 - T^{\delta}_2\delta ^2 - T_1^{\delta} \delta  -T_0^{\delta}e_4,  \\ 
&a_4a_4^*-\delta =\frac{t_1+2t_2+3t_3+4t_4}{4}e_4=-R_1-\frac{t}{5}e_4, \\
&a_4^*a_4 +B + C +R_1= 0,  \qquad B^2+R_1B+S_2=0, \\
&C^3-R_1C^2+R_2C-R_3 = 0.
\end{aligned}
\end{align*}
Let $e_C=e_0+e_4 \in e_{C'} \mathcal{A}_{\Gamma} e_{C'} \subset \mathcal{A}_\Gamma$, and then $e_C \mathcal{A}_{\Gamma} e_C$ can be presented as the path algebra over $\mathbb{H}_\Gamma$ of the following  quiver with relations.

\begin{align*}
\begin{aligned}
\begin{tikzpicture} [bend angle=15, looseness=1]
\node (C0) at (0,0)  {$0$};
\node (C1) at (2,0)  {$4$};
\draw [->,bend left] (C0) to node[above]  {\scriptsize{$a$}} (C1);
\draw [->,bend left] (C1) to node[below]  {\scriptsize{$a^*$}} (C0);
\draw [->, looseness=24, in=52, out=128,loop] (C1) to node[above] {$\scriptstyle{\delta}$} (C1);
\draw [->, looseness=24, in=-38, out=38,loop] (C1) to node[right] {$\scriptstyle{\gamma}$} (C1);
\draw [->, looseness=24, in=-128, out=-52,loop] (C1) to node[below] {$\scriptstyle{\beta}$} (C1);
\end{tikzpicture}
\end{aligned}
& \quad \quad
\begin{aligned}
&aa^*=(t^4 - T^{\delta}_2t^2 - T_1^{\delta} t -T_0^{\delta})e_0, \\
&a^*a=\delta^4 - T^{\delta}_2\delta^2 - T_1^{\delta} \delta -T_0^{\delta}e_4, \\
&ta = a\delta, \\
&\delta a^* = a^*t, \\
&\beta-\delta+\frac{t}{5}e_4=0 ,
\end{aligned}
\end{align*}
along with the two relations 
\[
\gamma \beta \gamma +\gamma^2 \beta +\gamma \beta^3=(R_1^2 - R_2 - S_2)\gamma \beta + (R_1R_2 - R_1S_2 - R_3) \gamma - R_3S_2e_4
\]
and 
\[
(\gamma+ \beta^2)^2+\beta \gamma \beta = (R_1^2 - R_2 - S_2)(\gamma+\beta^2)  + (R_1R_2 - R_1S_2 - R_3)\beta + (R_1R_3 -
    R_2S_2)e_4.
\]
where
\[
\beta := a_4a_4^*+R_1e_4 \quad \, \text{ and } \, \quad \gamma:=a_4 B a_4^*-S_2e_4.
\]
Then this is a presentation of the algebra $\mathcal{A}_5$ as a path algebra over $\mathbb{H}_5$ for the quiver and relations claimed in Theorem \ref{Theorem: Explicit presentations}.
\begin{align*}
\begin{aligned}
\begin{tikzpicture} [bend angle=15, looseness=1]
\node (C0) at (0,0)  {$0$};
\node (C1) at (2,0)  {$4$};
\draw [->,bend left] (C0) to node[above]  {\scriptsize{$a$}} (C1);
\draw [->,bend left] (C1) to node[below]  {\scriptsize{$a^*$}} (C0);
\draw [->, looseness=24, in=52, out=128,loop] (C1) to node[above] {$\scriptstyle{\delta}$} (C1);
\draw [->, looseness=24, in=-38, out=38,loop] (C1) to node[right] {$\scriptstyle{\gamma}$} (C1);
\draw [->, looseness=24, in=-128, out=-52,loop] (C1) to node[below] {$\scriptstyle{\beta}$} (C1);
\end{tikzpicture}
\end{aligned}
& \quad \quad
\begin{aligned}
& a\delta=ta, \qquad \delta a^* = a^*t, \\
&aa^*=(t^4 - T^{\delta}_2t^2 - T_1^{\delta} t -T_0^{\delta})e_0, \\
&a^*a=\delta^4 - T^{\delta}_2\delta^2 - T_1^{\delta} \delta -T_0^{\delta}e_4, \\
&\beta-\delta+\frac{t}{5}e_4=0, \\
&\gamma \beta \gamma +\gamma^2 \beta +\gamma \beta^3=-T_3\gamma \beta -T_2 \gamma -T_0e_4 \\
&
(\gamma+ \beta^2)^2+\beta \gamma \beta = -T_3(\gamma+\beta^2)  -T_2\beta -T_1e_4.
\end{aligned}
\end{align*}

\end{example}

\begin{example}(Length 6) The coloured Dynkin diagram $\Gamma_C$ associated to length 6 is
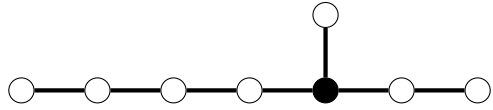
\[
\begin{tikzpicture}[node distance=1cm, main node/.style={circle,fill=black!100,draw,font=\sffamily\Large\bfseries}]
  \node[main node] (1) at (0,0) {};
  \node[circle, draw] (2) at (0,1) {};
  \node[circle, draw] (3) at (-1,0) {};
  \node[circle, draw] (3') at (-2,0) {};
    \node[circle, draw] (3'') at (-3,0) {};
        \node[circle, draw] (3''') at (-4,0) {};
  \node[circle, draw] (4)  at (1,0) {};
    \node[circle, draw] (4')  at (2,0) {};
 \draw [ultra thick] (1) to node {} (2);
 \draw [ultra thick] (1) to node {} (3);
 \draw [ultra thick] (1) to node {} (4);
  \draw [ultra thick] (4) to node {} (4');
   \draw [ultra thick] (3) to node {} (3');
    \draw [ultra thick] (3') to node {} (3'');
        \draw [ultra thick] (3'') to node {} (3''');
\end{tikzpicture}
\]
with associated Cartan polynomial algebra
\[
\mathbb{H}_{\Gamma}:=\frac{\mathbb{C}[t_0,t_1,t_2,t_3,t_4,t_5,t_6,t_7,t_8]}{(t_0+2t_1+3t_2+4t_3+5t_4+3t_5+2t_6+4t_7+6t_8)}.
\]

The subgroup $W_C \cong S_5 \times S_2 \times S_3$ is generated by the elements $\{ s_1,s_2,s_3,s_4 \} \times \{ s_5 \} \times \{ s_6, s_7 \}$. This subgroup fixes the element
\[
(5t_0+4t_1+3t_2+2t_3+t_4)
\]
and acts by the permutation representation on
\[
\{\tau^{\delta}_1,\tau^{\delta}_2,\tau^{\delta}_3,\tau^{\delta}_4,\tau^{\delta}_5 \} \times \{\tau_1^{\beta}, \tau_2^{\beta}\} \times \{\tau^{\gamma}_1,\tau^{\gamma}_2,\tau^{\gamma}_3 \}
\]
where
\small
\begin{align*}
\begin{aligned}
\tau^{\delta}_1&:=(t_1+2t_2+3t_3+4t_4)/5, \\
\tau^{\delta}_2&:=(t_1+2t_2+3t_3-t_4)/5, \\
\tau^{\delta}_3&:=(t_1+2t_2-2t_3-t_4)/5, \\
\tau^{\delta}_4&:=(t_1-3t_2-2t_3-t_4)/5, \\
\tau^{\delta}_5 &:=-(4t_1+3t_2+2t_3+t_4)/5,
\end{aligned}
\qquad
\begin{aligned}
\tau^{\beta}_1&:=(t_5)/2, \\
\tau^{\beta}_2&:=-(t_5)/2,
\end{aligned}
\qquad 
\begin{aligned}
\tau^{\gamma}_1&:=(t_6+2t_7)/3, \\
\tau^{\gamma}_2&:=(t_6-t_7)/3, \\
\tau^{\gamma}_3&:=-(2t_6+t_7)/3.
\end{aligned}
\end{align*}
\normalsize
Hence the following elementary symmetric polynomials generate the invariant ring
\small
\begin{align*}
\begin{aligned}
t & :=(5t_0+4t_1+3t_2+2t_3+t_4)/5, \\
T^{\delta}_0 & := -\sigma_5(\tau^{\delta}_1, 
\tau^{\delta}_2, 
\tau^{\delta}_3, 
\tau^{\delta}_4, 
\tau^{\delta}_5 ), \\
T^{\delta}_1&:=-\sigma_4(\tau^{\delta}_1, 
\tau^{\delta}_2, 
\tau^{\delta}_3, 
\tau^{\delta}_4, 
\tau^{\delta}_5 ), \\
T^{\delta}_2&:=-\sigma_3(\tau^{\delta}_1, 
\tau^{\delta}_2, 
\tau^{\delta}_3, 
\tau^{\delta}_4, 
\tau^{\delta}_5 ), 
\end{aligned}
\qquad
\begin{aligned}
T^{\delta}_3&:= -\sigma_2(\tau^{\delta}_1, 
\tau^{\delta}_2, 
\tau^{\delta}_3, 
\tau^{\delta}_4, 
\tau^{\delta}_5 ),\\
T^\gamma_1 &:=-\sigma_3(\tau^{\gamma}_1,
\tau^{\gamma}_2,
\tau^{\gamma}_3), \\
T^\gamma_0 &:= -\sigma_2(\tau^{\gamma}_1,
\tau^{\gamma}_2,
\tau^{\gamma}_3)/27, \\
T^{\beta}_0&:=- \sigma_2(\tau^{\beta}_1,
\tau^{\beta}_2)/4,
\end{aligned}
\end{align*}
\normalsize
and
\[
\mathbb{H}_6:=\mathbb{C}[t,T_{\delta}^0,T_{\delta}^1,T_{\delta}^2,T_{\delta}^3,T_{\gamma}^0,T_{\gamma}^1,T_{\beta}^0].
\]

Let $e_C:=e_0+e_8$. Then the algebra $e_C \mathcal{A}_\Gamma e_C$ can be presented as 
\begin{align*} 
\begin{aligned}
\begin{tikzpicture} [bend angle=15, looseness=1]
\node (C0) at (0,0)  {$0$};
\node (C1) at (2,0)  {$8$};
\draw [->,bend left] (C0) to node[above]  {\scriptsize{$a$}} (C1);
\draw [->,bend left] (C1) to node[below]  {\scriptsize{$a^*$}} (C0);
\draw [->, looseness=24, in=52, out=128,loop] (C1) to node[above] {$\scriptstyle{b}$} (C1);
\draw [->, looseness=24, in=-38, out=38,loop] (C1) to node[right] {$\scriptstyle{d}$} (C1);
\draw [->, looseness=24, in=-128, out=-52,loop] (C1) to node[below] {$\scriptstyle{c}$} (C1);
\end{tikzpicture}
\end{aligned}
 \quad \quad
\begin{aligned}
& d a^*=a^* (t_0 (+t_1+t_2+t_3+t_4 )),  \\
& a  d = (t_0+(t_1+t_2+t_3+t_4)) a, \\ 
&b(b-t_5)=0, & &  \\
&c(c-t_7)(c-(t_7+t_6))=0, &&\\ 
&b+c+ d + t_8e_5=0, 
\end{aligned}
\\
\begin{aligned}
& aa^*=t_0(t_0 + t_1)(t_0+(t_1+t_2))(t_0+(t_1+t_2+t_3))(t_0+(t_1+t_2+t_3+t_4))e_0, & &\\ 
& a^*a=d(d-t_4)(d-(t_3+t_4))(d-(t_2+t_3+t_4))(d-(t_1+t_2+t_3+t_4)), && \\ 
\end{aligned}
\end{align*}
where $a=a_0a_1a_2a_3a_4,\, a^*=a_4^*a_3^*a_2^*a_1^*a_0^*, \, b=a_5^*a_5,c=a_7^*a_7$ and $d=a_4^*a_4$. Then the change of basis
\[
\delta := d -(t_1+2t_2+3t_3+4t_4)/5, \quad
\beta := b-(t_5)/2, \quad \text{ and } 
\gamma := c-(2t_7+t_6)/3 
\]
rewrites the quiver and relations as 
\begin{align*}
\begin{aligned}
\begin{tikzpicture} [bend angle=15, looseness=1]
\node (C0) at (0,0)  {$0$};
\node (C1) at (2,0)  {$8$};
\draw [->,bend left] (C0) to node[above]  {\scriptsize{$a$}} (C1);
\draw [->,bend left] (C1) to node[below]  {\scriptsize{$a^*$}} (C0);
\draw [->, looseness=24, in=52, out=128,loop] (C1) to node[above] {$\scriptstyle{\beta}$} (C1);
\draw [->, looseness=24, in=-38, out=38,loop] (C1) to node[right] {$\scriptstyle{\delta}$} (C1);
\draw [->, looseness=24, in=-128, out=-52,loop] (C1) to node[below] {$\scriptstyle{\gamma}$} (C1);
\end{tikzpicture}
\end{aligned}
& \quad \quad
\begin{aligned}
& \delta a^*=a^* t, \qquad a \delta = t a, \\ 
&aa^*=(t^5-T_d^3t^3-T_d^2t^2-t T_d^1 -T_d^0)e_0,  \\ 
& a^*a=\delta^5 -T_d^3\delta^3-T_d^2\delta^2-\delta T_d^1 -T_d^0 \\ 
&\beta^2=T_{\beta}^0, \qquad \gamma^3=T_{\gamma}^1 \gamma + T_{\gamma}^0, \\ 
&\beta+\gamma + \delta = \frac{t}{6} \, e_8 \\ 
\end{aligned}
\end{align*}
which is the presentation of $\mathcal{A}_6$ as a path algebra with relations over $\mathbb{H}_6$ required by Theorem \ref{Theorem: Explicit presentations}.
\end{example}

\section{Examples and applications} \label{Sec: Examples} In the following section we show how the presentations of the universal flopping algebras can be used to construct explicit examples relating to simple threefold flops. In Example \ref{Example: Length 2} we reproduce the matrix factorisation description of the universal flop of length $2$ calculated by Curto and Morrison, and we also recover the presentation of Laufer flop of length two as a quiver with superpotential previously calculated by Aspinwall and Morrison \cite{AspinwallMorrison} in Example \ref{Ex:Laufer flop}. After this we calculate particular examples of NCCRs associated to simple threefold flops of length $3, \, 4, \, 5$ and $6$ presented as quivers with superpotentials in Examples \ref{Example:Length 3,4, and 6} and \ref{Example: Length 5 explicit}.

\begin{example}(Length $2$) \label{Example: Length 2}
Using the presentation of the length 2 universal flopping algebra $\mathcal{A}_2$ in Example \ref{Example:length 2} we calculate the algebra $\mathcal{R}_2$ and the $\mathcal{R}_2$-modules $\mathcal{N}_2$ and $\mathcal{N}_2^{+}$. As such we explicitly recover Curto and Morrison's universal flop of length $2$ from the universal flopping algebra $\mathcal{A}_2$. We recall $\mathcal{A}_2:=\mathbb{H}_2 Q_2/I_2$ for the quiver and relations
\begin{align*}
\begin{aligned}
\begin{tikzpicture} [bend angle=15, looseness=1]
\node (C0) at (0,0)  {$0$};
\node (C1) at (2,0)  {$4$};
\draw [->,bend left] (C0) to node[above]  {\scriptsize{$a$}} (C1);
\draw [->,bend left] (C1) to node[below]  {\scriptsize{$a^*$}} (C0);
\draw [->, looseness=24, in=52, out=128,loop] (C1) to node[above] {$\scriptstyle{\beta}$} (C1);
\draw [->, looseness=24, in=-38, out=38,loop] (C1) to node[right] {$\scriptstyle{\gamma}$} (C1);
\draw [->, looseness=24, in=-128, out=-52,loop] (C1) to node[below] {$\scriptstyle{\delta}$} (C1);
\end{tikzpicture}
\end{aligned}
& \quad \quad
\begin{aligned}
&aa^*=te_0, & & \beta^2=T^\beta_0 e_4, \\ 
&\gamma^2=T^\gamma_0 e_4, & & \delta^2=T^\delta_0 e_4, \\ 
&a^*a+\beta+\gamma + \delta = \frac{t}{2} \, e_4 \\ 
\end{aligned}
\end{align*}
where $ \mathbb{H}_2:= \mathbb{C}[t,T^\beta_0,T^\gamma_0,T^\delta_0]$.

We recall that $\mathcal{R}_2=e_0 \mathcal{A}_2 e_0$ is the commutative algebra generated by the loops at vertex $0$, and it is flat as an $\mathbb{H}_2$-algebra with fibre the Kleinian singularity $R_{\Gamma}$ above the origin. This implies that $\mathcal{R}_2$ has the same Hilbert polynomial as $\mathbb{H}_2 \otimes_{\mathbb{C}} R_{\Gamma}$, where we recall the grading on $R_{\Gamma}$ from Remark \ref{Remark:Group Quotient}. In particular, it can be presented as a polynomial ring with 7 generators satisfying a single equation. Explicitly
\[
x'=a\beta \gamma a^* , \, \, y= a \gamma a^*, \, \, \text{ and } z=a \beta a^*
\]
are a generating set for $\mathcal{R}_2$ as a $\mathbb{H}_2$ algebra, and hence  \[
x',\, y,\, z, \, t_a:=2te_0, \, t_b:=T^{\beta}_0e_0, \, t_c:=T^{\gamma}_0 e_0, \, t_d:= T^{\delta}_0 e_0
\] are a set of generators for $\mathcal{R}_2:=e_0 \mathcal{A}_2 e_0$ as a $\mathbb{C}$-algebra in degree $4,4,6,2,4,4,4$ respectively.

It can be verified that these generators satisfy the following relation in degree 12
\[
(x')^2+2t_a(y+z-t_d+t_c+t_b+t_a^2)x' = (y+z-t_d+t_c+t_b+t_a^2)yz- 4t_a^2t_bt_c  + t_by^2 + t_cz^2
\]
either by hand or by using any appropriate symbolic algebra calculator; the difficulty is in finding the equation not verifying it. See Appendix \ref{Appendix:Algebra} for more details on how the hypersurface equation can be found.

The following change of basis
\begin{align} \label{Equation:Change of Basis}
x &:= x' + t_a (z + y + t_b + t_c - t_d + t_a^2) & u &:= -t_b & t &:= 2t_a \\
v &:= -(z + y + t_b + t_c - t_d + t_a^2)/2 & w &:= -t_c  &  \nonumber
\end{align}
rewrites the degree 12 equation above in the simpler form
\[
f:=- x^2 - u y^2 - 2v y z - w z^2-( u w -v^2  )t^2.
\]
Then $\mathcal{R}_2=\mathbb{C}[s,y,z,u,w,v,t]/(f)$, and this is equal to the equation found by Morrison and Curto (presented in the introduction as Equation \ref{Equation:Morrison+Curto}).

The MCM $\mathcal{R}_2$-modules $\mathcal{N}_2:=e_0 \mathcal{A}_2 e_4$ and $\mathcal{N}_2^{+}:=e_4\mathcal{A}_2 e_0$ can be explicitly recovered from $\mathcal{A}_2$ in a similar way.  These $\mathcal{R}_2$-modules correspond to paths in $\mathcal{A}_2$ for vertex 0 to vertex 4 or from vertex 4 to vertex 0 respectively. They are finitely generated as $\mathcal{R}_2$-modules, and it is clear to see that $a,a \beta,a \gamma, a\beta \gamma$ are a generating set for $e_0 \mathcal{A}_2(1-e_0)$ as a (left) $\mathcal{R}_2$-module and $a^*,\beta a^*,\gamma a^*,\beta \gamma a^*$ are a generating set for $(1-e_0)\mathcal{A}_2 e_0$ as a (right) $\mathcal{R}_2$-module. By calculating the action of the generators of $\mathcal{R}_2$ on the generators of $\mathcal{N}_2$ and $\mathcal{N}_2^{+}$ these modules can be presented as cokernels of maps between free $\mathcal{R}_2$-modules, see Appendix \ref{Appendix:MF} for how such presentations can be calculated explicitly.

In particular, after the same change of basis as above (\ref{Equation:Change of Basis}), we can present $\mathcal{N}_2$ as the cokernel of the matrix
\[
\Psi_2:=\begin{pmatrix}
  -x-tv & y& -z& t \\
  -uy - 2vz& -x+tv& tu& z \\
  wz& -tw& -x-tv& y \\
  -tuw& -wz& -uy - 2vz&-x+ tv
\end{pmatrix},
\]
which, up to a change in minus sign for the third element of the basis, is half of the matrix factorisation found by Morrison and Curto and listed above as equation \ref{Equation:MatrixFactorisation}. Similarly $\mathcal{N}_2^+$ can be presented as the cokernel of the matrix
\[
\Psi_2^+:=\begin{pmatrix}
-x+tv&-y  & z &-t\\
2vz+uy&-x-tv& -tu& -z \\
-wz & tw & -x +tv & -y \\
tuw &  wz & 2vz + uy & -x-tv
\end{pmatrix}.
\]
Together these give a matrix factorisation of the equation 
\[
-f=x^2 + u y^2 + 2v y z + w z^2+( u w -v^2  )t^2.
\]
Further, in the language of Aspinwall and Morrison \cite[Section 4]{AspinwallMorrison} we can write the elements $a,\, a^*,\,\beta, \, \gamma$ and $\delta$ as $\mathcal{R}$-linear maps between the cokernels defined above:
\begin{align*}
&a=
\begin{pmatrix}
1 & 0 & 0 & 0
\end{pmatrix}, \,
a^*=
\begin{pmatrix}
t \\ z \\ y \\ x+tv
\end{pmatrix}, \, 
\beta=\begin{pmatrix}
0 & 1 & 0 & 0 \\
-u & 0 & 0  & 0 \\
2v & 0 & 0 & -1 \\
0 & 2v & u & 0
\end{pmatrix}, \,
\\
&\gamma=\begin{pmatrix}
0 & 0 & 1 & 0 \\
0 & 0 & 0 & 1 \\
-w & 0 & 0 & 0 \\
0 &-w & 0 & 0
\end{pmatrix}, \, \text{ and }
\delta=
 \begin{pmatrix}
-t/2 & -1& -1 & 0 \\
u-z & t/2 & 0  & -1 \\
w-2v-y& 0 & t/2 & 1 \\
0 & w-2v-y & -u+z & -t/2
\end{pmatrix}.
\end{align*}
These can be calculated in a similar manner to the cokernel presentation producing the matrix factorisation by finding the action of an element $\alpha \in \mathcal{A}_l$ on the generating set for $\mathcal{R}_2 \oplus \mathcal{N}_2$ as an $\mathcal{R}_2$ module. We note that the cokernel is defined up to addition by any row of $\Psi_2$, and so these matrices are only determined up to such alterations. For example, any row of these linear maps can be altered in this way and by looking at the top row of $\Psi_2$ we can write
\[
(x+tv,0,0,0)=(0,y,-z,t)
\]
and hence the following elements are in the same equivalence class
\[
a^*a= \begin{pmatrix} 
t & 0 & 0 & 0 \\ 
z & 0 & 0 & 0 \\
y & 0 & 0 & 0 \\
x+tv & 0 & 0 & 0
\end{pmatrix}
\cong 
\begin{pmatrix} 
t & 0 & 0 & 0 \\ 
z & 0 & 0 & 0 \\
y & 0 & 0 & 0 \\
0 & y & -z & t
\end{pmatrix}.
\]
Noting this, it is straightforward to verify that the relations on the quiver are satisfied. It is also easy to present the endomorphisms $xe_4,  y e_4, z e_4, t e_4, u e_4, v e_4, w e_4$ of $\mathcal{N}_2$ in terms of the maps $\beta, \gamma, \delta, a, a^*$:
\begin{align*}
&u e_4 \cong -\beta^2, \quad w e_4 \cong -\gamma^2,  \quad t e_4 \cong 2(\beta+\gamma+\delta + a^*a), \\
&z e_4 \cong  a^* a \beta + \beta a^* a - t \beta, \quad
y e_4 \cong a^* a \gamma + \gamma a^* a  - t \gamma,
\\
&v e_4 \cong  -\frac{1}{2}(t^2/4+y+z)e_4-\frac{1}{2}(\beta^2+\gamma^2-\delta^2), \\
&xe_4 \cong \beta \gamma a^*a + a^* a \beta \gamma - t \beta \gamma- 2 v a^* a +tve_4.
\end{align*}

The associated moduli space of 0-generated representations of dimension vector $(1,2)$ can also be explicitly calculated. A representation is defined by assigning vector spaces $V_0 \cong \mathbb{C}$ and $V_1 \cong \mathbb{C}^2$ to the vertices and maps 
\[
(a) \in \Mat(V_0,V_1), \quad (a^*) \in \Mat(V_1,V_0), \quad (\beta) \in \Mat(V_1,V_1), \quad (\gamma) \in \Mat(V_1,V_1)
\]
to the arrows such that the relations are satisfied. Two such representations are isomorphic if they lie in the same $\GL(V_1)$ orbit under base change. The 0-generated stability condition requires that the image of $V_0$ under the maps spans $V_1$. That is, we can assume that $(a)=(1,0)$ and $((\beta)_{1,2},(\gamma)_{1,2}) \neq (0,0)$. This produces two charts, $U_0$ where $(\beta)_{1,2} \neq 0$ and $U_1$ where $(\gamma)_{1,2} \neq 0$ that we now calculate

If $(\beta)_{1,2} \neq 0$ each $\GL(V_1)$ orbit contains one representative of the form 
\small
\[
(a)=\begin{pmatrix}
1,0
\end{pmatrix}, \;
(a^*)=\begin{pmatrix}
a_0^* \\
a_1^*
\end{pmatrix}, \;
(\beta)=\begin{pmatrix}
0 & 1  \\
b_{10} & b_{11} \\
\end{pmatrix}, \;
(\gamma)=\begin{pmatrix}
c_{00} & c_{01}  \\
c_{10} & c_{11} \\
\end{pmatrix}
, \;
(\delta)=\begin{pmatrix}
d_{00} & d_{01}  \\
d_{10} & d_{11} \\
\end{pmatrix}
\]
\normalsize
and after solving for the relations we find that the corresponding affine chart is
\[
U_0=\Spec \frac{\mathbb{H}_l[c_{00}, c_{10}, c_{01}, d_{10}]}{(T_0^\gamma - c_{00}^2 - c_{01} c_{10}, \, \, T_0^{\delta} - c_{00}^2 + d_{10} + c_{01} d_{10} - c_{00} t - \frac{1}{4}t^2)}
\]
which is a polynomial ring over $\mathbb{C}$ in 6 variables, and the map to $\mathcal{R}_2$ is defined by 
\begin{align*}
x' &= c_{00}( c_{10} + d_{10} +T_0^\beta) + c_{10} t , \\
 y &= -c_{01}( c_{10} + d_{10} +T_0^\beta) + c_{00} t ,\\
 z &= -(c_{10} + d_{10} + T_0^\beta).
\end{align*}
This allows the chart to be written explicitly in the form
$
U_0=\Spec \mathbb{C}[z,t,T_0^\beta,c_{00},c_{10},c_{01}]
$
where 
\begin{align*}
x'&=c_{00}z+c_{10}t, \quad y=-c_{01}z+c_{00}t, \quad T^\gamma_0=c_{00}^2+c_{01}c_{10}, \quad \text{ and} \\ 
T^\delta_0&=(1+c_{01})(z+c_{10}+T^{\beta}_0)+\frac{1}{4}t^2+c_{00}t+c_{00}^2.
\end{align*}
These can be compared with equations (44)-(48) in \cite{CurtoMorrison} defining an affine chart in the Grassman blowup, which are identical after the base changes outline in equation \ref{Equation:Change of Basis} with the additional Grassman blowup coordinates $\alpha_{1,2}$ and $\alpha_{2,2}$ corresponding to $c_{01}$ and $-c_{00}$ respectively.

Similarly, if $(\gamma)_{1,2} \neq 0$ each $\GL(V_1)$ orbit contains a representative of the form 
\small
\[
(a)=\begin{pmatrix}
1,0
\end{pmatrix}, \;
(a^*)=\begin{pmatrix}
A_0^* \\
A_1^*
\end{pmatrix}, \;
(\beta)=\begin{pmatrix}
B_{00} & B_{01}  \\
B_{10} & B_{11} \\
\end{pmatrix}, \;
(\gamma)=\begin{pmatrix}
0 & 1  \\
C_{10} & C_{11} \\
\end{pmatrix}, \;
(\delta)=\begin{pmatrix}
D_{00} & D_{01}  \\
D_{10} & D_{11} \\
\end{pmatrix}
\]
\normalsize
and after solving for the relations of the quiver, this defines the affine chart
\[
U_1=\Spec  \frac{\mathbb{H}_2[B_{00},B_{01},B_{10}, D_{10}]}{
(
T_0^\beta-B_{00}^2 - B_{01} B_{10}, \, \,
  T_0^\delta - B_{00}^2 + D_{10} + B_{01} D_{10} - B_{00} t -\frac{1}{4} t^2 )}
\]
with map to $\mathcal{R}_2$ defined by
\begin{align*}
x' &= -B_{00} (B_{10}+ D_{10} + T_0^\gamma) + B_{01} t T_0^\gamma  \\
y &= -(B_{10} + D_{10} + T_0^\gamma),   \\
z &= -B_{01}( B_{10} + D_{10}+T_0^\gamma) + B_{00} t .
\end{align*}

Transition maps between the two charts can be calculated by comparing the generators in the case both $b_{01}$ and $C_{01}$ nonzero. Using the $\GL(V_1)$ action to change the basis between the two chosen forms for the representations we find the change of basis
\begin{align*}
c_{00}&=-B_{00}/B_{01},  &  c_{01}&= 1/B_{01},  \\ 
c_{10}&=-B_{00}^2/B_{01}+B_{01}T_0^\gamma,  &  d_{10}&=-B_{00} t -B_{00}^2 + B_{00}^2/B_{01} + B_{01} D_{10}.
\end{align*}
\end{example}
\vspace{0.8cm}

A calculation similar to that in Example \ref{Example: Length 2} can calculate the algebra $\mathcal{R}_l$  and modules $\mathcal{N}_l$, and $\mathcal{N}_l^{+}$ as matrix factorisations from the universal flopping algebra $\mathcal{A}_l$ for lengths $\ge 2$. However, this relies on computer algebra packages to compute the ever larger equations and it becomes more difficult to efficiently display results so we do not do so here.

Instead, we now list explicit examples of threefold flops that can be recovered from universal flopping algebras.

\begin{example}(Laufer flop) \label{Ex:Laufer flop} This example shows how an explicit example of a simple threefold flop can be recovered from the universal flopping algebra. The following classifying map to the  universal flop of length two 
\[
 \mathcal{R}_2 \rightarrow R \cong \mathcal{R}_2/(w+t,u-y,v)
\] is considered in \cite[Section 4.2]{AspinwallMorrison}. The singularity $R:=\mathbb{C}[x,y,z,t]/(g)$ is a hypersurface defined by the equation 
\[
g:= x^2 + y^3  -tz^2-yt^3.
\]
It has a unique, isolated singular point at the origin.

We calculate the corresponding NCCR $A:=\mathcal{A}_2 \otimes_{\mathcal{R}_2} R$ from the the universal quiver, reproducing the calculation of Aspinwall and Morrison of quiver and superpotential for this particular flop of length 2. This can presented as the algebra over $\mathbb{C}[t]$ of the following quiver with relations
\begin{align*}
\begin{aligned}
\begin{tikzpicture} [bend angle=15, looseness=1]
\node (C0) at (0,0)  {$0$};
\node (C1) at (2,0)  {$4$};
\draw [->,bend left] (C0) to node[above]  {\scriptsize{$a$}} (C1);
\draw [->,bend left] (C1) to node[below]  {\scriptsize{$a^*$}} (C0);
\draw [->, looseness=24, in=52, out=128,loop] (C1) to node[above] {$\scriptstyle{\beta}$} (C1);
\draw [->, looseness=24, in=-38, out=38,loop] (C1) to node[right] {$\scriptstyle{\delta}$} (C1);
\draw [->, looseness=24, in=-128, out=-52,loop] (C1) to node[below] {$\scriptstyle{\gamma}$} (C1);
\end{tikzpicture}
\end{aligned}
& \quad \quad
\begin{aligned}
&aa^*=te_0,\qquad \beta^2= -y e_4, \\ 
&\gamma^2=t e_4,\qquad \delta^2 =(t^2/4+t+z) e_4,\\ 
&a^*a+\beta+\gamma + \delta = \frac{t}{2}e_4  
\end{aligned}\end{align*}
where
\[
y e_4= a^* a \gamma + \gamma a^*a-t \gamma \quad \text{ and } \quad z e_4= a^* a \beta + \beta a^* a - t \beta.
\]
This presentation is realised by flat base change from the universal flop of length two calculated above.

Using the linear relations we can eliminate $t e_0$, $t e_4$, $ye_4$, and $z e_4$, recalling that $t,y,z$ are central elements, to find the relations
\begin{align*}
& \gamma^2 \beta = \beta \gamma^2 , \quad \gamma^2 \delta = \delta \gamma^2,  \quad aa^*a=a \gamma^2, \quad  a^* a a^* = \gamma^2 a^*, \\
&\beta^2=  -a^* a \gamma - \gamma a^*a+ \gamma^3, \quad a^*a+\beta+\gamma + \delta - \gamma^2/2=0,  \\ 
&  \delta^2 =\gamma^4/4+\gamma^2+a^* a \beta + \beta a^* a -\gamma^2 \beta 
\end{align*}
among the remaining generators. The generator $\delta$ can also be eliminated which, after some simplification, leaves the relations
\begin{align*}
& \gamma^2 \beta = \beta \gamma^2, \quad aa^*a=a \gamma^2, \quad a^* a a^* = \gamma^2 a^*, \\
&\beta^2=  -a^* a \gamma - \gamma a^*a+ \gamma^3, \quad \text{and} \quad \gamma^3-\beta^2-\beta \gamma - \gamma \beta - \gamma a^* a- a^*a \gamma=0.
\end{align*}
These can be further simplified, the relation $\gamma^2 \beta= \beta \gamma^2$ can be eliminated, and the resulting quiver with relations can be presented as
\begin{align*}
\begin{aligned}
\begin{tikzpicture} [bend angle=15, looseness=1]
\node (C0) at (0,0)  {$0$};
\node (C1) at (2,0)  {$4$};
\draw [->,bend left] (C0) to node[above]  {\scriptsize{$a$}} (C1);
\draw [->,bend left] (C1) to node[below]  {\scriptsize{$a^*$}} (C0);
\draw [->, looseness=24, in=52, out=128,loop] (C1) to node[above] {$\scriptstyle{\beta}$} (C1);
\draw [->, looseness=24, in=-128, out=-52,loop] (C1) to node[below] {$\scriptstyle{\gamma}$} (C1);
\end{tikzpicture}
\end{aligned}
& \quad \quad
\begin{aligned}
&aa^*a =a \gamma^2  , \\
&a^* a a^* = \gamma^2 a^* , \\
&\beta^2=\gamma^3-a^* a \gamma - \gamma a^* a , \\
&\beta \gamma+ \gamma \beta=0.
\end{aligned}\end{align*}
These relations are defined by the superpotential
\[
\frac{1}{2}a a^* a a^*-a \gamma^2 a^*- \gamma \beta^2 + \frac{1}{4}\gamma^4,
\]
i.e. the 4 relations are produced by the cyclic derivatives of the superpotential. This superpotential was calculated in \cite[Equation (56)]{AspinwallMorrison}.

The contraction algebra can be calculated from $A$ by setting $e_0=0$, and from this the Gopakumar-Vafa invariants can be recovered. The contraction algebra can be presented as
\[
A_{con}:=A/Ae_0A \cong \frac{\mathbb{C} \langle \beta, \gamma \rangle}{\langle \gamma^3-\beta^2, \beta \gamma+ \gamma \beta \rangle}
\]
and it can be calculated that
\[
\dim A_{con}= 9 \quad \text{ and } \quad \dim \frac{A_{con}}{[A_{con},A_{con}]}=5
\]
See \cite[Example 1.3]{WemyssDonovan}, and so $n_1=5$ and $9=5+4n_2$ by Proposition \ref{Prop:GKV}. Hence $n_1=5$ and $n_2=1$. 

\end{example}

We now consider several examples in length $>2$. In each case we produce a NCCR corresponding to a small resolution of a Gorenstein threefold singularity contracting an irreducible rational curve and present the algebra as the path algebra of a quiver with relations.

\begin{example}(An explicit flop of length $3$) \label{example:length 3 example}
We consider the following classifying map on the base of the universal flop of length 3
\[
\mathbb{H}_3 \rightarrow H\cong \mathbb{C}[T]:= \mathbb{H}_3[T]/(t,t_1^b, t_1^c, T-t_0^b, T-t_0^c, T-t_0^d).
\]
using the change of basis
\[
x= w-1/2(z^2+T y)
\]
the generators of $R':=\mathcal{R}_3 \otimes_{\mathbb{H}_2} H$ satisfy the single nonhomogeneous equation 

\[
-x^2-T^5 + 4T^3y + T^2z^2 + \frac{1}{4}T^2y^2 + \frac{1}{2}Tz^2y + \frac{1}{4}z^4 - y^3
\]
with matrix factorisation $(x Id -C)(x Id + C)$ where $C$ is the 6 by 6 matrix 
\small
\begin{align*}  & \left( \begin{matrix}
  \frac{1}{2}(Ty - z^2) & -T^2& T^2 - y&\\
  -Ty + zy& -\frac{1}{2}(Ty +z^2)& -T^2 + Tz& \\
  -T^3 - T^2z& T^3 - Ty + zy& -\frac{1}{2}(Ty - z^2)& \\
  -T^3 + y^2& Ty& Tz + Ty& \\
  -T^2z - T^2y& T^3 + T^2z - y^2& -T^3 - Ty &\\
  -T^3z - T^2y + Tzy - z^2y& T^4 - T^2z - 2T^2y + Tzy& T^3 - y^2&
  \end{matrix} \right. \\
  & \hspace{5cm}
  \left. \begin{matrix}
 & &-y& -z& -T\\
 & & T^2& y& -z \\
 & & Tz& -T^2& y \\
 & & -\frac{1}{2}(Ty + z^2)& -T^2& -y \\
 &&T^3 - Ty + zy& \frac{1}{2}(Ty + z^2)& T^2 \\
 &&-T^3& -T^3 + Ty - zy & \frac{1}{2}(Ty + z^2)
  \end{matrix}
  \right)
\end{align*}
\normalsize
found using the methods outlined in Appendices \ref{Appendix:Algebra} and \ref{Appendix:MF}, and in particular Example \ref{Example:Length 3}. 

This defines the algebra $R':= \mathbb{C}[T,x,y,z]/(g)$ and a variety $X':=\Spec R'$. This variety is singular with an isolated singular point at the origin, alongside other isolated singular points. Taking the completion of $R'$ with respect to the maximal ideal at the origin $\mathfrak{m}=(T,x,y,w)$ produces a complete local isolated singularity $R$. By base change from the explicit presentation of $\mathcal{A}_3$, the NCCR $A'$ of $R'$ can be presented as the path algebra over $\mathbb{C}[T]$ of the following quiver with relations.
\begin{align*}
\begin{aligned}
\begin{tikzpicture} [bend angle=15, looseness=1]
\node (C0) at (0,0)  {$0$};
\node (C1) at (2,0)  {$6$};
\draw [->,bend left] (C0) to node[above]  {\scriptsize{$a$}} (C1);
\draw [->,bend left] (C1) to node[below]  {\scriptsize{$a^*$}} (C0);
\draw [->, looseness=24, in=52, out=128,loop] (C1) to node[above] {$\scriptstyle{\beta}$} (C1);
\draw [->, looseness=24, in=-38, out=38,loop] (C1) to node[right] {$\scriptstyle{\delta}$} (C1);
\draw [->, looseness=24, in=-128, out=-52,loop] (C1) to node[below] {$\scriptstyle{\gamma}$} (C1);
\end{tikzpicture}
\end{aligned}
& \quad \quad
\begin{aligned}
&aa^*= -Te_0, \\ 
& a^*a=\delta^2 - Te_6, \\ 
& \delta a^*=0, \quad  a \delta = 0,
\end{aligned}
\quad
\begin{aligned}
&\beta^3=  Te_6, \\ 
&\gamma^3= Te_6, \\ 
&\beta+\gamma + \delta = 0. 
\end{aligned}
\end{align*}
The elements $Te_0, T e_6$ and $\delta$ can be eliminated from the presentation, and this reduces the relations to the following collection
\begin{align*}
\begin{aligned}
&aa^*a=-a \beta^3, \quad 
a^*aa^*=- \beta^3 a^*, \quad
 \beta^3=\gamma^3, \\
& a^*a=(\beta+\gamma)^2 -\beta^3,  \quad
 (\beta+ \gamma) a^*=0, \quad
 a (\beta + \gamma)= 0.
\end{aligned}
\end{align*}
The relations are redundant in this collection, and after simplification the algebra $A$ can be presented as follows.
\begin{align*}
\begin{aligned}
\begin{tikzpicture} [bend angle=15, looseness=1]
\node (C0) at (0,0)  {$0$};
\node (C1) at (2,0)  {$6$};
\draw [->,bend left] (C0) to node[above]  {\scriptsize{$a$}} (C1);
\draw [->,bend left] (C1) to node[below]  {\scriptsize{$a^*$}} (C0);
\draw [->, looseness=24, in=52, out=128,loop] (C1) to node[above] {$\scriptstyle{\beta}$} (C1);
\draw [->, looseness=24, in=-128, out=-52,loop] (C1) to node[below] {$\scriptstyle{\gamma}$} (C1);
\end{tikzpicture}
\end{aligned}
& \quad \quad
\begin{aligned}
&(\beta +\gamma) a^* =0, \\
&a( \beta+\gamma) =0, \\
&a^*a=(\beta+\gamma)^2- \beta^3, \\
&a^*a=(\beta+\gamma)^2 - \gamma^3. \\
\end{aligned}
\end{align*}
These relations are defined by the superpotential
\[
\Phi :=a \beta a^* + a \gamma a^* - \beta^4 - \gamma^4 - (-\beta-\gamma)^3.
\]

The contraction algebra associated to this flop can presented as
\[
\frac{\mathbb{C}\langle \langle \beta, \gamma  \rangle \rangle}{\langle (\beta+\gamma)^2-\beta^3, (\beta+\gamma)^3-\gamma^2 \rangle  \\}.
\]
It can be calculated that (the complete version of the contraction algebra) has
\[
\dim A_{\text{con}}=27 \text{ and } \dim A_{\text{con}}/[A_{\text{con}},A_{\text{con}}] = 6
\] so the Gopakumar-Vafa invariants $n_1,n_2,n_3$ of this can satisfy  $n_1=6$ and $6+4n_2+9n_3=27 $ by Proposition \ref{Prop:GKV}, and hence $n_1=6, n_2=3,\text{ and } n_3=1.$ In particular, the example has length 3.

\end{example}

\begin{example} \label{Example:Length 3,4, and 6} We now give similar examples  produced from universal flops of length $4$ and $6$ which can be reduced to paths algebras over quivers with relations defined by a superpotential. These are CY3 noncommutative crepant resolutions. For the classifying maps $\mathbb{H}_l \rightarrow \mathbb{C}[T]$ defined by $t=0$ and, for $x \in \{ \beta, \gamma, \delta\}$, $T^{x}_i=0$ for $i>0$ and $T^{x}_0=T$ otherwise. The corresponding base changes of $\mathcal{A}_l$ can can be presented as the following paths algebras over $\mathbb{C}[T]$ with relations
\begin{align*}
\begin{aligned}
\begin{tikzpicture} [bend angle=15, looseness=1]
\node (C0) at (0,0)  {$0$};
\node (C1) at (2,0)  {$1$};
\draw [->,bend left] (C0) to node[above]  {\scriptsize{$a$}} (C1);
\draw [->,bend left] (C1) to node[below]  {\scriptsize{$a^*$}} (C0);
\draw [->, looseness=24, in=52, out=128,loop] (C1) to node[above] {$\scriptstyle{\beta}$} (C1);
\draw [->, looseness=24, in=-38, out=38,loop] (C1) to node[right] {$\scriptstyle{\delta}$} (C1);
\draw [->, looseness=24, in=-128, out=-52,loop] (C1) to node[below] {$\scriptstyle{\gamma}$} (C1);
\end{tikzpicture}
\end{aligned}
\qquad
\begin{aligned}
&aa^*=T e_0, & &  a^*a=\delta^k-T e_1,  & & \\ 
& \delta a^*=0, & & a \delta = 0, \\ 
&\beta^i= Te_1,  & & \gamma^j = T e_1\\ 
&\beta+\gamma + \delta = 0. \\ 
\end{aligned}
\end{align*}
where $(i,j,k)$ equals $(2,4,3)$, and $(2,3,5)$ for lengths 4 and 6 respectively.

After eliminating $Te_0, T e_1$ and $\delta$ the algebra can be presented as the following paths algebra over $\mathbb{C}$ of the following quiver with relations.
\begin{align*}
\begin{aligned}
\begin{tikzpicture} [bend angle=15, looseness=1]
\node (C0) at (0,0)  {$0$};
\node (C1) at (2,0)  {$1$};
\draw [->,bend left] (C0) to node[above]  {\scriptsize{$a$}} (C1);
\draw [->,bend left] (C1) to node[below]  {\scriptsize{$a^*$}} (C0);
\draw [->, looseness=24, in=52, out=128,loop] (C1) to node[above] {$\scriptstyle{\beta}$} (C1);
\draw [->, looseness=24, in=-128, out=-52,loop] (C1) to node[below] {$\scriptstyle{\gamma}$} (C1);
\end{tikzpicture}
\end{aligned}
\qquad
\begin{aligned}
& (\beta+ \gamma) a^*=0, \\
& a (\beta+ \gamma)  = 0, \\ 
&\beta^i= -a^*a+(-\beta- \gamma)^k \\ 
& \gamma^j =-a^*a+(-\beta- \gamma)^k
\end{aligned}
\end{align*}
Moreover, the relations can be encoded in the superpotential
\[
\Phi:= a \beta a^* + a \gamma a^* - \beta^{i+1} - \gamma^{j+1} - (-\beta-\gamma)^{k+1}.
\]

We note that in each of these examples there is an isolated singularity at the origin, however there are also other singularities away from the origin. Completing at the origin removes these extra singularities.
\end{example}

\begin{example} \label{Example: Length 5 explicit}
For the classifying map 
\[
\mathbb{H}_5 \rightarrow H\cong \mathbb{C}[T]:= \mathbb{H}_5[T]/(t,T^{\delta}_2, T^{\delta}_1, T_3, T_2, T-T_0^{\delta}, T-T_1, T-T_0).
\]
the corresponding base change of $\mathcal{A}_5$ produces an algebra as the path algebra over $\mathbb{C}[T]$ of the following quiver with relations.
\begin{align*}
\begin{aligned}
\begin{tikzpicture} [bend angle=15, looseness=1]
\node (C0) at (0,0)  {$0$};
\node (C1) at (2,0)  {$4$};
\draw [->,bend left] (C0) to node[above]  {\scriptsize{$a$}} (C1);
\draw [->,bend left] (C1) to node[below]  {\scriptsize{$a^*$}} (C0);
\draw [->, looseness=24, in=52, out=128,loop] (C1) to node[above] {$\scriptstyle{\beta}$} (C1);
\draw [->, looseness=24, in=-38, out=38,loop] (C1) to node[right] {$\scriptstyle{\delta}$} (C1);
\draw [->, looseness=24, in=-128, out=-52,loop] (C1) to node[below] {$\scriptstyle{\gamma}$} (C1);
\end{tikzpicture}
\end{aligned}
& \quad \quad
\begin{aligned}
&a\delta=0, \qquad \delta a^* = 0, \\
&aa^*= -T e_0, \\
&a^*a=\delta^4  -T e_4, \\
&\beta-\delta=0, \\
&\gamma \beta \gamma +\gamma^2 \beta +\gamma \beta^3=T e_4 \\
&
(\gamma+ \beta^2)^2+\beta \gamma \beta =0.
\end{aligned}
\end{align*}
Subject to the other relations, the following three relations are equivalent:
\begin{align*}
&\gamma \beta \gamma + \gamma^2 \beta + \gamma \beta^3 =T e_4, \\
&\gamma \beta \gamma - \beta^2 \gamma \beta -\beta \gamma \beta^2 - \beta^5= T e_4, \text{ and } \\
&\gamma \beta \gamma + \beta \gamma^2  +  \beta^3 \gamma =T e_4.
\end{align*}
From these relations it's easy to deduce that $T$ commutes with all other generators. Then, after eliminating $Te_0, Te_4$, and $\delta$ and simplifying the relations, this algebra can be rewritten as a path algebra over $\mathbb{C}$ of the following quiver with relations.
\begin{align*}
\begin{aligned}
\begin{tikzpicture} [bend angle=15, looseness=1]
\node (C0) at (0,0)  {$0$};
\node (C1) at (2,0)  {$4$};
\draw [->,bend left] (C0) to node[above]  {\scriptsize{$a$}} (C1);
\draw [->,bend left] (C1) to node[below]  {\scriptsize{$a^*$}} (C0);
\draw [->, looseness=24, in=52, out=128,loop] (C1) to node[above] {$\scriptstyle{\beta}$} (C1);
\draw [->, looseness=24, in=-128, out=-52,loop] (C1) to node[below] {$\scriptstyle{\gamma}$} (C1);
\end{tikzpicture}
\end{aligned}
\quad
\begin{aligned}
&a \beta =0, \qquad \beta a^* = 0, \\
& a^*a +   \gamma \beta \gamma +\gamma^2 \beta+ \beta \gamma^2 +\gamma \beta^3 +\beta^3 \gamma + \beta^2 \gamma \beta + \beta \gamma \beta^2+\beta^5=\beta^4 \\
&
\gamma^2+ \gamma \beta^2+\beta^2 \gamma +\beta \gamma \beta +\beta^4=0.
\end{aligned}
\end{align*}
These relations are determined by the superpotential 
\[
a \beta a^* - \frac{1}{5}\beta^5+ \beta^2 \gamma^2+ \frac{1}{2}\beta \gamma \beta \gamma  + \beta^4 \gamma.
\]
Again, this has an isolated singularity at the origin and other isolated singularities away from the origin.
\end{example}

\begin{appendix} \label{Appendix}
\section{Recovering hypersurfaces and matrix factorisations}

 In this appendix we give a brief outline of how a computer algebra package can be used to explicitly recover the hypersurface equation and matrix factorization description of a universal flop from the corresponding universal flopping algebra $\mathcal{A}_l$. Whilst these computations are not central to the results of this paper, they do allow the explicit calculation of the matrix factorisations conjectured by Curto and Morrison and, with some minimal adjustment, can be used to help calculate explicit examples of threefold flops.

We will display examples of code in the computation algebra system Magma \cite{Magma}. Whilst we use Magma, similar calculations could be carried out in any symbolic algebra package that is able to compute noncommutative Gr\"{o}bner basis.

In each step of the appendix, after discussing the general case, as a running example we present the code required to recover the description of the universal flop of length 2 as a matrix factorisation in the language of Curto \& Morrison as calculated in Example \ref{Example: Length 2}.

\subsection{Inputting a universal flopping algebra}
The following code inputs the algebra $\mathcal{A}_2$ we use for our running example as the quotient of a freely generated algebra in Magma.

\vspace{0.3cm}
\begin{adjustwidth}{2.5em}{0pt}
\begin{verbatim}
K := RationalField(); 
Kt<ta,tb,tc,td,z,y>:=RationalFunctionField(K,6);
F<a,A,d,c,b> := FreeAlgebra(Kt, 5);
B := [ z-a*b*A, y-a*c*A, b*b-tb, c*c-tc, d*d-td,
A*a+b+c+d-ta, a*A-2*ta  ];
G:=GroebnerBasis(B,6); 
\end{verbatim}
\end{adjustwidth}
\vspace{0.3cm}

Alternatively, the following code inputs the algebra after the change of basis (\ref{Equation:Change of Basis}) used in Example \ref{Example: Length 2}.

\vspace{0.3cm}
\begin{adjustwidth}{2.5em}{0pt}
\begin{verbatim}
K := RationalField(); 
Kt<t,u,w,v,z,y>:=RationalFunctionField(K,6);
F<a,A,d,c,b> := FreeAlgebra(Kt, 5);
B := [ z-a*b*A, y-a*c*A, b*b-u, c*c-w, d*d-(-2*v+z+y+u+w+(1/4)*t^2),
A*a+b+c+d-(1/2)*t, a*A-t  ];
G:=GroebnerBasis(B,6); 
\end{verbatim}
\end{adjustwidth}
\vspace{0.3cm}

For any length, the algebra $\mathcal{A}_l$ can be entered into Magma via similar code using the explicit descriptions calculated in this paper. 

\begin{remark}
In these examples, when encoding the algebra we do not enter the idempotent elements and define certain generators to be in a fraction field. This does not affect our later calculations and lowers the required computational power.

A partial noncommutative Gr\"{o}bner basis is calculated where the second integer coefficients determines how many iterations are computed. It is a matter of trial and error to select a large enough value for a particular example.
\end{remark}

\subsection{Calculation of $\mathcal{R}_l$.} \label{Appendix:Algebra}
We now explain how to explicitly recover the commutative algebra $R_{l}:= e_0 \mathcal{A}_l e_0$ from the explicit presentation of $\mathcal{A}_l$.

The algebra $\mathcal{R}_l$ is a graded deformation of the Kleinian singularity  $R_{\Gamma}$ over the base $\mathbb{H}_l$. As such it is a hypersurface, and, as this deformation is $\mathbb{Z}$-graded, the Hilbert polynomial of $\mathcal{R}_l$ equals the Hilbert polynomial of $\mathbb{H}_l \otimes_{\mathbb{C}} R_{\Gamma}$. From this data we deduce that $\mathcal{R}_l$ can be explicitly presented as a homogeneous hypersurface in $\mathbb{H}_l[y,z,x']$ for a choice of $y, z, x'$ that generate $e_0 \mathcal{A}_l e_0$ as a $\mathbb{H}_l$-module. 

Over the central fibre at the origin of $\mathbb{H}_l$ the algebra $\mathcal{R}_l$ is the A, D, E Kleinian singularity $R_\Gamma$. This is generated by the restriction of three elements $y, \, z$ and $x' \in e_0\mathcal{A}_l e_0$ to the central fibre. We list a choice of such generators  in $\mathcal{A}_l$ for each length $l$ along with the hypersurface equation they satisfy to define $R_{\Gamma}$ at the central fibre. 

\begin{center}
\begin{tabular}{ c |c|c |c | c } 
 Length & $x'$  & $y$ & $z$ & hypersurface \\ 
 \hline
  1 & $a_0a_1$  & $a_1^* a_0^*$ & $a_0a_0^*$ & $xy-z^{2}$ \\
  2 & $a \beta \gamma a^*$  & $a \beta a^*$ & $a \gamma a^*$ & $x^2-zy^2-z^{2}y$ \\
  3 & $a \gamma^2 \beta \gamma a^*$  & $a \gamma^2 a^*$ & $a \gamma a^*$ & $x^2-z^2x +y^3$ \\
  4 & $a \gamma^3 \beta \gamma^2 a^*$  & $a \gamma^3 a^*$ & $a \gamma a^*$ & $x^2-y^3+yz^3$ \\
  5 & $a \gamma^3 \beta \gamma^2 a^*$  & $a \gamma^3 a^*$ & $a \gamma a^*$ & $x^2-y^3-z^5$ \\
  6 & $a \gamma^2 \beta \gamma^2 \beta \gamma \beta \gamma^2 a^*$  & $a \gamma^2 \beta \gamma^2 a^*$ & $a \gamma a^*$ & $x^2-y^3+z^5$ \\
\end{tabular}
\end{center}

As $\mathcal{R}_l$ is a graded deformation of $\mathcal{R}_{\Gamma}$, the polynomial defining the hypersurface $\mathcal{R}_l$ is quadratic in $x'$ and can be presented in the form $x'^2 + x'P+Q$ for $P,Q \in \mathbb{H}_l[y,z]$. Such a presentation can be explicitly found by calculating $x'^2$ in the appropriate noncommutative Grobner basis within the algebra $\mathcal{A}_l$. After the base change $x:=x'-P/2$  such an equation can be expressed as $x^2=P^2/4+Q$.

In our continuing example running the following code, beneath the code entering the algebra above, recovers the hypersurface equation. 
\vspace{0.3cm}
\begin{adjustwidth}{2.5em}{0pt}
\begin{verbatim}
xdash:=(a*b*c*A);
P:=MonomialCoefficient(NormalForm(xdash*xdash,G),xdash);
x:=xdash-(1/2)*P;
g:=NormalForm(x*x,G);
printf"Hypersurface equation is -x^2+";
NormalForm(x*x,G); printf"\n \n";
\end{verbatim}
\end{adjustwidth}
\vspace{0.3cm}
This computation can be straightforwardly generalised for lengths $\ge 2$ by changing the algebra input and the description of the element \texttt{xdash}. 

\subsection{Calculation of a matrix factorisation} \label{Appendix:MF}
We now explain how the matrix factorisation description of the universal flop can be recovered from the universal flopping algebra.

To calculate a matrix factorisation we identify a set of generators for $\mathcal{N}_l$ as an $\mathcal{R}_l$-module, which we can find using the presentation of $\mathcal{A}_l$ and definition of $\mathcal{N}_l:=e_0 \mathcal{A}_l(1-e_0)$ as a module of paths in $\mathcal{A}_l$. We list a choice of $2l$ generators $\{ g_i \}$ for each $l$ below.

\begin{center}
\begin{tabular}{ c |c  } 
 Length & Generators    \\
 \hline
  1 &  $a_0$ and $a_1^*$.  \\
    2 & $a, \,  a\beta,\,  a \gamma,$ and $ a \beta \gamma$ \\
     3 &  $a, \, a \gamma, \,  a \gamma^2, \, a \gamma \beta, \,  a \gamma^2 \beta,$ and $a \gamma^2 \beta \gamma$\\
        4 & $a, \,  a \gamma, \, a \gamma^2, \, a \gamma^3, \, a \gamma^2 \beta, \, a \gamma^3 \beta, \, a \gamma^3 \beta \gamma ,$ and $ a \gamma^3 \beta \gamma^2 $. \\
     5 &$a, 
a \gamma, \,
a \gamma^2 ,\,
a \gamma \beta, \,
a  \gamma^3, \,
a \gamma^2 \beta, \,
a \gamma^3 \beta , \, 
$
\\
& 
$
a \gamma^3 \beta \gamma, \,
a \gamma^3 \beta^2,
$ and $
a \gamma^3 \beta \gamma^2$ \\
        6 &  $a, \,
a \gamma, \,
a \gamma^2, \,
a \gamma^2 \beta,\,
a \gamma^2 \beta \gamma, \,
a \gamma^2 \beta \gamma^2, \,
a \gamma^2 \beta \gamma \beta,  \,
a \gamma^2 \beta \gamma^2 \beta,\,$ \\
 & 
$a \gamma^2 \beta \gamma^2 \beta \gamma, \,
a \gamma^2 \beta \gamma^2 \beta \gamma \beta, \,
a \gamma^2 \beta \gamma^2 \beta \gamma \beta \gamma,$ and $
a \gamma^2 \beta \gamma^2 \beta \gamma \beta \gamma^2$
\end{tabular}
\end{center}

We note that after restricting to the central fibre these are generators of the corresponding module for the Kleinian singularity over the central fibre at the origin of $\mathbb{H}_l$.

Recalling the description of $\mathcal{R}_l$ above, the action of $x$ from the left on each of the generators ${g_i}$ can be calculated and expressed in terms of this generating set. This produces a collection of polynomials $C_{i,j} \in \mathbb{H}_l[y,z]$ such that relations of the form $x.g_i= \sum C_{i,j} g_j$ hold. This defines matrices $C=(C_{i,j})$ and 
$
M:=C- x I_{2l}
$
such that the $\mathcal{R}_l$-module $\mathcal{N}_l$ is the cokernel of $M$:
\[
\mathcal{R}_l^{\oplus 2l} \xrightarrow{M} \mathcal{R}_l^{\oplus 2l} \rightarrow \mathcal{N}_l \rightarrow 0.
\]

Then $(C - x I_{2l})(C + x I_{2l})=f_l I_{2l}$ is a matrix factorisation of the equation $f_l$ defining the corresponding hypersurface. This verifies that the presentation of the matrix factorisation suggested in Conjecture \ref{conj:CurtoMorrison2} exists.

Having input the algebra $\mathcal{A}_l$, the following code finds the matrix  $C$ in our running length 2 example.
\vspace{0.3cm}
\begin{adjustwidth}{2.5em}{0pt}
\begin{verbatim}
L:=[a, a*b, a*c, a*b*c];

printf"One half of MF is x Id + C where C:= \n {\n ";
for X in L do
printf"{ ";
for Y in L do 
g:=NormalForm(x*X,G); 
printf"%o", MonomialCoefficient(g,Y);
if Y ne L[#L] then printf", ";
end if;
end for; printf"}";
if X ne L[#L] then printf", ";
end if; printf" \n ";
end for; printf"} \n \n";
\end{verbatim}
\end{adjustwidth}
\vspace{0.3cm}
The outcome of this calculation is used in Example \ref{Example: Length 2}. 

This code can be easily amended to work for other lengths by changing \texttt{L} to the list of generators given above for the corresponding length.

\begin{example}  \label{Example:Length 3}
We consider flop of length 3 in example \ref{example:length 3 example}, where the hypersurface equation and matrix factorisation can be calculated by the following Magma code.
\vspace{0.3cm}
\begin{adjustwidth}{2.5em}{0pt}
\begin{verbatim}
K := RationalField(); 
Kt<T,z,y>:=RationalFunctionField(K,3);
F<a,A,d,b,c> := FreeAlgebra(Kt, 5);
B := [
z-a*c*A, y-a*c^2*A,b^3-T*b-T,c^3-T*c-T,
b+c+d, a*A+T, A*a-d*d+T, a*d, d*A
 ];
G:=GroebnerBasis(B,6);

xdash:=a*c^2*b*c*A;
P:=MonomialCoefficient(NormalForm(xdash*xdash,G),xdash);
x:=xdash-(1/2)*P;
g:=NormalForm(x*x,G);
printf"Hypersurface equation is -x^2+";
NormalForm(x*x,G); printf"\n \n";

L:=[a, a*c, a*c*c, a*c*b, a*c*c*b, a*c*c*b*c];

printf"One half of MF is x Id + C where C:= \n {\n ";
for X in L do
printf"{ ";
for Y in L do 
g:=NormalForm(x*X,G); 
printf"%o", MonomialCoefficient(g,Y);
if Y ne L[#L] then printf", ";
end if;
end for; printf"}";
if X ne L[#L] then printf", ";
end if; printf" \n ";
end for; printf"} \n \n";
\end{verbatim}
\end{adjustwidth}
\end{example}

\begin{remark}
It is more computationally efficient to recursively calculate the normal form of each $x \cdot g_i$ from the normal form of a previous $x \cdot g_j$ rather than calculate each from scratch. This can become more necessary for larger lengths. In addition, the polynomials involved can become very long and it can also become necessary to extract each individual coefficient $C_{i,j}$ individually.
\end{remark}

\subsection{Calculating the dimension of contraction algebras} \label{App:Dimension of contraction algebras}
For completeness we include Magma code that can calculate the dimension of contraction algebras. We also refer the reader to \cite[Appendix]{GKVWemyssBrown}, where examples of code for understanding flops and contraction algebras are given in Singular \cite{Singular} and Magma.

The following code inputs the contraction algebra considered in Example \ref{Ex:Laufer flop}, calculates its dimension, determines its indecomposable projectives with associated radicals, and displays the maps between the projectives in a grid.

\vspace{0.3cm}
\begin{adjustwidth}{2.5em}{0pt}
\begin{verbatim}
K := RationalField();
F<b,c> := FreeAlgebra(K, 2);
R := [c^3-b^2,b*c+c*b];
A<b,c>:=quo <F|R>;

printf "Algebra has dimension = %o \n \n", Dimension(A);

B, g := MatrixAlgebra(A);
M := RModule(B);
I := IndecomposableSummands(M);

for P in I do
d:=Dimension(P);
printf "Projective has dimension = %o \n", d;
J, h := JacobsonRadical(P);
printf "Radical has dimensions = %o \n", Dimension(J);
print ""; 
end for;

for P in I do
for Q in I do
printf "%o ", Dimension(AHom(P, Q));
end for; print "";
end for;
\end{verbatim}
\end{adjustwidth}
\vspace{0.3cm}
In particular, this example has dimension 9 and 1 indecomposable projective. The start of this code can be easily altered to make the same calculation for other contraction algebras.

\begin{remark}
This code calculates over $\mathbb{Q}$ rather than $\mathbb{C}$. This does not alter the dimension.
\end{remark}
\end{appendix}

\end{document}